\documentclass[11pt]{amsart}
\usepackage{amsmath,amsfonts,amsthm,amssymb,amscd,verbatim,graphicx,tikz}
\usepackage{color}
\usepackage{xcolor}
\usepackage{hyperref}
\usepackage{multicol}

\def\classification#1{\def\@class{#1}}
\classification{\null}

\newcommand{\Fq}{\mathbb{F}_q}

\newcommand{\Aut}{\mathrm{Aut}}
\newcommand{\PSL}{{\mathrm{PSL}}}

\newcommand{\PGL}{{\mathrm{PGL}}}
\newcommand{\PGammaL}{{\mathrm{P\Gamma L}}}
\newcommand{\GL}{{\mathrm{GL}}}
\newcommand{\veps}{{\varepsilon}}
\newcommand{\ord}{{\mathrm{ord}}}
\newcommand{\lcm}{{\mathrm{lcm}}}
\newcommand{\wt}{\widetilde}
\newcommand{\ovl}{\overline}
\newcommand{\hp}{$\phantom{|^|}$}

\usepackage[margin=1in]{geometry}

\DeclareFontFamily{OT1}{rsfs}{}
\DeclareFontShape{OT1}{rsfs}{n}{it}{<-> rsfs10}{}
\DeclareMathAlphabet{\mathscr}{OT1}{rsfs}{n}{it}

\newtheorem{prop}{Proposition}[section]
\newtheorem{thm}{Theorem}

\newtheorem{cor}[prop]{Corollary}
\newtheorem{lem}[prop]{Lemma}

\numberwithin{equation}{section}
\title{Non-orientable regular maps with negative prime-power Euler characteristic}

\author{Marston Conder}
\address{Department of Mathematics, University of Auckland, Private Bag 92019, Auckland\, New Zealand}
\email{m.conder@auckland.ac.nz}

\author{Nick Gill}
\address{School of Mathematics and Statistics, The Open University, Walton Hall, Milton Keynes, MK7 6AA, UK}
\email{nick.gill@open.ac.uk}

\author{Jozef {{\v{S}}ir{\'a}{\v{n}}}}
\address{Department of Mathematics, FCI, Slovak University of Technology, Bratislava, Slovakia,
and School of Mathematics and Statistics, The Open University, Milton Keynes, United Kingdom}
\email{jozef.siran@stuba.sk}

\begin{document}

\maketitle

\tableofcontents

\begin{abstract}
In this paper we provide a classification of all regular maps on surfaces of Euler characteristic $-r^d$ for some odd prime $r$ and integer $d\ge 1$.  Such maps are necessarily non-orientable, and the cases where $d = 1$ or $2$ have been dealt with previously.

This classification splits naturally into three parts, based on the nature of the
automorphism group $G$ of the map, and particularly the structure of its quotient $G/O(G)$ 
where $O(G)$ is the largest  normal subgroup of $G$ of odd order.  In fact $G/O(G)$ is isomorphic 
to either a $2$-group (in which case $G$ is soluble), or $\PSL(2,q)$ or $\PGL(2,q)$ where $q$ is 
an odd prime power.
The result is a collection of $18$ non-empty families of regular maps, with conditions on the associated parameters.
\end{abstract}


\section{Introduction}\label{sec:intro}
\medskip


A {\em map} is a cellular embedding of a connected graph on a 2-dimensional surface. 
A map is said to be {\em uniform} and of {\em type} $\{m,n\}$ if all its faces are bounded by 
closed walks of length $m$ and all its vertices have valency $n$. 
Except for degenerate cases not considered here, an {\em automorphism} of a map can be identified with a permutation of its mutually incident face-vertex-edge triples (flags) which preserves the underlying graph and all faces of the map. All such automorphisms form the {\em automorphism group} of a map under composition of permutations. As only the trivial automorphism fixes a given flag, the automorphism group of a map acts as 
a semi-regular permutation group on its flag set. 
Accordingly, the largest `level of symmetry' of a map arises when its automorphism group acts regularly on flags, and in that case, the map is called {\em regular}, and is necessarily uniform. 
Such maps with $m \le  2$ or $n \le 2$ are either degenerate or maps on the sphere or the real projective plane, 
and  henceforth we assume that $m,n\ge 3$. In what follows we sum up a few relevant facts concerning regular maps, referring to \cite{siran} for details.

The regularity of the automorphism group $G=\Aut(M)$ of a regular map $M$ of type $\{m,n\}$ on flags enables one to identify $M$ with a presentation of $G$ in terms of three particular generating involutions $a$, $b$ and $c$, in the form
\begin{equation}\label{eq:abc} G=\langle\, a,b,c\ |\ a^2,b^2,c^2,(ab)^m,(bc)^n,(ac)^2, \ldots \,\rangle,  \end{equation}
where the dots indicate the possibility of additional relators.  Without additional relators, the above presentation is one for the {\em full $(2,m,n)$-triangle group} $\Delta(2,m,n)$, and its index $2$ subgroup generated by 
$x = ab$ and $y = bc$ is the {\em ordinary $(2,m,n)$-triangle group} $\Delta^+(2,m,n)$, with definiing 
relations $x ^ m = y^n = (xy)^2 = 1$.

The  involutions $a$, $b$, $c$ generating $G$ are automorphisms taking a fixed flag $F=(f,v,e)$ formed by a mutually incident face-vertex-edge triple to the three neighbouring flags of $F$, so that the compositions $ab$, $bc$ and $ac$ act locally as an $m$-fold rotation of the map about a centre of the face $f$, an $n$-fold rotation about the vertex $v$, and a two-fold rotation about a centre of the edge $e$. Identification of $M$ with a presentation of $G$ of the form \eqref{eq:abc} will be reflected in the notation by writing $M = (G; a,b,c)$.

The group $G$ is finite if and only if the carrier surface of $M$ is compact, in which case the Euler characteristic $\chi=\chi_G$ of the surface is given by 
\begin{equation}\label{e: sunny}
-\chi_G = \frac{|G|}{2}\left(\frac12-\frac{1}{m}-\frac{1}{n}\right) = \frac{|G|}{4mn}(mn-2m-2n)\ .
\end{equation}

If the carrier surface of a regular map $M=(G;a,b,c)$ is orientable, then $a$, $b$ and $c$ are orientation-reversing automorphisms while $x=ab$ and $y=bc$ are orientation-preserving, so that $\langle ab,bc\rangle$ is a subgroup of $G$ of index $2$; the converse holds as well. It follows that the carrier surface of $M$ is non-orientable if and only if $G=\langle a,b,c\rangle = \langle x,y \rangle$.

 The facts described in this Introduction are well known, and can be tracked down in the fundamental papers \cite{JoSi,BrSi}, which developed the algebraic theory of maps and regular maps. A summary can also be found in the survey paper \cite{siran}, and in an extended form in the monograph \cite{CJST}, together with numerous connections of the theory of maps to group theory, hyperbolic geometry, Riemann surface theory, and Galois theory.

Interest in classifying all regular maps on a given compact surface goes back to as early as the beginning of the 20th century. Nevertheless, it appears that classification results for {\em infinite} families of surfaces are scarce and restricted to those of Euler characteristic $-r$ \cite{bns}, $-2r$ in the orientable case \cite{CST} (see also \cite{BeJo}), $-3r$ \cite{cns} and $r^2$ \cite{cps} for an odd prime $r$.

The first and the last of these results provide motivation to consider regular maps on surfaces of Euler characteristic $-r^d$ for some odd prime $r$ and integer $d\ge 3$; note that such maps are necessarily non-orientable. This is the task we set ourselves in this paper.

It turns out that this task splits naturally into three parts. To see this, suppose that $G=\langle a,b,c\rangle$ is the finite automorphism group of a regular map on a surface of odd Euler characteristic, and let $O=O(G)$ for the largest odd-order normal subgroup of $G$. It is easy to show that $G$ must have dihedral Sylow 2-subgroups \cite[Lemma 3.2]{cps}. This allows us to invoke the Gorenstein--Walter theorem \cite{gor2, gor3},  which states that if $G$ is an arbitrary finite group with dihedral Sylow 2-subgroups, then the quotient $G/O$ is isomorphic to either a Sylow 2-subgroup, or to $A_7$, or to a subgroup of ${\rm P\Gamma L}(2,q)$ containing $\PSL(2,q)$ for a prime-power $q\ge 5$.

We find (by Proposition \ref{p: odd order} below) that our supposition rules out many of the possibilities given by Gorenstein and Walter, and indeed $G/O$ is isomorphic either to a Sylow 2-subgroup of $G$ (so that $G$ has a 2-complement and hence is soluble), or to $\PSL(2,q)$, or to $\PGL(2,q)$, for some $q$.

This suggests approaching the problem of classification of finite regular maps $M=(G;a,b,c)$ on surfaces of Euler characteristic $-r^d$ for odd primes $r$ by separately considering the following three cases:\ , (A) $G/O\cong \PSL(2,q)$ where $q\geq 5$;\, (B) $G/O\cong \PGL(2,q)$ where $q\geq 5$; and\, (C) $G$ is soluble (in which case $G/O$ is isomorphic either to $\PGL_2(3)=S_4$ or to a 2-group, by our Theorem~\ref{t: odd order}).

To state our results we need some more terminology and notation. For a non-orientable regular map $M=(G;a,b,c)$ of type $\{m,n\}$, with $G$ as in \eqref{eq:abc} such that $G=\langle x,y\rangle$ for $x=ab$ and $y=bc$ (having orders $m$ and $n$), the group $G$ will be referred to as a $(2,m,n)^*$-group, and $\chi_G$ will be its Euler characteristic. We also say that a group is {\em almost Sylow-cyclic} (aSc) if its Sylow subgroups of odd order are cyclic and its Sylow $2$-subgroups are trivial or contain a cyclic subgroup of index $2$. The main result follows.
\medskip

\begin{thm}\label{t: main}
Let $G$ be a finite $(2,m,n)^*$-group defining a non-orientable regular map of Euler characteristic $\chi_G=-r^d$ for some odd prime $r$ and $d\ge 1$, and let $O = O(G)$ be the largest odd-order normal subgroup of $G$. Then $G/O\cong \PSL(2,q)$ or $\PGL(2,q)$ for some odd prime-power $q \ge 5$, or $G$ is soluble,
and in fact the quadruple $(m,n,r,d)$ lies in one of $18$ non-empty families, as follows:
\begin{enumerate}
\item[{\rm (A)}] Families {\rm A1} to {\rm A4} for $G/O\cong \PSL_2(q)$, as in {\rm Table \ref{t: a}};
\item[{\rm (B)}] Families {\rm B1} to {\rm B7} for $G/O\cong \PGL_2(q)$, and $N$ a normal $r$-subgroup of $O$, as  in {\rm Table \ref{t: b}};
\item[{\rm (C)}] Families {\rm C1} to {\rm C7} for soluble $G$ and a normal $r$-subgroup $N$ of $G$, as in {\rm Table \ref{t: c}}.
\end{enumerate}
\end{thm}
\vspace{2mm}

\begin{center}
 \begin{table}[ht]
  \begin{tabular}{|c|c|c|c|c|c|}
   \hline
   & $q$ & $\{m,n\}$ & $r$ & $d\geq $ & $-\chi_G \ (\,= r^d)$ \\ \hline
   A1 & 5 & $\{5,5\}$ & 3 & 1 & \hp$3\,|O|$\hp \\
   A2 & 5 & $\{3,15\}$& 3 & 1 & $3\,|O|$ \\
   A3 & 13 & $\{3,13\}$ & 7 & 2 & $7^2\,|O|$ \\
   A4 & 13 & $\{3,7\}$& 13 & 1 & $13\,|O|$ \\ \hline
  \end{tabular}
\vspace{2mm}
\caption{Families in which $G$ is a $(2,m,n)^*$-group with 
$G/O=\PSL_2(q)$}\label{t: a}
 \end{table}
\end{center}
\vspace{-12mm}

\begin{center}
 \begin{table}[ht]
  \begin{tabular}{|c|c|c|c|c|c|c|}
   \hline
   & $q$ & $\{m,n\}$ & $r$ & $d\geq $ & $-\chi_G \ (\,= r^d)$  & Notes\\ \hline
   B1 & 5 & $\{4,6\}$ & 5 & 1 & $5\,|O|$ & \hp$\ell=1$ \\
   B2 & 5 & $\{20,30\}$ & 5 & 5 & $5^2\,|O|$ & \hp$\ell=1$ \\
   B3 & 7 & $\{3\ell r^s,8r^s\}$ & 7 & 1 & $7\left(\frac{|N|}{7^s}\right)(12{\cdot}7^s\ell-3\ell-8)$ & \\
   B4 & 9 & $\{5\ell r^s,8r^s\}$ & 3 & 1 & $9\left(\frac{|N|}{3^s}\right)(20{\cdot}3^s\ell-5\ell-8)$ & \\
   B5 & $p$ & $\{\ell pr^s, (p{+}1)r^s\}$ & $\frac{p{-}1}{2}{=}r^t$ & $r^t{+}3t{+}1$ & $\frac{(p{-}1)|N|}{2r^s}{\cdot} \left[r^s\ell \frac{p(p{+}1)}{2}{-}\ell p{-}p{-}1\right]$ & $p \ge 7, \, s\geq 1$ \\
   B6 & $p$ & $\{\ell p, (p{-}1)\}$ & $\frac{p{+}1}{2}{=}r^t$ & $t^2\log_2r{+}2t^*$ &
   $\frac{p{+}1}{2} {\cdot} \left[\ell \frac{p(p-1)}{2}{-}\ell p{-}p{+}1\right]$ & $|N| = 1$\\
   B7 & $p$ & $\{\ell pr^s, (p{-}1)r^s\}$ & $\frac{p{+}1}{2}{=}r^t$ & $r^t{+}3t{-}1^*$  & $\frac{(p{+}1)|N|}{2r^s}{\cdot} \left[r^s\ell \frac{p(p{-}1)}{2}{-}\ell p{-}p{+}1\right]$ & $|N| > 1$\\ \hline
   \end{tabular}
\vspace{2mm}
\caption{Families in which $G$ is a $(2,m,n)^*$-group with $G/O\cong\PGL_2(q)$}
\label{t: b}
 \end{table}
\end{center}
\vspace{-12mm}

\begin{center}
 \begin{table}[ht]
  \begin{tabular}{|c|c|c|c|c|c|c|}
   \hline
   & $G/N$ & $\{m,n\}$ & $r$ & $d\geq $ & $-\chi_G \ (\,= r^d)$ & Notes\\ \hline
   C1 & $D_{\ell}$, $\ell=3+3^i$ & $\{6,\ell\}$ & 3 & 1 & \hp$3^{i-1}|N|$ & $|N|\ge 9$ if $i = 0$ \\
   C2 & $D_{\ell}$, $\ell=1+3^i$ & $\{6,3\ell\}$ & 3 & 1 & $3^{i-1}|N|$ & \hp$|N|\ge 9$ \\
   C3 & $D_{j}{\times} D_{k}$ & $\{2j, 2k\}$ & $3$ mod $4$ & 1 & $|N|(jk-j-k)$ & $G$ aSc if $|N| = 1$ \\
   C4 & $D_{j}{\times} D_{k}$ & $\{2jr^\alpha, 2kr^\beta\}$ & $3$ mod $4$ & $5^*$ &
   $\frac{|N|}{r^\alpha}(jkr^\alpha{-}jr^{\alpha-\beta}{-}k)$ & $|N|\ge r^{\alpha{+}1}$ \\ 
   C5 & $(C_2\times C_2)\rtimes D_{\ell}$  & $\{4,\ell\}$ & $5$ mod $6$ & 1 & $|N|(\ell-4)$ & $G$ aSc if $|N| = 1$ \\
   C6 & $(C_2\times C_2)\rtimes D_{\ell}$ & $\{4{\cdot}3^\alpha,\ell{\cdot}3^\beta\}$ & $3$ & $3$ & $\frac{|N|}{3^\alpha}
   (2\ell{\cdot}3^\alpha{-}4{\cdot}3^{\alpha-\beta}{-}\ell)$ & $|N|\ge 3^{\alpha{+}1}$ \\
   C7 & $(C_2\times C_2)\rtimes D_{\ell}$  & $\phantom{\frac{^||}{|}}\{4r^\alpha,\ell r^\beta\}\phantom{\frac{|^|}{|}}$ & $5$ mod $6$ & $d_{r,\alpha,\beta}$ &
   $\frac{|N|}{r^\alpha}(2\ell r^\alpha{-}4r^{\alpha-\beta}{-}\ell)$
   &  $|N|\ge r^{\alpha{+}1}$ \\
 \hline
 \end{tabular}
\vspace{2mm}
\caption{Families in which $G$ is a soluble $(2,m,n)^*$-group}
\label{t: c}
 \end{table}
\end{center}

Some of the content of Tables~\ref{t: b} and~\ref{t: c} needs additional explanation.

For all families in Table~\ref{t: b} we have $K\cong \PSL_2(q)$, $\ell$ a non-negative integer and $G\cong N.(K{\times}C_\ell).2$, with $r\nmid\ell$ and $\gcd(\ell,|K|)=1$. For each family, $s$ is a non-negative integer, and $r^s$ divides $|N|$. In family B6, the entry `$t^2\log_2r{+}2t^*$' means that $d>t^2\log_2r{+}2t$ when $t\geq 2$; if $t=1$, then $d\geq 5$. In family B7, the entry `$r^t{+}3t{-}1^*$' means that $d>r^t{+}3t{-}1$ when $p>5$.

In Table~\ref{t: c}, we have non-negative integers $i,j,k,\ell, \alpha$ and $\beta$. For families C3 and C4, $j$ and $k$ are odd and coprime. Integers $\alpha$ and $\beta$ occur in families C4, C6 and C7 and we require that, in all three cases, $\alpha \geq 1$ and $\alpha \geq \beta$. Next, in family C4, the entry `$5^*$' means that $d \ge  5$ when $r > 3$. Finally, $d_{r,\alpha,\beta} = \alpha^2\log_2 r{-}\alpha{-}\beta{+}1$ in family C7.


Although Theorem~\ref{t: main} classifies the regular maps under consideration, some mysteries remain: For the families in Table~\ref{t: a} it does not completely specify $O(G)$, and for the families in Tables~\ref{t: b} and \ref{t: c} it does not completely specify the normal subgroup $N$. If we focus on the case $d\le 4$, however, we can give a complete list.

\begin{cor}\label{c: main}
Let $G=\langle a,b,c\rangle$ be a $(2,m,n)^*$-group with $\chi_G=-r^d$ for some odd prime $r$ and positive integer
$d\leq 4$. Then one of the following holds$\,:$
\begin{enumerate}
 \item $d\in\{1,3\}$, $r\equiv 3$ mod $4$, $G\cong D_{j}\times D_{k}$, $\{m,n\}=\{2j,2k\}$ and $r^d+1=(j-1)(k-1)$
for  odd coprime integers $j$ and $k\,;$
\item $d\in\{1,3\}$, $r\equiv 5$ mod $6$, $G\cong (C_2\times C_2)\rtimes D_{r^d+4}$ and $\{m,n\}=\{4, r^d+4\}\,;$
\item 
 the triple $(G,\{m,n\},\chi_G)$ appears in one of the rows of the following split table, in which ${\rm E}_{r^\alpha}$ is an elementary-abelian $r$-group of order $r^\alpha$ and ${\rm He}_3$ is the Heisenberg group of order $27$, the left-most column gives the relevant family from Theorem $\ref{t: main}$, and the right-most column gives the designation of the associated map in the list of \cite{C600} (if small enough)$:$
\end{enumerate}
{\small
\begin{table}[ht]
\centering
\begin{tabular}{|c|c|c|c|c|}
\hline
{\rm Case} & $G$ & $\{m,n\}$ & $-\chi_G$ & \!\cite{C600}\! \\ \hline
{\rm A1} & $\PSL_2(5)$ & $\{5,5\}$ & $3$ & {\rm N5.3} \\
{\rm A3} & $\PSL_2(13)$ & $\{3,13\}$ & $7^2$ & {\rm N51.1} \\
{\rm A4} & $\PSL_2(13)$ & $\{3,7\}$ & $13$ & {\rm N15.1} \\
{\rm A4} & $\!E_{13^3}{\cdot}\PSL_2(13)\!$ & $\{3,7\}$ & $13^4$ & -- \\ \hline
{\rm B6} & $\PGL_2(5)$  & $\{4,5\}$ & $3$ & {\rm N5.1} \\
{\rm B1} & $\PGL_2(5)$ & $\{4,6\}$ & $5$ & {\rm N7.1} \\
{\rm B3} & $\PGL_2(7)$ & $\{3,8\}$ & $7$ & \!{\rm N9.1,2}\! \\
{\rm B1} & ${\rm E}_{5^3}{\cdot}\PGL_2(5)$ & $\{4,6\}$ & $5^4$ &  -- \\
{\rm B3} & ${\rm E}_{7^3}{\cdot}\PGL_2(7)$ & $\{3,8\}$ & $7^4$ &  -- \\ \hline
{\rm C1} & ${\rm E}_{3^2}\rtimes D_4$ & $\{4,6\}$ & $3$ & {\rm N5.2} \\
{\rm C1,2,4 } & ${\rm E}_{3^2}\rtimes D_2$ & $\{6,6\}$ & $3$ & {\rm N5.4} \\ \hline
\end{tabular}
\quad
\begin{tabular}{|c|c|c|c|c|}
\hline
{\rm Case} & $G$ & $\{m,n\}$ & $-\chi_G$ & \cite{C600} \\ \hline
\ \ {\rm C1}\ \ & ${\rm He}_3\rtimes D_4$ & $\{4,6\}$ & $3^2$ & {\rm N11.1} \\
{\rm C1,2,4} & ${\rm He}_3\rtimes D_2$ & $\{6,6\}$ & $3^2$ & {\rm N11.2} \\
{\rm C6} & ${\rm E}_{3^3}{\cdot}(D_2{\rtimes}D_3)$ & $\{3,12\}$ & $3^3$ & {\rm N29.1} \\
{\rm C1,2,4} & $(C_3\wr C_3)\rtimes D_2$ & $\{6,6\}$ & $3^3$ & {\rm N29.2} \\
{\rm C1,2} & ${\rm E}_{3^3}\rtimes D_4$ & $\{6,12\}$ & $3^3$ & {\rm N29.3} \\
{\rm C2} & ${\rm He}_3\rtimes D_4$ & $\{6,12\}$ & $3^3$ & \!{\rm N29.4,5}\! \\
{\rm C1,2,4} & ${\rm E}_{3^2}\rtimes D_{10}$ & $\{6,30\}$ & $3^3$ & {\rm N29.6} \\
{\rm C1} & $({\rm E}_{3^2}{\cdot}{\rm He}_3)\rtimes D_4$ & $\{4,6\}$ & $3^4$ & {\rm N83.1} \\
{\rm C1,2,4} & $({\rm E}_{3^2}{\cdot}{\rm He}_3)\rtimes D_2$ & $\{6,6\}$ & $3^4$ & {\rm N83.2} \\
{\rm C1,2} & $(C_3{\times}{\rm He}_3)\rtimes D_4$ & $\{6,12\}$ & $3^4$ & {\rm N83.3} \\
{\rm C1,2,4} & ${\rm He}_3\rtimes D_{10}$ & $\{6,30\}$ & $3^4$ & {\rm N83.4} \\
\hline
\end{tabular}
\end{table}
}
\end{cor}

In particular, it follows that there are infinitely many regular maps of Euler characteristic $-r^d$ when $r$ is an odd prime if $d\in \{1,3\}$, but there are only a finite number of such maps if $d\in \{2,4\}$.

The rest of this paper is structured as follows. In Section \ref{s: chi odd} we present general structural theory for the automorphism group $G$ of a regular map on a (non-orientable) surface with odd Euler characteristic $\chi_G$. Then in Sections \ref{s: insol-PSL} and \ref{s: insol-PGL} we use this theory to consider the cases of such groups $G$ when $\chi_G=-r^d$ and $G/O(G)$ is isomorphic to $\PSL_2(q)$ or $\PGL_2(q)$ for some prime-power $q\ge 5$, followed by an analysis of soluble groups $G$ with $\chi_G=-r^d$ in Section \ref{s: soluble}. These three sections yield a proof of Theorem~\ref{t: main}. Finally, in Section \ref{s: constructions} we present constructions for automorphism groups $G$ of regular maps with $\chi_G=-r^d$, and give a proof of Corollary~\ref{c: main}.
\medskip


\section{\texorpdfstring{Structural theory for groups of odd characteristic}{chi is odd}}\label{s: chi odd}
\medskip


Throughout this section, let $G$ be a finite $(2,m,n)^*$-group with odd Euler characteristic $\chi_G$. By
\cite[Lemma 3.2]{cps}, oddness of $\chi_G$ implies that the Sylow $2$-subgroups of $G$ are dihedral. With help of a deep theorem by Gorenstein and Walter \cite{gor2, gor3} on groups with dihedral Sylow $2$-subgroups, we will derive structural information about $G$, in order to substantially extend Propositions 5.1 to 5.3 from \cite{bns}. As before, let $O=O(G)$ be the largest odd-order normal subgroup of $G$.
Also let $\ord(g)$ denote the order of an element $g$ in a group understood from the context, and let $|k|_p$ (or $k_p$ when there is no ambiguity) denote the largest power of the prime $p$ that divides the positive integer $k$.

\begin{prop}\label{p: odd order}
Let $G$ be a $(2,m,n)^*$-group with odd $\chi_{G}$. Then one of the following holds$\,:$
\begin{enumerate}
\item[{\rm (a)}] $G$ has a normal $2$-complement (so $G/O(G)$ is isomorphic to a Sylow $2$-subgroup $P$ of $G$);
\item[{\rm (b)}] $G/O(G)$ is isomorphic to $\PSL_2(q)$ or $\PGL_2(q)$ for some odd prime-power $q$.
\end{enumerate}
\end{prop}

\begin{proof}
Since $G$ has dihedral Sylow $2$-subgroups, we may apply the Gorenstein-Walter Theorem \cite{gor2, gor3}, which tells us that $H=G/O(G)$ is isomorphic to a Sylow $2$-subgroup of $G$, or to $A_7$, or to a subgroup of $\PGammaL_2(q)$ containing $\PSL_2(q)$ for some odd prime-power $q$.

If $H$ is isomorphic to a Sylow $2$-subgroup of $G$, then (a) holds and we are done. Next suppose that $H\cong A_7$. In this case, the lists at \cite{condersmall2,C600} show that $A_7$ is not a $(2,m',n')^*$-group for any positive integers $m',n'\le 7$  (the only possible orders of elements in $A_7$), and hence this possibility can be excluded.

So from now on we can assume that $\PSL_2(q)\le H\le \PGammaL_2(q)$ for some odd $q=p^e$. If $q$ is prime, then (b) must hold. Hence we may assume that $q\geq 9$ and, in particular, that $H$ is insoluble.

As $G$ is generated by involutions, so are $H$ and $H/\PSL_2(q)$. On the other hand $\PGammaL(2,q) /\PSL_2(q)$ is isomorphic to $C_2 \times C_e$ for some $e$, by \cite[Prop. 2.2.3]{kl}, and therefore $H/\PSL_2(q)$ is an elementary abelian $2$-group (of order at most 4).
Also if $e$ is odd, then $H\cong \PSL_2(q)$ or $H\cong \PGL_2(q)$, and (b) holds.
To complete the proof, we assume that $e$ is even, and that $H$ is not isomorphic to $\PSL_2(q)$ or $\PGL_2(q)$, and derive a contradiction.

Under the above assumptions, $q = p^e \equiv 1$ mod $4$,
and $H$ is isomorphic to one of the following three groups:
$\PSL_2(q)\rtimes \langle \delta \rangle,\,$ or $\,\langle \PSL_2(q), g\delta \rangle,\,$
or $\,\PGL_2(q)\rtimes \langle \delta \rangle,\,$ where $\delta$ is the unique non-trivial involution
in the Galois group of $\Fq$, taking $z$ to $z^{q_0}$ for all $z$, where $q_0 = p^{e/2}$ (so $q = q_0^{\,2}$),
and $g$ is an element of $\PGL_2(q) \backslash \PSL_2(q)$. 
Note here that $\delta$ centralises the `canonical' subgroup $\PSL_2(q_0) < \PSL_2(q)$ defined
over the subfield ${\mathbb F}_{q_0}$.
It follows that $H$ cannot contain $\delta$, because the Sylow $2$-subgroups of $H$ are dihedral
while the Sylow $2$-subgroups of $\PSL_2(q)\rtimes \langle \delta \rangle$
and $\PGL_2(q)\rtimes \langle \delta \rangle$ are not.
Hence $H = \langle \PSL_2(q), g\delta \rangle$.
Moreover, as $q = q_0^{\,2} \equiv 1$ mod $8$, a Sylow $2$-subgroup of $\PSL_2(q)$ is dihedral of order $2^j$ for some $j\ge 3$, and so $q - 1$ is divisible by $2^j$, and also a Sylow $2$-subgroup of $H$ is dihedral of order $2^{j+1}$. In particular,  all $2$-elements in $H\backslash \PSL_2(q)$ have order $2$ or $2^j$, and hence no element of $H\backslash \PSL_2(q)$ has order $4$.

Next, let $L$ be the subgroup of $\GL_2(q)$ consisting of scalar matrices, and consider elements of $\PSL_2(q)$ as projective images of those elements of $\GL_2(q)$ with determinant being a square in $\Fq.$ We may take $g$ to be the projective image of the matrix
$$\left(\begin{matrix}
   \zeta & 0 \\
 0 & 1 \\
  \end{matrix}\right)$$
where $\zeta$ is a non-square element in $\Fq^*$ of multiplicative order $(q-1)_2 = 2^j$. 
Now
$$(g\delta)^2 L = \left(\begin{matrix}
   \zeta & 0 \\
 0 & 1 \\
  \end{matrix}\right)\delta\left(\begin{matrix}
   \zeta & 0 \\
 0 & 1 \\
  \end{matrix}\right)\delta L = \left(\begin{matrix}
   \zeta & 0 \\
 0 & 1 \\
  \end{matrix}\right)\left(\begin{matrix}
   \zeta^{q_0} & 0 \\
 0 & 1 \\
  \end{matrix}\right) L = \left(\begin{matrix}
  \zeta^{1+q_0} & 0 \\
 0 & 1 \\
  \end{matrix}\right) L \ .$$
If $q_0\equiv 3$ mod $4$, then since $q-1 = q_0^{\,2}-1 = (q_0 +1)(q_0-1)$ we see that $2^{j-1}$ divides $q_0+1$, but $2^j$ does not, and $\zeta^{1+q_0} = \zeta^{q_0+1} = \zeta^{(q-1)/2} = \zeta^{2^j/2} = -1$, so that $(g\delta)^2 L \ne L$ but $(g\delta)^4 L = L$ and hence $g\delta$ has order $4$, a contradiction. Hence $q_0\equiv 1$ mod $4$.  This time let $h$ be the projective image of the matrix
$$\left(\begin{matrix}
   0& 1 \\
 1 & 0 \\
  \end{matrix}\right), $$
which is an element of $\PSL_2(q)$ because $-1$ is a square in $\Fq$. Then we find
$$(hg\delta)^2 L = \left(\begin{matrix}
   0 & 1 \\
 \zeta & 0 \\
  \end{matrix}\right)\delta\left(\begin{matrix}
   0 & 1 \\
 \zeta & 0\\
  \end{matrix}\right)\delta L = \left(\begin{matrix}
   0 & 1 \\
 \zeta & 0 \\
  \end{matrix}\right)\left(\begin{matrix}
   0 & 1 \\
 \zeta^{\delta} & 0\\
  \end{matrix}\right) L = \left(\begin{matrix}
  \zeta^{\delta} & 0 \\
 0 & \zeta \\
  \end{matrix}\right) L\ .$$
Now $2^{j-1}$ divides $q_0-1$, but $2^j$ does not, which implies that $\zeta^2\in{\mathbb F}_{q_0}$ and so $\zeta^2 =
(\zeta^2)^\delta = (\zeta^\delta)^2$. Thus $(hg\delta)^2 L \ne L$ but $(hg\delta)^4 L = L$, another contradiction.
\end{proof}
\medskip

The main aim of this section is to further develop the findings of Proposition \ref{p: odd order} and describe the overall structure of $(2,m,n)^*$-groups with odd Euler characteristic. In preparation for that, we need to introduce a few concepts and prove a number of helpful auxiliary facts.

A group $K$ is a {\em section} of a group $H$ if $K$ is isomorphic to a quotient of a subgroup of $H$. Our aim is to look at possibilities for a section $K$ of a group $\PGL_k(r^j)$ where $r$ is an odd prime, such that $K\cong \PSL_2(q)$ for an odd prime-power $q\ge 5$. For this, let $\veps(q,r)$ be the integer defined for such $q$ and $r$ as follows:
$$\veps(q,r)=\left\{\begin{array}{cl}
2 & \hbox{when} \ q=p^e  \ \hbox{and} \  r = p; \\
3 & \hbox{when} \ q=9 \ \hbox{and} \  r\neq 3; \\
\frac{q-1}{2} & \hbox{when} \ q=p^e \ne 9 \ \hbox{and} \  r \neq p.
\end{array}\right.$$

\begin{lem}\label{l: covers of K}
Let $r$ be a prime, let $j$ and $k$ be positive integers, and suppose that the group $\PGL_k(r^j)$ contains a section isomorphic to $\PSL_2(q)$ for some odd prime-power $q=p^e\geq 5$. Then $k\geq \veps(q,r)$.
\end{lem}

\begin{proof}
If $r=p$  then the conclusion is immediate, while if $q=9$ and $r\neq 3$, then it follows from \cite{dickson}.  Hence we may suppose that $r \neq p$ and $q\neq 9$. If the section $\PSL_2(q)$ itself is isomorphic to a {\em subgroup} of $\PSL_k(r^j)$, then the required conclusion is classical and originates from the work of Frobenius on minimum-degree projective modular representations, and is summed up as part of \cite[Proposition 5.3.9 and Table 5.3.A]{kl}. On the other hand, if
$\PSL_2(q)$ is a section but not a subgroup of $\PSL_k(r^j)$, then the bound $k\geq \veps(q,r)$ is a consequence of observations in \cite[Proposition 4.1]{FT} and then \cite[Theorem~3]{kl1}. (We note that the last two citations apply to sections of the groups $\PGL_k(F)$ for algebraically closed fields of
characteristic $r$, and our conclusion is a consequence of them.)
\end{proof}

The statement of our next helpful fact involves the Frattini subgroup $\Phi(L)$ of a finite group $L$, which is the (normal) subgroup of $L$ formed by intersection of all maximal subgroups of $L$. All we need here is that if $L$ is a $p$-group, then $\Phi(L) = L^pL'$, which is the subgroup of $L$ generated by the commutator subgroup $L'$ and $p$-th powers of all elements of $L$, so that $L/\Phi(L)$ is an elementary abelian $p$-group, isomorphic to $C_p^{\, j}$ for some $j\ge 1$. 

\begin{lem}\label{l: vector block}
Let $H$ be a group, $p$ a prime, and $L$ a normal $p$-subgroup of $H$. If $L/\Phi(L)\cong C_p^{\, j}$,
then $|\ord(h)|_p \leq |H|_p/p^{j-1}$ for every $h \in H$.
\end{lem}

\begin{proof}
First let $p^i = |H/L|_p$ and $p^k = |\Phi(L)|_p$, so that $|H|_p = |H:L|_p\,|L: \Phi(L)||\Phi(L)| = p^{i+j+k}$.
Next, if $\ord(h) = uv$ with $p$-part $u$ ($= \ord(h)_p$) and $p'$-part $v$, then we may take $h' = h^v$,
and find that $\ord(h)_p = u = \ord(h') = \ord(h')_p$, so we can replace $h$ by $h'$, or more simply,
assume without loss of generality that $h$ is a $p$-element.  Then $h^{p^i} \in L$, so $h^{p^{i+1}} = (h^{p^{i}})^p \in \Phi(L)$
and therefore $h^{p^{i+1+k}} = (h^{p^{i+1}})^{p^k} = 1$, so $\ord(h)_p$ divides $p^{i+1+k} = |H|_p/(p^{j-1})$.
\end{proof}

We continue by taking a more detailed look into the structure of an insoluble $(2,m,n)^*$-group $G$ with $\chi_G$ odd, that is, one satisfying property (b) in Proposition \ref{p: odd order}. Actually, instead of working with $\chi_G$ and its prime factorisation, we will use the the prime factorisation of $|G|/ \lcm(m,n)$, because it appears to behave better when taking quotients of $G$. In the proof we will use properties of the Fitting subgroup and the generalised Fitting subgroup of $G$.

The Fitting subgroup $F(G)$ of $G$ is the (unique) largest nilpotent normal subgroup of $G$, while the generalised Fitting subgroup $F^*(G)$ of $G$ is the unique subgroup of $G$ of the form $E(G)F(G)$, where $F(G)$ is the Fitting subgroup of $G$ and $E(G)$ is the so-called {\em layer} of $G$, which is the unique largest semi-simple normal subgroup of $G$. Note that semi-simplicity of $E(G)$ means that $E(G)$ is a central product of a finite number of quasi-simple subgroups of $G$, where a subgroup $H < G$ is quasi-simple if it is perfect (that is, its commutator subgroup $H'$ is equal to $H$) and the quotient $H/Z(H)$ of $H$ by its centre $Z(H)$ is a simple group. For more details about the generalised Fitting subgroup of a group we refer the reader to \cite[Chapter 11]{aschfgt}.

\begin{prop}\label{l: odd order ind}
Let $G$ be a $(2,m,n)^*$-group such that $G/O(G)$ is isomorphic to $\PSL_2(q)$ or $\PGL_2(q)$ for some odd prime-power $q\geq 5$, and suppose that $|G|/\lcm(m,n)=2^{a_0}p_1^{a_1}\cdots p_k^{a_k}$, where $p_1,\dots, p_k$ are odd primes and $a_1,\dots, a_k$ are positive integers.

Then $G$ has a single non-abelian composition factor isomorphic to $\PSL_2(q)$, and has a normal subgroup $N$ of odd order such that {\rm (i)} every prime divisor of $|N|$ is equal to $p_i$ for some $i$, and {\rm (ii)} if $p_i$ divides $|N|$ then $a_i\geq \veps(q,p_i)-1$, and {\rm (iii)} one of the following holds$\,:$
\begin{enumerate}
\item\label{e(K)} $N=O(G)$ and $G/O(G)\cong \PSL_2(q)$ for some odd prime-power $q\ge 5\,;$
\item $G/N\cong 3.A_6\,;$
\item\label{direct cyclic} $G/N\cong (\PSL_2(q) \times C).2$ for some odd prime-power $q\ge 5$ and a cyclic group
$C$ of odd order$\,;$
\item\label{3A6} $G/N$ is isomorphic to $3.(A_6\times C).2$ where $C$ is a cyclic group of odd order.
\end{enumerate}
 \end{prop}

\begin{proof}
First, the conclusion that $G$ has a single non-abelian composition factor isomorphic to $\PSL_2(q)$ is obvious.
We will proceed by induction on the order of $G$ to prove the other conclusions (about $N$), by assuming that the entire conclusion is valid for all $(2,m',n')^*$-groups $H$ of order less than $|G|$ such that $H/O(H)$ is isomorphic to $\PSL(2,q)$ or $\PGL(2,q)$ for some odd prime-power $q\ge 5$. For a fixed $(2,m,n)^*$-group $G$ as in the statement, let $F^*(G)=E(G)F(G)$ be the generalised Fitting subgroup of $G$, where $E(G)$ is the layer of $G$ and $F(G)$ is the Fitting subgroup of $G$.

We divide the proof into a series of claims.
\medskip

{\bf Claim 1:} {\sl The statement of Proposition \ref{l: odd order ind} is valid if the layer $E(G)$ of $G$ is
non-trivial.}
\medskip

{\bf Proof of Claim 1.}
As $E(G)$ is a central product of quasi-simple groups, each of those would give rise to a non-abelian composition factor of $G$, but $G$ has a single non-abelian composition factor $K$, so $E=E(G)$ itself is quasi-simple and hence is perfect, with $E/Z(E)\cong K$, and so $E$ is a perfect central extension of $K$.
In other words, $E$ is isomorphic to a quasi-simple cover of $K$.  The theory of perfect central extensions, as initiated by Schur \cite{sch}, shows that (in our situation) the possibilities of a perfect group $E$ having the property that $E/Z(E)\cong K$ are limited in that the centre $Z(E)$ is a subgroup of a unique maximal group $M(K)$, namely the {\em Schur multiplier} of $K$. Schur multipliers of finite simple groups are known and listed (for example) in \cite[Theorem 5.1.4]{kl}. The relevant facts for us are that $M(K)\cong C_2$ if $K\cong \PSL_2(q)$ for $q\ge 5$ and $q\ne 9$, and $M(K) \cong C_6$ if $K\cong \PSL_2(9)\cong A_6$. In particular, $|M(K)|$ divides $6$ in all these cases.

Since $Z=Z(E)$ is characteristic in the normal subgroup $E$ of $G$, it follows that $Z\lhd G$. With $O=O(G)$,  there is a normal chain $O\lhd OZ\lhd OE\lhd G$ with the middle inclusion being proper, and with $|Z|\le 6$. Then factoring out $O$ gives a chain $1\lhd OZ/O\lhd OE/O\lhd G/O$, but by our assumption the last term is either simple or contains a simple subgroup of index $2$, and it follows that $OE$ is equal to $G$ or has index $2$ in $G$, and $OZ/O$ is trivial, so that $Z \le O$, and then the bound on the order of $Z$ implies that $Z$ is trivial or isomorphic to $C_3$. Invoking Schur's theory and \cite[Theorem 5.1.4]{kl} again, we find that $E=E(G)$ is isomorphic to $\PSL_2(q)$ with $q$ odd, or to $3.A_6$.

Next, as $[G:OE]\le 2$ we have $|G/E|_2\leq 2$. If $|G/E|_2=1$, then $G/E$ is a group of odd order generated by three elements of order $1$ or $2$ and hence is trivial, so $G=E=E(G)$, and then \eqref{e(K)} or (b) holds with $N$ trivial. On the other hand, if $|G/E|_2=2$, then $|G:OE|=2$ and $G/O\cong \PGL_2(q)$, and so $E\cong \PSL_2(q)$ or $3.A_6$. Moreover, if $G=\langle a,b,c\rangle$ with presentation \eqref{eq:abc}, then $G/E$ is generated by three elements $aE, bE$ and $cE$ of order $1$ or $2$ with $[bE,cE]$ trivial, and since $|G/E|_2=2$ it follows that $\langle bE, cE\rangle$ has order $1$ or $2$, and so $G/E$ is generated by at most two involutions and hence is cyclic or dihedral. Then $OE/E$ is cyclic (of odd order), and since $OE/E\cong O/(O\cap E)$ we conclude that $O/(O\cap E)$ is cyclic. If $O\cap E$ is trivial, then $EO \cong E\times O$, so $G\cong (E\times O).2$ and either \eqref{direct cyclic} or \eqref{3A6} holds with $N$ trivial, while if $E\cap O$ is non-trivial, then $E\cap O$ is normal in $E$ and the only possibility $E=3.A_6$ with  $|E\cap O|=3$, so \eqref{3A6} holds with $N$ trivial. \hfill $\Box$
\medskip

From now on we can assume that $E=E(G)$ is trivial, so that $F^*(G)=F(G)$, which is the largest normal nilpotent subgroup of $G$ and hence a direct product of groups of prime-power order, each characteristic in $F=F(G)$ and hence normal in $G$. Now if some $2$-group $L$ appears as a direct factor of $F$, then the normal chain $O\lhd OL\lhd G$ and the corresponding chain $1\lhd OL/O\lhd G/O$ together with our assumption on $G/O$ imply that $|OL|$ is equal to $|G|$ or $|G|/2$ and hence $|L|$ is equal to $|G/O|$ or $|G/O|/2$, a contradiction since $|L|$ is a power of $2$. It follows that $F = R_1 \times\cdots \times R_t$ where each $R_i$ is a non-trivial $r_i$-group for some odd prime $r_i$, and all the $r_i$ are distinct ($1\le i\le t$). In particular, $F=F(G)$ is a subgroup of $O=O(G)$, the largest odd-order normal subgroup of $G$.
\medskip

{\bf Claim 2:} {\sl If the layer $E(G)$ of $G$ is trivial, and a direct factor of the Fitting subgroup $F(G)$ has a soluble centraliser in $G$, then the conclusion of Proposition \ref{l: odd order ind} is valid for the group $G$.}
\medskip

{\bf Proof of Claim 2.}
With notation as above, suppose that for some $i\in \{1,2,\ldots,t\}$ the centraliser $C_G(R_i)$ is soluble, and for notational convenience, let $R = R_i$ and $r = r_i$ (so that $R$ is an $r$-group). Then since $R$ is normal in $G$, we know that $C_G(R)$ normal in $G$, and so the quotient $G/C_G(R)$ is insoluble and contains $K\cong \PSL_2(q)$ as a section.
Moreover, as $G/C_G(R)$ is isomorphic to a subgroup of $\Aut(R)$, the group $K$ is isomorphic to a section of $\Aut(R)$ as well.

Next, every automorphism of $R$ preserves the (characteristic) Frattini subgroup $\Phi(R)$ of $R$, and so there exists a natural epimorphism $\phi:\ \Aut(R) \to \Aut(R/\Phi(R))$.
By a theorem of Burnside \cite[Theorem 1.4, p.174]{gor}, the kernel of $\phi$ is an $r$-group, and hence is soluble, and it follows that $\Aut(R)/\ker\phi\cong \Aut(R/\Phi(R))$ also contains a copy of $K$ as a section. But $R/\Phi(R)$ is elementary abelian, so $\Aut(R/\Phi(R))$ is isomorphic to $\GL_k(r)$ for some $k$,
and then by Lemma~\ref{l: covers of K} we find that $k\geq\veps(q,r)$. Moreover, by Lemma~\ref{l: vector block} applied to elements of order $m$ and $n$ in our $(2,m,n)^*$-group $G$ and to its normal $r$-subgroup $R$, we  have $|m|_r, |n|_r \leq |G|/ r^{\veps(q,r)-1}$. It follows that $r^{\veps(q,r)-1}$ (and then also $r$) divides  $|G|/\lcm(m,n)$, and hence that $r\in \{p_1,p_2, \ldots, p_k\}$, which is the list of primes in the statement of Proposition \ref{l: odd order ind}, say $r=p_j$, and that $a_j\geq \veps(q,p_j)-1$ for the exponent $a_j$ of the prime $p_j$ in the same statement.

The normal chain $R\lhd F\lhd O\lhd G$ induces the chain $1\lhd F/R\lhd O/R \lhd G/R$, which implies that the largest odd-order normal subgroup $O(G/R)$ of the quotient $H=G/R$ may be identified with $O/R$, so that $G/O\cong (G/R)/(O/R) \cong H/O(H)$. It follows that the group $H$ inherits one of the basic properties assumed of $G/O$, namely that $H/O(H) \cong \PSL(2,q)$ or $\PGL(2,q)$ for some odd prime-power $q\ge 5$.  Thus $K= \PSL_2(q)$ is a single non-abelian composition factor of $H$. Furthermore, as $G=\langle a,b,c\rangle$ is a $(2,m,n)^*$-group with $m=\ord(ca)$ and $n=\ord(ab)$, it follows that $H = G/R = \langle aR,bR,cR\rangle$ is a $(2,m',n')^*$-group for $m'=\ord(caR)$ and $n'=\ord(abR)$, with both $m/m'$ and $n/n'$ being powers of $r$.

We may now apply our induction hypothesis, and assume that the conclusions of Proposition \ref{l: odd order ind} hold for the group $H=G/R$. In particular, writing $|H|/\lcm(m',n')=2^{a_0}p_1^{a_1}\cdots p_{j-1}^{a_{j-1}} p_j^{\alpha_j} p_{j+1}^{a_{j+1}}\cdots p_k^{a_k}$ where, except for the subscript $j$ for which $p_j=r$, the primes $p_1,\dots, p_k$ and the positive integers $a_1,\ldots, a_k$ are as in the statement of Proposition \ref{l: odd order ind}, while $\alpha_j \ge 0$. Also by the induction hypothesis, $H$ has a normal subgroup $L$ of odd order such that ${\rm (i')}$ every prime divisor of $|L|$ is equal to $p_s$ for some $s\in \{1,2,\ldots,k\}$, and ${\rm (ii')}$ if $p_s$ divides $|N|$ then $a_s\geq \veps(q, p_s)-1$ for $s\ne j$, while $\alpha_j\ge \veps(q,p_j)-1$  if (and only if) $\alpha_j\ge 1$, and ${\rm (iii')}$ one of the following holds$\,:$
\begin{enumerate}
\item[${\rm (a')}$]\label{e(K')} $L=O(H)$ and $H/O(H)\cong \PSL_2(q)$ for some odd prime-power $q\ge 5\,;$
\item[${\rm (b')}$]\label{3A7'} $H/L \cong 3.A_6\,;$
\item[${\rm (c')}$]\label{direct cyclic'} $H/L\cong (\PSL_2(q) \times C).2$ for some odd prime-power $q\ge 5$ and a cyclic group $C$ of odd order$\,;$
\item[${\rm (d')}$]\label{3A6'} $H/L$ is isomorphic to $3.(A_6\times C).2$ where $C$ is a cyclic group of odd order.
\end{enumerate}

By the correspondence theorem for groups, there exists a group $N$ such that $R\lhd N\lhd O\lhd G$ and  $N/R\cong L$, and then $G/N\cong (G/R)/(N/R) \cong H/L$. By ${\rm (i')}$, every prime divisor of $|N|$ is either $p_j=r$, or $p_s$ for some $s\ne j$. Also ${\rm (ii')}$ shows that $a_s\geq \veps(q, p_s)-1$ for $s\ne j$, and the bound $a_j\ge \veps(q, p_j)-1$ for the exponent of $p_j=r$ follows from the above, independently of the value of $\alpha_j$. This establishes properties (i) and (ii) from the statement of Proposition
\ref{l: odd order ind}, and  validity of items (a) to (d) of property (iii) for the subgroup $N\lhd G$ with $N/R\cong L$ follows from items ${\rm (a')}$ to ${\rm (d')}$ above, with the help of the isomorphism $H/L\cong G/N$.
\hfill $\Box$
\medskip

{\bf Claim 3.} {\sl If $F^*(G)=F(G)$, then some direct factor of $F(G)$ has soluble centraliser in $G$}.
\medskip

{\bf Proof of Claim 3.}
Define $G_o=G$ if $G/O\cong K$, while if $G/O\cong \PGL(2,q) \cong K{\cdot}2$ define $G_o$ as the unique subgroup of index $2$ in $G$  such that $G_o/O\cong K$.  Then $O=O(G)=O(G_o)$ in both cases. Also as before, let $F=F^*(G) = F(G) = R_1 \times\cdots \times R_t$ with non-trivial $r_i$-groups $R_i$ for distinct odd primes $r_i$ ($1\le i\le t$), so that $F\lhd O$; and now let $C_{(i)}$ be the centraliser $C_G(R_i)\cap G_o$ of $R_i$ in $G_o$ for $1 \le i \le t$.

Assume that the centralisers $C_G(R_i)$ are all insoluble, for  $1\le i\le t$.  Then so are the centralisers $C_{(i)}$ in $G_o$, and it follows that $C_{(i)}O/O \cong C_{(i)}/(C_{(i)}\cap O)$ is also insoluble, for $1 \le i \le t$. Now simplicity of $K$ and the correspondence theorem for groups imply that the normal subgroup $C_{(i)}O/O$ of $G_o/O\cong K$ is either trivial or isomorphic to $K$, which means that $C_{(i)}O = O$ or $G_o$. The first possibility would give $C_{(i)} \le O$, contrary to insolubility of $C_{(i)}$, and therefore $C_{(i)}O=G_o$ for {\em every} $i\in \{1,2,\ldots,t\}$. Then if $C_*$ is the intersection $C_{(1)}\cap \ldots\cap C_{(t)}$ of all the $C_{(i)}$,  then $C_*O=G_o$, and in particular, $C_*$ is insoluble. But also $C_*$ centralises $R_1 \times\cdots \times R_t = F$, which enables us to use Bender's theorem \cite[Theorem 31.13]{aschfgt}, which gives $C_G(F^*(G))\le (F^*(G))$ and hence implies in our case (where $F^*(G)=F(G)=F$) that $C^*$ is soluble, a contradiction.
\hfill $\Box$
\medskip

This completes the proof of Proposition \ref{l: odd order ind}: if $F^*(G)\ne F(G)$,
then it follows from Claim 1, while if $F^*(G)=F(G)$, then by Claim 3 at least one direct factor of $F(G)$ has a
soluble centraliser in $G$, and then the induction step is accomplished by Claim 2.
\end{proof}
\medskip

We can now prove our main structure theorem for odd $\chi_G$.

\begin{thm}\label{t: odd order}
Let $G$ be a $(2,m,n)^*$-group with $\chi_G$ odd,
and  suppose $|G|/\lcm(m,n)=2^{a_0}p_1^{a_1}\cdots p_k^{a_k}$
where $a_0$ is a non-negative integer, $k$ is a positive integer,
$p_1,\dots, p_k$ are odd primes, and $a_1,\dots, a_k$ are positive integers.

If $G$ is soluble, then $G/O(G)$ is isomorphic to either a $2$-group or $S_4$.

If $G$ is insoluble, then $G/O(G)$ is isomorphic to $\PSL_2(q)$ or $\PGL_2(q)$ for some odd prime-power $q\ge 5$.
Moreover, in that case $G$ has a single non-abelian composition factor $K\cong \PSL_2(q)$, and has a normal subgroup $N$ of odd order such that {\rm (i)} every prime divisor of $|N|$ is equal to $p_j$ for some $j$, {\rm (ii)}  if $p_j$ divides $|N|$ then $a_j\geq \veps(q, p_j)-1$, and {\rm (iii)} one of the following holds$\,:$
\begin{enumerate}
\item $N=O(G)$ and $G/O(G)\cong \PSL_2(q)$ for some odd prime-power $q\ge 5\,;$
\item $G/N\cong 3.A_6\,;$
\item $G/N\cong (\PSL_2(q) \times C).2$ for some odd prime-power $q\ge 5$ and cyclic group $C$ of odd order$\,;$
\item $G/N$ is isomorphic to $3.(A_6\times C).2$ where $C$ is a cyclic group of odd order.
\end{enumerate}
\end{thm}

\begin{proof}
We simply apply Proposition~\ref{p: odd order} and the first sentence of its proof. If $G$ is soluble, then either $G/O(G)$ is a $2$-group (and we are done), or $G/O(G)\cong \PSL_2(q)$ or $\PGL_2(q)$ for some odd $q$, necessarily with $q < 5$ and hence $q = 3$. If $G/O(G)\cong PSL_2(3) \cong A_4$, then $G$ has a non-trivial quotient of odd order, which is impossible, so $G/O(G)\cong PGL_2(3) \cong S_4$. On the other hand, if $G$ is insoluble, then Propositions~\ref{p: odd order} and ~\ref{l: odd order ind} apply, and the remaining conclusions follow.
\end{proof}

\subsection{Some notation and useful basic lemmas}

Theorem~\ref{t: odd order} will be the basis for the proof of Theorem~\ref{t: main} which now follows. Now is a good moment to fix some notation.

For $G= \langle a,b,c\rangle$ a $(2,m,n)^*$-group with presentation \eqref{eq:abc}, we will continue to write $O=O(G)$ and we will set $\overline{G}=G/O$. We will use the symbols $\overline{m}$ and $\overline{n}$ for the orders of the elements $abO$ and $bcO$ in $\overline{G}$,  respectively; in particular $\overline{G}$ is a $(2,\overline{m}, \overline{n})^*$-group. We will continue to use the notation $|j|_t$ (or sometimes $j_t$) to denote the largest power of the prime $t$ that divides the positive integer $j$.

We conclude this section with three useful lemmas, the first of which is a re-statement of Lemma 3.2 in \cite{cps}.

\begin{lem}\label{l: sylow}
If $G$ is a $(2,m,n)^*$-group, and $t$ is an odd prime divisor of $|G|$ with $\gcd(t,\chi_G)=1$, then every
Sylow $t$-subgroup of $G$ is cyclic.
\end{lem}

\begin{lem}\label{l: 2}
If $G$ is a $(2,m,n)^*$-group with odd $\chi_G$, and $|\overline{G}|_2 > 2\,|\overline{m}|_2 \, |\overline{n}|_2$, then $|G|_2=4$, and $m$ and $n$ are both odd.
\end{lem}

\begin{proof}
Suppose at least one of $m$ and $n$ is even. Then Euler's formula \eqref{e: sunny} with $\chi_G$ odd
implies that $-\chi_G = u(m,n)\cdot v(m,n)$, where
\[u(m,n)=\frac{|G|}{2\,\lcm(m,n)} \ \ \ {\rm and} \ \ \ v(m,n)=\frac{mn-2m-2n}{2\gcd(m,n)} \]
are both odd integers, noting that $|G|$ is divisible by $2\,\lcm(m,n)$ because $G$ contains dihedral
subgroups of orders $2m$ and $2n$. Letting $m=2^im'$ and $n=2^jn'$ with $m', n'$ odd, it is easy to see
that oddness of $u(m,n)$ implies that $|G|_2=2^{\max(i,j)}$, but then $2|m|_2|n|_2 = 2^{1+i+j} \ge |G|_2$,
and hence also $2\,|\overline{m}|_2 \, |\overline{n}|_2 \ge |\overline{G}|_2$,
a contradiction.
Therefore $m$ and $n$ are both odd, and then from Euler's formula \eqref{e: sunny} it follows that
$|G|_2 = ||G|(mn-2m-2n)|_2 = |4mn|_2 = 4$.
\end{proof}

\begin{lem}\label{l: odd prime}
If $G$ is a $(2,m,n)^*$-group with $\chi_G=-r^d,$ and $t$ is an odd prime such that $|\overline{G}|_t >
|\overline{m}|_t\, |\overline{n}|_t$, then $t = r$.
\end{lem}

\begin{proof}
As $|\overline{G}|_t > |\overline{m}|_t\, |\overline{n}|_t \ge |\lcm(\overline{m}, \overline{n})|_t$, we find that $t$ divides $|\overline{G}|/ \lcm(\overline{m}, \overline{n})$, and hence $t$ divides $|G|/\lcm(m,n)$. Then Euler's formula \eqref{e: sunny}  implies that $t$ is a divisor of $-\chi_G = r^d$, and so $t=r$.
\end{proof}


\section{The case where \texorpdfstring{$G/O(G)\cong \PSL_2(q)$}{G/O is PSL(2,q)}: proof of Theorem \ref{t: main} for family A}\label{s: insol-PSL}


Our aim in this section is to prove part (A) of Theorem~\ref{t: main}, the conclusions of which are summed up in Table~\ref{t: a}. Specifically, we will prove the following.

\begin{prop}\label{p: psl2 main}
Let $G$ be a $(2,m,n)^*$-group with $\chi_G=-r^d$ for some odd prime $r$, and suppose $\overline{G} = G/O\cong \PSL_2(q)$ for some odd prime-power $q\ge 5$. Then $O=O(G)$ is an $r$-group, and there are four possibilities for the quintuple $(q,m,n,r,d)$, as listed in rows {\rm A1} to  {\rm A4} of {\rm Table~\ref{t: a}}.
\end{prop}

The proof of Proposition~\ref{p: psl2 main} will be preceded by a series of lemmas. For $G$ as given, $\overline{m}$ and $\overline{n}$ are orders of elements in $\PSL_2(q)$ for $q=p^e\ge 5$ and hence each is a divisor of one of the relatively prime integers $(q-1)/2$, $(q+1)/2$ and $p$. Since the corresponding regular maps are non-orientable, the theory summed up in the Introduction implies that elements of order $\overline{m}$ and $\overline{n}$ must form a generating pair for $\overline{G}\cong \PSL_2(q)$ with the property that their product has order $2$. This means that either $\{\overline{m},\overline{n}\}=\{3,5\}$ or the pair $(\overline{m}, \overline{n})$ is hyperbolic, that is, $1/\overline{m} + 1/\overline{n} < 1/2$.
\medskip

In statements of all auxiliary findings in this section we will assume the following hypothesis.
\medskip

\noindent {\bf Hypothesis 1:} {\sl The group $G= \langle a,b,c\rangle$ with presentation of the form \eqref{eq:abc} is a $(2.m.n)^*$-group with Euler characteristic $\chi_G=-r^d$ for some odd prime $r$ and positive integer $d$, such that $G/O(G)$ is isomorphic to $\PSL_2(q)$ for some odd prime-power $q\ge 5$.}
\medskip

\begin{lem}\label{l: not A5}
Either $q=5$, and $\{\overline{m},\overline{n}\}=\{3,5\}$ or $\{5,5\}$,
or $r=p$, or $r$ divides one of $(q\pm 1)/2$.
\end{lem}

\begin{proof}
The conclusion for $q=5$ is one easy possibility (realised by the only ways to obtain $A_5 \cong \PSL(2,5)$ as a smooth quotient of an ordinary $(2,\overline{m},\overline{n})$ triangle group),
and so now we may suppose that $q>5$.
Also if $|\overline{G}| = q(q^2-1)/2$ is divisible by some odd prime $t$ that does not divide $\overline{m}\,\overline{n}$, then Lemma~\ref{l: odd prime} implies that $t=r$, and so in that case either $r=p$, or $r$ divides one of $(q\pm 1)/2$.
Note that this must happen if $\overline{m}=\overline{n} =p$.

So from now on, suppose that no such $t$ exists, so that every odd prime divisor $t$ of $|\overline{G}|$ divides either $\overline{m}$ or $\overline{n}$; indeed we may suppose that every such $t$ divides exactly one of $\overline{m}$ and $\overline{n}$.  Taking $t = p$, this implies without loss of generality that $\overline{m}=p$,
while $\overline{n}$ divides one of $(q\pm 1)/2$, and moreover, $(q+1)/2$ and $(q-1)/2$ cannot both have
an odd prime divisor.  It follows that $q\mp 1$ must be a power of $2$, and divisible by $4$ because  $q>5$,
so $|\overline{G}|_2 = (q\mp 1)/2 \ge 4$.  Moreover, since $\gcd((q+1)/2,(q-1)/2) = 1$, we see that $\overline{n}$
must be odd. In particular, $|\overline{G}|_2 > |\overline{m}|_2|\overline{n}|_2 = 1$,
which implies that $(q\mp 1)/2 = |\overline{G}|_2 = 4$ by Lemma~\ref{l: 2}. But then $q = 7$ or $9$,
so either $\overline{G} = \PSL(2,7)$ and $(2,\overline{m},\overline{n}) = (2,7,3)$,
or $\overline{G} = \PSL(2,9)$ and $(2,\overline{m},\overline{n}) = (2,3,5)$,
neither of which is possible for a  $(2,m,n)^*$-group with negative Euler characteristic.
\end{proof}

We now look into the possibilities for $r$ in relation to $q$, in more detail.

\begin{lem}\label{l: psl restrictions 1}
If $r=p$, then $q\ge 11$ and one of the following holds$\,:$
\begin{enumerate}
\item\label{a} if $|G|_2=4$, then either $\{\overline{m},\overline{n}\} = \{(q\pm1)/2, (q\mp1)/4\}$ for $q\equiv \pm 1$ {\rm mod} $4$, \\ ${}$ \ or $\{\overline{m},\overline{n}\} = \{(q-1)/2,(q+1)/2\}\,;$
\item\label{b} if $|G|_2>4$, then $\{\overline{m},\overline{n}\} = \{(q-1)/2,(q+1)/2\}$.
\end{enumerate}
\end{lem}

\begin{proof}
If there exists an odd prime $t$ dividing $(q-1)/2$ or $(q+1)/2$ such that $|\overline{m}|_t\, |\overline{n}|_t < |q^2-1|_t$, then Lemma~\ref{l: odd prime} implies that $t = p$, a contradiction. Similarly, if $|\overline{m}|_2\, |\overline{n}|_2 < |(q^2-1)/8|_2$, then $|\overline{m}|_2\, |\overline{n}|_2 < |\overline{G}|_2\,/\,4$ and so
$|\overline{G}|_2>4$, contradicting Lemma~\ref{l: 2}. Hence 
one of $\overline{m}$ and $\overline{n}$ must be divisible by $(q\pm 1)/2$ and the other by $(q\mp 1)/4$,
and therefore $\{\overline{m}, \overline{n}\} = \{(q\pm1)/2, (q\mp1)/4\}$ or $\{(q-1)/2,(q+1)/2\}$. Also if $\{\overline{m}, \overline{n}\} = \{(q\pm1)/2, (q\mp1)/4\}$, then $|\overline{m}|_2\, |\overline{n}|_2 =|(q^2-1)/8|_2 = |\overline{G}|_2\,/\,4$, and then by Lemma~\ref{l: 2} we have $|G|_2=4$. Finally, the lower bound on $q$ is a consequence of the non-existence of non-orientable regular maps with the required type for $q = 7$ or $9$,
as shown by the list at \cite{C600}.
\end{proof}

Hypothesis 1 and Lemma \ref{l: psl restrictions 1} now strongly restrict the structure of the group $O=O(G)$.

\begin{lem}\label{lem:r-gp} If the group $O=O(G)$ is non-trivial, then it is an $r$-group.
\end{lem}

\begin{proof} When $\chi_G=-r^d$, Euler's formula \eqref{e: sunny}  implies that the only odd prime divisor of the integer $|G|/\lcm(m,n)$ is $r$, and so $|G|/\lcm(m,n)$ has prime factorisation of the form $2^i r^j$ for some non-negative integers $i$ and $j$. Then Theorem \ref{t: odd order} implies that $G/O\cong \PSL_2(q)$,
and that $G$ has a normal subgroup $N$ of odd order, indeed a normal $r$-subgroup (because $r$ is the only odd prime divisor of $|G|/\lcm(m,n)$). Moreover, since $N \le O$, with $O/N$ having odd order, cases (c) and (d) of Theorem \ref{t: odd order} are both impossible, so case (a) or (b) holds.
If case (b) holds, however, then $G/O\cong A_6 \cong \PSL_2(9)$, with $O/N \cong C_3$.
Now the Sylow $3$-subgroups of $A_6$ are not cyclic, so by Lemma~\ref{l: sylow} we find that $3$ divides $\chi_G$ and hence $r = 3$ ($= p$), and also $|G|_2 \ge |A_6|_2 = 8$, and therefore by Lemma \ref{l: psl restrictions 1} we find that $\{\overline{m}, \overline{n}\} = \{(9-1)/2,(9+1)/2\} = \{4,5\}$. But the list at \cite{C600} does not include a non-orientable regular map of type $\{4,5\}$ with automorphism group of order $360=|A_6|$, and so we have a contradiction.
%
Hence case (a) of Theorem \ref{t: odd order} holds, and gives $O=N$, so that $O$ is either trivial or an $r$-group.
\end{proof}

\begin{lem}\label{l: psl restrictions 2}
If $r$ is a divisor of $q\pm 1$, then $q=p$, so $e=1$ and $p\ge 5$, and $\overline{G} \cong \PSL(2,p)$.
Furthermore, if $\beta\ge 1$ is the largest integer for which $r^\beta\mid p\pm 1$,
then one of the following holds$\,:$
\begin{enumerate}
\item[{\rm (c)}] 
either $\{\overline{m},\overline{n}\} = \{p, (p-1)/4\}$ where $p\ge 11$, $p\equiv 1$ {\rm mod} $4$,
   and $p+1 = 2r^\beta$, \\
or $\{\overline{m},\overline{n}\} = \{p, (p+1)/4\}$ where $p\ge 11$, $p\equiv -1$ {\rm mod} $4$,
   and $p-1 = 2r^\beta\,;$
\item[{\rm (d)}] 
either $\{\overline{m},\overline{n}\} = \{p, (p-1)/2\}$ where $p\ge 7$ and $(p+1)/r^\beta$ is a power of $2$, \\
or $\{\overline{m},\overline{n}\} = \{p, (p+1)/2\}$ where $p\ge 7$ and $(p-1)/r^\beta$ is a power of $2\,;$
\item[{\rm (e)}] $\overline{m}=\overline{n}=p=5$ and $r=3\,;$
\item[{\rm (f)}] $\overline{m}=p=5$ and $\overline{n}=r=3$.
\end{enumerate}
\end{lem}

\begin{proof}
As $r$ divides $q\pm 1$, we have $r\ne p$, so $\gcd(p,\chi_G) = 1$, and hence by Lemma \ref{l: sylow}, the Sylow $p$-subgroups of $G$ and hence those of $\overline{G}$ are cyclic. But the latter are elementary abelian of rank $e$ when $q=p^e$, and therefore $e = 1$ and $p = q \ge 5$.

Next, suppose that $\overline{m}= \overline{n}=p$. Then $m/p$ and $n/p$ are divisors of $|O|$, so $m= pr^i$ and $n = pr^j$ for some $i$ and $j$, and $|G|=|G/O|{\cdot}|O| = p(p^2-1){\cdot}|O|/2$,
so Euler's formula \eqref{e: sunny} gives
$$
r^d = -\chi_G
= \frac{p(p^{2}-1)|O|}{2}\, \left ( \frac{p^{2}r^{i+j} - 2pr^i - 2pr^j}{4p^{2}r^{i+j}}\right )
= \frac{(p^{2}-1)|O|}{8r^{i+j}}\, \left ( pr^{i+j} - 2r^i - 2r^j \right ).
$$
Hence the integer $(p^2-1)/8$ is a power of $r$. But only one of $(p \pm 1)/2$ can be divisible by $r$ here, and so $p \mp 1$ must be equal to $4$, and then since $p\ge 5$ we find that $p=5$ and $r=3$, giving (e).

For the rest of the proof we will assume without loss of generality that $\overline{n}\ne p$. Also since $r$ divides $q \pm 1 = p\pm 1$, we will now assume that $r$ divides $p+1$, in which case $|p+1|_r = r^\beta$,
and comment on the alternative possibility that $r$ divides $p-1$ at the end.

Suppose that $t$ is an odd prime satisfying $|\overline{m}|_t\,|\overline{n}|_t < |p(p\pm 1)|_t$. Then by Lemma
~\ref{l: odd prime} with $|\overline{G}|=p(p^2-1)/2$ it follows that $t = r$. On the other hand, as $p \ne r$, Lemma
~\ref{l: odd prime} gives also $|\overline{m}|_p\, |\overline{n}|_p \ge |p(p\pm 1)|_p= p$.
Then since each of $\overline{m}$ and $\overline{n}$ is equal to $p$ or a divisor of $(p\pm 1)/2$, but $\overline{n}\ne p$ and $|\overline{m}|_p\, |\overline{n}|_p \ge p$, we find that $\overline{m}=p$, while $\overline{n}$ divides either $(p-1)/2$ or $(p+1)/2$. 
Also for every odd prime $t\ne r$, Lemma~\ref{l: odd prime} gives
$|\overline{m}|_t\, |\overline{n}|_t \ge |p(p\pm 1)|_t$, and so $|\overline{n}|_t \ge |(p\pm 1)|_t$
(including in the case where $t = p$, trivially).

Now suppose that $\overline{n}$ divides $(p-1)/2$.  Then since $p$ and $r$ do not divide $p-1$,
the odd part of $\overline{n}$ is equal to the odd part of $(p-1)/2$.
Moreover, if $|\overline{n}|_2 \ge \frac{1}{2} |\frac{p-1}{2}|_2$, then $\overline{n} = (p - 1)/4$ or $(p - 1)/2$,
while if $|\overline{n}|_2 < \frac{1}{2} |\frac{p-1}{2}|_2$, then
$2|\overline{m}|_2\, |\overline{n}|_2 = 2|\overline{n}|_2 < |\frac{p - 1}{2}|_2 < |\overline{G}|_2$,
in which case Lemma~\ref{l: 2} tells us that $|\overline{G}|_2 = |G|_2 = 4$ and $\overline{n}$ is odd,
so $|\frac{p - 1}{2}|_2 \le 2$ and again $\overline{n} = (p - 1)/4$ or $(p - 1)/2$.
Also if $\overline{n} = (p - 1)/4$ then $p \equiv 1$ mod $4$ but $|\frac{p - 1}{2}|_2 = 2$,
so $\frac{p +1}{2}$ is odd and hence $\frac{p +1}{2} = r^\beta$, while if $\overline{n} = (p - 1)/2$
then $\frac{p +1}{r^\beta}$ is a power of $2$ (and equal to $2$ or greater than $2$,
depending on whether $\frac{p +1}{2}$ is even or odd).
These give the first halves of conclusions (c) and (d), with the inequalities for $p$ resulting from the requirements that elements of order $\overline{m}$ and $\overline{n}$ form a generating pair for $\overline{G} \cong \PSL(2,p)$ where $p \ge 5$, and that the odd prime $r$ divides $(p-1)/2$.

On the other hand,  suppose $\overline{n}$ divides $(p+1)/2$.  Then $|\overline{n}|_t = |(p+1)|_t$ for every odd prime $t \ne r$ (including $p$, trivially), and it follows that $(p-1)/2$ must be a power of $2$, and so both ${\overline n}$ and $(p+1)/2$ must be odd.
Next let $\alpha$ ($\le \beta$) be the integer defined by the equation
$r^\alpha\,|\overline{n}|_r = |(p + 1)/2|_r$, so that $\overline{n} = (p+1)/(2r^\alpha)$,
and let $\lambda$ and $\mu$ be the integers defined by $m = {\overline m}r^\lambda = pr^\lambda$
and  $n = {\overline n}r^\mu$.
Also by Lemma~\ref{lem:r-gp}, we may suppose that $|O(G)| = r^c$ for some integer $c$,
which gives  $|G| = |G/O(G)||O(G)| = |{\overline G}|\,|O(G)| = \frac{p(p+1)(p-1)}{2}\,r^c$.
We can now apply Euler's formula \eqref{e: sunny} to the $(2,m,n)^*$-group $\overline{G}$, and find that
\smallskip

\begin{center}
\begin{tabular}{lcl}
$-r^d \ = \ -\chi$ & $=$ & $|G| \left ( \frac{1}{4} - \frac{1}{2m} - \frac{1}{2n} \right )
  \ = \ |\overline{G}| \left ( \frac{1}{4} - \frac{1}{2pr^\lambda} - \frac{1}{2\overline{n}r^\mu} \right )$ \\[+4pt]
 & $=$ & $\frac{p(p-1)(p+1)}{2}\,r^c
    \left ( \frac{p\overline{n}r^{\lambda+\mu} - 2\overline{n}r^\mu - 2pr^\lambda}{4p\overline{n}r^{\lambda+\mu}} \right )$ \\[+4pt]
 & $=$ & $\frac{(p-1)(p+1)}{2}\,r^c
    \left ( \frac{p\overline{n}r^{\lambda+\mu} - 2\overline{n}r^\mu - 2pr^\lambda}{4\overline{n}r^{\lambda+\mu}} \right )$ \\[+4pt]
 & $=$ & $\frac{(p-1)}{4} \frac{(p+1)}{2\overline{n}} \,r^c
    \left ( \frac{p\overline{n}r^{\lambda+\mu} - 2\overline{n}r^\mu - 2pr^\lambda}{r^{\lambda+\mu}} \right )$, \\[+4pt]
\end{tabular}
\end{center}
where $\frac{(p-1)}{4}$ is a power of $2$, and $\frac{(p+1)}{2\overline{n}} = r^\alpha$.
Then since $r^{\lambda+\mu}$ is odd,
it follows that $\frac{(p-1)}{4} = 1$, so $p = 5$, and then since $\overline{n}$ divides $(p+1)/2 = 3$,
also $\overline{n} = 3$ and $r = 3$, which is (f).

In summary, these findings for the case where $r$ divides $p+1$ imply the first halves of conclusions (c) and (d), as well as conclusion (f).

On the other hand, if $r$ divides $p-1$ then it is sufficient to interchange the `plus' signs with the `minus' signs in the above argument, and obtain the second halves of conclusions (c) and (d).
In this case there is no analogue of conclusion (f), because if $r$ divides $p-1$, and $\overline{n}$ divides $(p-1)/2$, then it follows from Euler's formula for the $(2,m,n)^*$-group $\overline{G}$ that $\frac{p+1}{4} = 1$,
and so $p = 3$ and then $r$ divides $2$, a contradiction.
\end{proof}

We can now prove our main finding of this section, namely Proposition~\ref{p: psl2 main}, with the help of a sequence of further lemmas, all of which refer to the cases listed in Lemmas \ref{l: not A5}, \ref{l: psl restrictions 1} and
\ref{l: psl restrictions 2}. In what follows we let $m_o= m/\overline{m}$ and $n_o=n/\overline{n}$, which are powers of $r$ dividing $|O|$ by Lemma \ref{lem:r-gp}, and we write $n_o/m_o=r^a$ where $a$ is an integer (which could positive, zero or negative).

\begin{lem}\label{l: q5}
If $q=5$, then $r=3$ and $(\{\overline{m},\overline{n}\}, \{m,n\}) = (\{3,5\}, \{3,15\})$ or $(\{5,5\}, \{5,5\})$.
\end{lem}

\begin{proof}
Here $p = 5$, and Lemmas \ref{l: not A5} and \ref{l: psl restrictions 1}  when taken together for $q=5$ tell us that either $\{\overline{m},\overline{n}\} = \{3,5\}$ or $\{5,5\}$, with no immediate condition on $r$, or otherwise $r$ divides $(q\pm1)/2$ and hence $r = 3$.  But even in the latter case, Lemma \ref{l: psl restrictions 2} gives $\{\overline{m},\overline{n}\} = \{5,3\}$ or $\{5,5\}$. So now we consider these two cases in turn.

Suppose $\{\overline{m},\overline{n}\}=\{3,5\}$. Then without loss of generality, $(\overline{m}, \overline{n}) =(3,5)$, and Euler's equation \eqref{e: sunny} for the $(2,3,5)^*$-group $G$ with $|G|=|\PSL(2,5||O| = 60|O|$ and $(m,n)=(3m_o,5n_o)$ gives
\begin{equation}\label{eq:Eu60}
r^d = -\chi_{G} = 60\,|O|\left(\frac{15m_on_o-6m_o-10n_o}{60m_on_o}\right) = \frac{|O|}{\lcm(m_o, n_o)}f
\end{equation}
where the odd integer $f=(15m_on_o-6m_o-10n_o)/\gcd(m_o, n_o)$ must be a power of the prime $r$. We will show that $(m,n) = (3,15)$ and $r = 3$, by considering possibilities for the integer $a$ for which $n_o/m_o=r^a$.

If $a = 0$, then  $m_o = n_o$ and so $f=15m_o-16$, which cannot be $1$ or divisible by $r$, so $a \ne 0$..
Also if $a < 0$, then $\gcd(m_o,n_o) = n_o$ and so $f=15m_o - 6r^{-a} -10$, which is congruent to $2$ mod $3$ and hence cannot be $1$, so $f$ is divisible by $r$, and then since $m_o > n_o \ge 1$, also $10 = (15m_o - 6r^{-a}) - f$ is divisible by $r$, so $r=5$. But then  if $a = -1$ then $f = 15m_o - 40 \equiv 0-40 \equiv 10$ mod $25$, while if $a\leq -2$ then $25$ divides $m_o$ and so $f =15m_o - 6r^{-a} -10\equiv 0-10 \equiv 15$ mod $25$, and both of these are impossible for a power of $5$.

Hence $a > 0$, and $\gcd(m_o,n_o) = m_o$ with $n_o > 1$, and $f=15n_o - 6 -10r^{a}$. In particular, $f \equiv -1$ mod $5$, so $f$ cannot be $1$, and again $f$ is divisible by $r$, and then $6 = (15m_o - 10r^{a}) - f$ is  divisible by $r$, so $r=3$.
Also if $a\geq 2$, then $9$ divides $n_o$, so $f =15n_o - 6 -10r^{a} \equiv -6\mod 9$ and therefore $f = 3$, which again is impossible since $f \equiv -1$ mod $5$.  Thus $a = 1$.

Now $f=15n_o - 6 -30 = 15n_o - 36$, and if $n_o$ is divisible by $9$ then $f \equiv -36 \equiv 18$ mod~$27$, which is impossible for a power of $3$, so $n_o = 3$, giving $m_o = n_{o}/r^{a} = 1$, and so $(m,n) = (\overline{m},3\overline{n}) = (3,15)$.

For the other case, suppose $\overline{m}=\overline{n}=5$.
Then $|\overline{G}|_3 = |60|_3 = 3 > 1 = |\overline{m}|_3 |\overline{n}|_3$, so  Lemma~\ref{l: odd prime} implies that $r = 3$, and Euler's formula for $G$ reduces to $3^d = 3|O|f/\lcm(m,n)$, where $f$ is the integer $(5m_on_o-2m_o-2n_o)/\gcd(m_o,n_o)$, which must be a power of $3$.  In this case we consider the integer $a$ for which $n_o/m_o=3^a$, and suppose without loss of generality that $a\ge 0$ (since $\overline{m}=\overline{n}$).

If $a\geq 1$, then $\gcd(m_o,n_o) = m_o$ and $n_o > 1$, so $f=5n_o-2r^{a}-2$, which is not divisible by $3$, and hence  $f = 1$, and $5n_o-2r^{a} = 3$, which is not divisible by $9$, so $a = 1$, but then $5n_o = 3+2r = 9$, which is impossible. Thus $a=0$, which gives $m_o = n_o$, and so $f= (5m_on_o-2m_o-2n_o)/\gcd(m_o,n_o)=5n_o-4$, which is not divisible by $3$, and so $5n_o-4 = f = 1$, from which it follows that $m_o=n_o=1$ and therefore $(m,n) = (\overline{m},\overline{n}) = (5,5)$.
\end{proof}

\begin{lem}\label{l: a-f}
If one of the cases {\rm (a)} to {\rm (f)} of Lemmas~{\rm \ref{l: psl restrictions 1}} and~{\rm \ref{l: psl restrictions 2}} applies, then $(m_o, n_o)=(1,1)$.
\end{lem}

\begin{proof}
Observe that in cases (e) and (f) of Lemma ~\ref{l: psl restrictions 2}, the conclusion that $m_o=n_o=1$ follows from Lemma \ref{l: q5}.
For the remaining cases, we may take $|G|=|\PSL_2(q)|\,|O|$, and $m= \overline{m}m_o$ and $n=\overline{n}n_o$ in Euler's formula \eqref{e: sunny}, and suppose without loss of generality that $m_o\ge n_o$, and obtain
\begin{equation}\label{e: v1}
r^d = \frac{1}{2} \cdot \frac{q(q^2-1)}{4\,\overline{m}\,\overline{n}}\cdot \frac{|O|}{m_o}\cdot f\  \ \ \
{\rm where} \ \ \ f=\overline{m}\,\overline{n}\,m_o-2\,\overline{m}\frac{m_o}{n_o}-2\,\overline{n}.
\end{equation}

Also one may check in the cases (a) to (d) of Lemmas~~\ref{l: psl restrictions 1} and ~\ref{l: psl restrictions 2} that $q(q^2-1)/(4\,\overline{m}\, \overline{n})$ is always an integer, and so is $f$ since $m_o$ is a multiple of $n_o$, while $|O|/m_o$ is quite obviously a power of $r$. Hence from the expression for $r^d$  in equation \eqref{e: v1} as half of a product of three positive integers, it follows that one of $q(q^2-1)/(4\,\overline{m}\, \overline{n})$ and $f$ is a power of $r$, while the other is twice a power of $r$.

Note here that if $m_o \ge 3$, then $m \ge 9$ and so
$f n_o = mn-2m-2n = (m-2)(n-2) - 4 \ge 7(3n_o-2)-4 > 2n_o$,
which gives $f > 2$ and therefore $f$ is divisible by $r$.
Hence in particular, if $m_o>n_o$, then $f$ is divisible by $r$, and then since also $r$ divides both $m_o$ and $m_o/n_o$, we find that $r$ divides~$\overline{n}$. By inspection, however, $r$ divides neither $\overline{m}$ nor $\overline{n}$ in any of the cases (a) to (d) of Lemmas ~\ref{l: psl restrictions 1} and ~\ref{l: psl restrictions 2}, and so $m_o = n_o$.
Accordingly, the expression for $f$ in equation \eqref{e: v1} reduces to $f=\overline{m}\,\overline{n}\,m_o -2(\overline{m} +\overline{n})$.

We now proceed to show that $m_o=1$, by assuming  that ($m_o =$) $n_o \ge 3$.  As shown above, this implies that  $f$ is divisible by $r$, and therefore so is $\overline{m} +\overline{n}$, and we can use that fact when checking the cases (a) to (d) in Lemmas ~\ref{l: psl restrictions 1} and ~\ref{l: psl restrictions 2}.

In case (a) of Lemma~\ref{l: psl restrictions 1}, where $\{\overline{m},\overline{n}\} = \{(q\pm1)/2, (q\mp1)/4\}$ or $\{\overline{m},\overline{n}\} = \{(q-1)/2,(q+1)/2\}$, we get an immediate contradiction since $r = p$ does not divide $\overline{m} +\overline{n}$.

In case (b), where $r=p$ and $\{\overline{m},\overline{n}\} = \{\frac{q-1}{2},\frac{q+1}{2}\}$ for some $q \ge 11$,
we have $q(q^2-1)/(4\,\overline{m}\, \overline{n}) = q$, which is odd, so $f/2$ must be a power of $r$,
with $f/2 + q = f/2 + (\overline{m}+\overline{n}) = \overline{m}\,\overline{n}\,m_o = m_o(q^2-1)/8$.
Then since $(q^2-1)/8$ is coprime to $p$ ($= r$), comparing the $r$-parts of $f/2 + q$ and $m_o(q^2-1)/8$, we find that $m_o = \min(f/2,q) = f/2$ or $q$, and then $q$ or $f/2$  is equal to $m_o(q^2-1)/8 - m_o = m_o(q^2-9)/8$.
With either possibility,  $(q^2-9)/8$ is a power of $r$, so $r$ divides $q^2-9$ and hence divides $9$,
and therefore $r = p = 3$.  But then $q \ge 27$ (since $q \ge 11$) and so $q^2-9 \equiv -9$ mod $27$
and hence $(q^2-9)/8$ cannot be a power of $3$, which is another contradiction.

In case (d) of Lemma~\ref{l: psl restrictions 2}, where either $\{\overline{m},\overline{n}\} = \{p, (p-1)/2\}$
with $r$ dividing $p+1$, or $\{\overline{m},\overline{n}\} = \{p, (p+1)/2\}$ with $r$ dividing $p-1$,
we find that either $r$ divides $p+(p-1)/2 = p+1 + (p- 3)/2$ and hence divides both $p+1$ and $p-3$,
or $r$ divides $p+(p+1)/2 = p-1 + (p+3)/2$ and hence divides both $p-1$ and $p+3$,
and both are impossible.

Finally, consider case (c) of Lemma~\ref{l: psl restrictions 2}, where $p = q > 5$, and
either $\{\overline{m},\overline{n}\} = \{p, (p-1)/4\}$ and $r$ divides $p+1$,
or $\{\overline{m},\overline{n}\} = \{p, (p+1)/4\}$ and $r$ divides $p-1$.
Here the same argument as in the previous paragraph shows that $r$ divides both $p+1$ and $p-5$,
or both $p-1$ and $p+5$, so $r = 3$.
Moreover, $q(q^2-1)/(4\, \overline{m}\, \overline{n}) = p(p^2-1)/(p(p\mp 1)) = p\pm 1$, which is even,
and so $f$ and $(p\pm 1)/2$ are powers of $r = 3$, with $p = 2{\cdot}3^{\delta} \mp 1 = 6t \mp 1$
where $t = 3^{\delta-1}$ for some integer $\delta \ge 1$ (but $p \ne 5$).

Now if $r$ divides $p+1$, then
$f = \overline{m}\,\overline{n}\,m_o -2(\overline{m} +\overline{n}) = \frac{p(p-1)m_o}{4} - \frac{5p-1}{2}
= \frac{(6t-1)(3t-1)m_o}{2} - (15t-3),$ and if $m_o \ge 9$ then this cannot be a power of $3$ because
$15t-3$ is not divisible by $9$, so $m_o = 3$.
It follows that $f =  \frac{3(6t-1)(3t-1)}{2} - (15t-3) = \frac{3(18t^2 - 9t +1 - (10t-2))}{2} = \frac{3(18t^2 - 19t + 3)}{2}$,
which is easily seen to be a power of $3$ only when $t = 1$, in which case $p = 6t-1 = 5$, a contradiction.

Similarly if  $r$ divides $p-1$, then
$f = \overline{m}\,\overline{n}\,m_o -2(\overline{m} +\overline{n}) = \frac{p(p+1)m_o}{4} - \frac{5p+1}{2}
= \frac{(6t+1)(3t+1)m_o}{2} - (15t+3),$ and if $m_o \ge 27$ then this cannot be a power of $3$ because
$15t+3$ is not divisible by $27$, so $m_o = 3$ or $9$.
It follows that either $f =  \frac{3(6t+1)(3t+1)}{2} - (15t+3) = \frac{3(18t^2 +9t +1 - (10t+2))}{2} = \frac{3(18t^2 - t + -1)}{2}$, which is never a power of $3$, or  $f =  \frac{9(6t+1)(3t+1)}{2} - (15t+3) = \frac{3(54^2 +27t +13- (10t+2))}{2} = \frac{3(54t^2 + 17t + 11)}{2}$, which again is never a power of $3$, a final contradiction.

Thus $(m_o,n_o) = (1,1)$, and so $(m,n) = (\overline{m},\overline{n})$.
\end{proof}

With $(m,n) = (\overline{m},\overline{n})$, possibilities become a lot easier to deal with.
For one thing, both $G$ and $\overline G = G/O$ are $(m,n,2)^*$-groups, and it follows from Euler's formula that
$$
\chi_G = |G|(1/(2m) - 1/4 + 1/(2n)) = |{\overline G}|\,|O|\,(1/(2m) - 1/4 + 1/(2n)) = |O|\,\chi_{G/O}.
$$
The latter observation leads to the following.

\begin{lem}\label{PSL lemma three}
If $(m,n) = (\overline{m},\overline{n})$, then
the triple $(q, \{m,n\}, \chi_{G/O})$ is one of the three possibilities appearing in Table {\em \ref{t:three}}.
\end{lem}

${}$\\[-36pt]
\begin{center}
\begin{table}[ht]
\begin{tabular}{|c|c|c|}
\hline
$q$ & $\{m,n\}$ & $\chi_{G/O}$ \\
\hline
$5$ & $\{5,5\}$ & $-3$ \\
$13$ & $\{3,7\}$ & $-13$ \\
$13$ & $\{3,13\}$ & $-49$ \\ \hline
\end{tabular}
\vspace{3mm}
\caption{Possibilities for triples $(q, \{m,n\}, \chi_{G/O})$ when $m_o=n_o=1$.}\label{t:three}
\end{table}
\end{center}

\vspace{-11mm}
\begin{proof}
Here $-\chi_{G/O}$ is a power of the odd prime $r$, but cannot be $1$, as the pair $(m,n)$ is hyperbolic.
For the same reason, if $q=5$ then it follows from Lemma \ref{l: not A5} and Lemma \ref{l: psl restrictions 2}
that $m=n=5$, in which case $-\chi_{G/O}=60(1/10-1/4+1/10) = -3$, giving the first row of Table \ref{t:three},
which is realised by the map N5.3 at \cite{C600}.
On the other hand, if $q\ne 5$, then Euler's formula gives
\begin{equation}\label{e: v2}
-\chi_{G/O} \, = \, \frac{1}{2} \cdot \frac{q(q^2-1)}{4mn}\cdot \left(mn-2m-2n\right),
\end{equation}
and the possibilities to consider are given by cases (a) to (d) in Lemmas \ref{l: psl restrictions 1} and \ref{l: psl restrictions 2}.
In each case $q(q^2-1)/(4mn)$ is an integer, and so $f=mn-2(m+n)$ must be one of $r^s$ or $2r^s$ for some nonnegative integer $s$, depending on whether $q(q^2-1)/(4mn)$ is even or odd.

In  cases (a) and (b) of Lemma \ref{l: psl restrictions 1}, we have $r=p$ and hence $q$ is a power of $r$, with $q \ge 11$. If $\{m,n\} = \{(q-1)/2,(q+1)/2\}$, then $q(q^2-1)/(4mn) = q$, which is odd, so $f = 2r^s$ for some $s$, giving $2r^s = f = mn-2(m+n) = (q^2-1)/4 -2q = (q^2-8q-1)/4$. Then since $q$ is a power
of $r$ it follows that $s = 0$, so $q^2-8q-1 = 8$ and hence $0 = q^2-8q-9 = (q-9)(q+1)$, which gives $q = 9$,
a contradiction.
Otherwise $\{m,n\} = \{(q\pm1)/2, (q\mp1)/4\}$ for $q\equiv \pm 1$ mod $4$,
with $q(q^2-1)/(4mn) = 2q$, which is even, so that $r^s = f = mn-2(m+n) = (q^2-12q\mp 4-1)/8$.

If $\{m,n\} = \{(q+1)/2, (q-1)/4\}$, then $8r^s =  q^2-12q-5$, and it follows that either $s=0$
and $0= q^2-12q-13 = (q+1)(q-13)$ so $q = 13$, giving the second row of Table \ref{t:three}, realised by the map N15.1 at \cite{C600}.
Otherwise $s > 0$ and then $p = r$ divides $(q^2-12q)-8r^s = 5$, so $r = p = 5$ but then
$q(q-12) = 5+8\cdot 5^s = 5(1+8\cdot 5^{s-1})$, which has no $5$-power solution for $q$.
Similarly, if $\{m,n\} = \{(q-1)/2, (q+1)/4\}$, then $8r^s =  q^2-12q+3$, and then either $s=0$
and $0= q^2-12q-5$, which has no solution for $q$, or otherwise $s > 0$ and then
$p = r$ divides $8r^s - (q^2-12q) = 3$, so $r = p = 3$ but then
$q(q-12) = 8\cdot 3^s - 3 = 3(8\cdot 3^{s-1} - 1)$, which has no $3$-power solution for $q$.

Next, we consider case (c) of Lemma \ref{l: psl restrictions 2}, in which $\{m,n\} = \{p, \frac{p\mp 1}{4}\}$ where $q = p\equiv \pm 1$ {\rm mod} $4$, and $p\ge 11$, and $p\pm 1 = 2r^\beta$ for some $\beta\ge 1$.

If $\{m,n\} = \{p, \frac{p- 1}{4}\}$, then equation \eqref{e: v2} gives $-\chi_{G/O} = (p+1)(p^2-11p+2)/8$, and as this is a proper power of $r$, while $p+1 \equiv 2$ mod $4$, it follows that $p^2-11p+2$ is a multiple of $4$ (but not~$8$). Now if $r$ does not divide $p^2-11p+2$, then $p^2-11p+2 = 4$, which has no solution for $p$. Hence $r$ divides $p^2-11p+2 = (p+1)(p-12) +14$, and so divides~$14$, which gives $r=7$. Then also because $(p+1) = 2 \cdot r^\beta -1 = 2 \cdot 7^\beta -1$, we have $p^2-11p+2 = 4 \cdot 7^{2\beta} - 26 \cdot 7^\beta + 14$, which cannot be $8$ times a power of $7$ if $\beta\ge 2$, and so $\beta = 1$, with $p^2-11p+2 = 196 - 182 + 14 = 28$.  This implies that $0 = p^2-11p-26 = (p+2)(p-13)$, and therefore $p = 13$, with $\{m,n\}=\{13,3\}$ and $-\chi_{G/O}=(14 \cdot 28)/8 = 49$, giving the third row of Table \ref{t:three}, realised by the map N51.1 at \cite{C600}.

On the other hand , if $\{m,n\} = \{p, \frac{p+1}{4}\}$, then equation \eqref{e: v2} gives $-\chi_{G/O} = (p-1)(p^2-9p-2)/8$, but $p^2-9p-2$ cannot be $4$,  so $r$ divides $p^2-9p-2 = (p-1)(p-8) - 10$ and hence $r = 5$. It follows that $p=2 \cdot 5^\beta+1$, so $p^2-9p-2 = 4 \cdot 5^{2\beta} - 14 \cdot 5^\beta - 10$, which cannot be $8$ times a power of $5$ if $\beta\ge 2$, so $\beta = 1$, with $p^2-9p-2 = 100 - 70 - 10 = 20$.  This implies that $0 = p^2-9p-22 = (p+2)(p-11)$, and therefore $p = 11$, with $\{m,n\}=\{11,3\}$ and $-\chi_{G/O}=(10 \cdot 20)/8 = 25$.  But there is no non-orientable regular map with characteristic $-25$ (see \cite{C600}), a contradiction.

Finally we consider case (d) of Lemma \ref{l: psl restrictions 2}, in which $\{m,n\} = \{p, \frac{p\mp 1}{2}\}$ where $q = p\ge 7$, and $(p\pm 1)/r^\beta$ is a power of $2$ for some $\beta\ge 1$, and show that no further possibility for $(m,n)$ arises.

If $\{m,n\} = \{p, \frac{p- 1}{2}\}$, then equation \eqref{e: v2} gives $-\chi_{G/O} = (p+1)(p^2-7p+2)/8$,
and as this is a proper power of $r$, it follows that one of $p+1$ and $p^2-7p+2$ is a multiple of $4$ (but not~$8$), while the other is a multiple of $2$ but not $4$. Now if $r$ does not divide $p^2-7p+2$, then $p^2-7p+2 = 2$ or $4$, which has only one solution for $p$, namely $p = 7$, but then $p+1$ is divisible by $8$.
Hence $r$ divides $p^2-7p+2 = (p+1)(p-8) +10$, and so divides $10$, which gives $r=5$.
Then because $(p+1)/5^\beta$ is a power of $2$, we have $p = 2^i \cdot 5^\beta -1$ for some $i \in \{1,2\}$, and therefore $p^2-7p+2 = 2^{2i}5^{2\beta}-9{\cdot}2^i5^\beta+10$, which cannot be $2$ or $4$ times a power of $5$ if $\beta\ge 2$, and so $\beta = 1$, and $p = 2 \cdot 5 -1 = 9$ or $4 \cdot 5 -1 = 19$. But clearly $p \ne 9$,
and also $p \ne 19$ for otherwise $p^2-7p+2 = 230$, which is not twice a power of $5$.

On the other hand , if $\{m,n\} = \{p, \frac{p+1}{2}\}$, then equation \eqref{e: v2} gives $-\chi_{G/O} = (p-1)(p^2-5p-2)/8$, and $p^2-5p-2$ cannot be $2$ or $4$, so $r$ divides $p^2-5p-2 = (p-1)(p-4)-6$ and hence $r = 3$.
Now it is easy to check (using residues mod $9$) that $p^2-5p-2$ is not divisible by $9$, and it follows that $p^2-5p-2\in \{2r, 4r\} = \{6,12\}$, and the only feasible solution is $p = 7$, giving $\{m,n\}=\{7,4\}$ and $-\chi_{G/O}=9$.  But there is no non-orientable regular map with these parameters (see \cite{C600}), another contradiction.
\end{proof}

We can now complete this section by proving Proposition~{\rm\ref{p: psl2 main}}.

\begin{proof}[Proof of Proposition~{\rm\ref{p: psl2 main}}]
Let $G$ be a $(2,m,n)^*$-group with the property that $G/O(G)\cong \PSL_2(q)$ for some prime-power $q\ge 5$,  such that $-\chi_G = r^d$ for some odd prime $r$. Then by Lemma~\ref{lem:r-gp}, the group $O=O(G)$ is trivial or   an $r$-group. The only possibilities for $(m,n)\ne (\overline{m},\overline{n})$ were considered in Lemmas \ref{l: q5} and \ref{l: a-f}, and that produced just one, namely $\{m,n\}=\{3,15\}$, which gives the entry $A2$ in Table~\ref{t: a}. On the other hand, if $(m,n) = (\overline{m},\overline{n})$ then by Lemma \ref{PSL lemma three} there are three possibilities for the triple $(q, \{m,n\}, \chi_{G/O})$, as listed in Table \ref{t:three}, and these give the entries A1, A4 and A3 in Table~\ref{t: a}. Moreover, those three entries are realised by the maps N5.3, N15.1 and N51.1 at \cite{C600}, while the entry A2 can be shown by computation using {\sc Magma} \cite{magma}
to be realisable by a map with characteristic $\chi_G =-3^7$ and $|O(G)|=3^6$.
\end{proof}
\medskip


\section{The case when \texorpdfstring{$G/O(G)\cong \PGL_2(q)$}{G/O is PGL(2,q)}: proof of Theorem \ref{t: main} for family B}\label{s: insol-PGL}
\bigskip


Our aim in this section is to prove part (B) of Theorem~\ref{t: main}, covering the possibilities summed up in Table~\ref{t: b}. This will again be based on a series of auxiliary findings, and our working hypothesis throughout this section will be the following.
\medskip

\noindent {\bf Hypothesis 2:} {\sl The group $G= \langle a,b,c\rangle$ with presentation of the form \eqref{eq:abc} is a $(2,m,n)^*$-group with Euler characteristic $\chi_G=-r^d$ for some odd prime $r$ and positive integer $d$, such that $G/O(G)$ is isomorphic to $\PGL_2(q)$ for some odd prime-power $q=p^e\ge 5$.}
\medskip

We start by deriving a useful consequence of Theorem~\ref{t: odd order} (which was proved in Section \ref{s: chi odd} using Proposition \ref{p: odd order}), under Hypothesis 2.

\begin{prop}\label{p: pgl2 main}
The group $G$ has the following subgroups:
\begin{enumerate}
\item a normal $r$-subgroup $N$ such that $G/N\cong (K\times C_\ell).2$ for some odd positive integer $\ell$ coprime to both $q(q^2-1)$ and $r$, where $K$ is the unique copy of \,$\PSL_2(q)$ in $G/N$,
\item a subgroup $G_o$ of index $2$ containing $N$ such that $G_o/N\cong K \times C_\ell$,
\item a cyclic subgroup $C$ of order $\ell$ with the property that $O(G)=NC\cong N\rtimes C_\ell$, and
\item a normal subgroup $\wt{K}$ such that $G_o=\wt{K}C$, $K\cong \wt{K}/N$, $K\times C_\ell \cong G_o/N$ and $\wt{K}\cap O(G)=N$.
\end{enumerate}
\end{prop}

\begin{proof}
Euler's formula \eqref{e: sunny} with $\chi_G = -r^d$ implies that the only odd prime dividing $|G|/\lcm(m,n)$
is $r$, and it follows that the prime factorisation of $|G|/\lcm(m,n)$ has the form $2^i r^j$ for some $i,j\ge 0$.
Accordingly, if we take $N\lhd G$ as the subgroup given by Theorem \ref{t: odd order}, then this is an $r$-group.

Next, as $G/O\cong \PGL_2(q)$, either case (c) or (d) of Theorem~\ref{t: odd order} applies. If case (c) holds, then the conclusion that $G/N\cong (\PSL_2(q) \times C_\ell).2$ for  some odd $\ell$ is immediate, and so we may suppose that (d) holds, namely that $G/N$ is isomorphic to $3.(A_6\times C).2$ where $C$ is a cyclic group of odd order. Now because $A_6\cong \PSL_2(9)$ has non-cyclic Sylow $3$-subgroups, Lemma~\ref{l: sylow} tells us that $r=3$, and then since $N$ is a subgroup of index $3$ in a normal subgroup $N_*$ of $G$ for which $G/N_*\cong(\PSL(2,9)\times C_\ell).2$, 
we can take this subgroup $N_*$ in place of $N$ to obtain the desired conclusion in case (d).
Also with $G/N\cong (\PSL_2(q) \times C_\ell).2$, the correspondence theorem in group theory implies that
$G$ contains a (normal) subgroup $G_o$ of index $2$ containing $N$ such that $G_o/N\cong \PSL_2(q) \times C_\ell$, and then the form of $G/N$ together with the assumption $-\chi_G=r^d$ and the conclusion of
Lemma~\ref{l: sylow} imply that $\gcd(q(q^2-1),\ell)=1$.

Now suppose that $r$ divides $\ell$. Then $G_o/N$ contains a cyclic subgroup $C^*$ of order $r$,
characteristic in $G_o/N$ and hence normal in $G/N$, and by the correspondence theorem again,
this implies that $G$ has a normal $r$-subgroup $N^*$ containing $N$ such that $N^*/N\cong C^*$.
But then $G/N^*\cong (\PSL_2(q) \times C_{\ell/r}).2$, and hence we can take $N^*$ in place of $N$,
and suppose that $r$ does not divide $\ell$.

The arguments made so far prove parts (a) and (b) of our Proposition. To go further, observe that the chain $1\lhd C_\ell \lhd (\PSL_2(q)\times C_\ell)$ in $G/N$ lifts to the chain $N\lhd L\lhd G_o$, with $G_o$ as above
and $L$ such that $L/N\cong C_\ell$. But $\gcd(|N|,\ell)=1$ (since $\gcd(r,\ell) =1$), so the Schur-Zassenhaus theorem implies that $L$ is a split extension of $N$ by a subgroup $C\cong C_\ell$, so that $L=NC$, and then
since $NC\lhd O=O(G)$, comparing orders we find that $L=O=NC\cong N\rtimes C_\ell$, which proves (c).

Finally, note for (d) that the chain of normal subgroup $1\lhd K\lhd K\times C_\ell \lhd (K\times C_\ell)
{\cdot}2$ corresponds to the chain of the form $N\lhd \wt{K}\lhd G_o\lhd G$, where $K\cong \wt{K}/N$ and
$K\times C_\ell \cong G_o/N$, with $G_o$ as in the proof of Proposition \ref{p: pgl2 main}. Also $G_o=\wt{K}C$,  and as $N\le \wt{K}\cap O$ and $\gcd(\ell,|K|) =1$, we conclude that $\wt{K}\cap O=N$.
\end{proof}

Completing the proof of the part (B) of Theorem~\ref{t: main} now relies on establishing existence
of the parameters given in Table~\ref{t: b} and the properties they need to satisfy.
This requires quite some amount of preparation, which we will do in another series of auxiliary findings.
Recalling that our $(2,m,n)^*$ group $G= \langle a,b,c\rangle$ represents a non-orientable regular map, with
$\langle a,b,c\rangle = \langle ab,bc\rangle$, we let $g=ab$ and $h=bc$ be the canonical elements of orders $m$ and $n$ in $G$. 
The primes $p$ and $r$, the prime-power $q = p^e$ and the integers $d$ and $\ell$ are as in Hypothesis 2 and
Proposition~\ref{p: pgl2 main}.

\begin{lem}\label{l: dihedral}
If $\ell>1$, then the quotient $G/\wt{K}$ is isomorphic to a dihedral group of order $2\ell$. Moreover, one of the elements $g\wt{K}$ and $h\wt{K}$ generates the cyclic part of $G/\wt{K}$ (of odd order $\ell$) and is contained in the product $\wt{K}C=G_o$, while the other one lies outside $G_o$.
\end{lem}

\begin{proof}
Obviously $G/\wt{K} = \langle a\wt{K},b\wt{K},c\wt{K}\rangle = \langle g\wt{K},h\wt{K}\rangle$ is isomorphic to a split extension $C_\ell\rtimes C_2$ (as $\ell$ is odd). Then exactly one of the three elements $a\wt{K},b\wt{K}$ and $c\wt{K}$ (of order dividing $2$) is trivial, for otherwise either $g\wt{K} = a\wt{K}{\cdot}b\wt{K}$ and $h\wt{K} = b\wt{K} {\cdot} c\wt{K}$ both lie in the index $2$ subgroup $G_o/\wt{K}$ of order~$\ell$, or $g\wt{K}$ and $h\wt{K}$ generate a subgroup of order at most $2$, both of which are impossible as  they generate $G/\wt{K}$.  Similarly $b\wt{K}$ is non-trivial, for otherwise $g\wt{K} = a\wt{K}$ and $h\wt{K} = c\wt{K}$ generate a subgroup of order dividing $4$. Hence exactly one of $a\wt{K}$ and $c\wt{K}$ is trivial, and then $G/\wt{K}=\langle a\wt{K}, b\wt{K},c\wt{K}\rangle$ is dihedral, generated by two involutions, namely $b\wt{K}$ and either $a\wt{K}$ or $c\wt{K}$. Moreover, the order of the product of those two involutions must be $\ell$, so one of $g\wt{K}$ and $h\wt{K}$ generates $G_o = \wt{K}C$, while the other one is an involution lying outside $G_o$.
\end{proof}

Topologically, the natural projection $G\to G/\widetilde{K}\cong D_{\ell}$ from $G$ to the dihedral group of order $2\ell$ from Lemma \ref{l: dihedral} induces a {\em folded} branched covering $\theta:\ {\mathcal S}\to {\mathcal D}$ from the non-orientable carrier surface ${\mathcal S}$ of the regular map $M=(G;a,b,c)$ determined by the $(2,m,n)^*$-group $G$ to a {\em closed disc} ${\mathcal D}$. In particular, if $c\in \widetilde{K}$, then the $\theta$-image of the map $M$ is simply a map on ${\mathcal D}$ with a single face bounded by an $\ell$-sided polygon which, at the same time, forms the boundary of ${\mathcal D}$. Cellular maps on a disc were called {\em degenerate} in \cite{LiS}, and the situation just described corresponds to the degenerate map no. 6 in Proposition 1 of \cite{LiS}.

In general, the degenerate maps of \cite{LiS} arise when at least one of the canonical generators $a,b,c$ for a regular map $M=(G;a,b,c)$ is the identity element of $G$, and Proposition 1 of \cite{LiS} handles the resulting seven types of degeneracies. For completeness, we note that while branched coverings without folds preserve non-orientability of surfaces, this is not also true for (branched) coverings with folds, which is the case for the covering $\theta$ appearing in Lemma \ref{l: dihedral}.

\smallskip

Now continuing, as in the previous Section we use $\overline{m}$ and $\overline{n}$ to denote the orders of the elements $gO = abO$ and $hO = bcO$ in the group $\overline{G}=G/O$, but in our analysis we also need to determine the orders $m^*$ and $n^*$ of $gN$ and $hN$ in $G/N$, which we do next.

\begin{lem}\label{l:G?N}
Suppose $\ell > 1$.  If $g\in \wt{K}C$ while $h\notin \wt{K}C$, then $m^*=\ord(gN) =\ell \, \overline{m}$ while $n^*=\ord(hN) =\overline{n}$, and this happens when $\overline{n}$ divides $q\pm 1$ but not $(q\pm 1)/2$. An analogous assertion (interchanging the triples $(g,\overline{m},m^*)$ and $(h,\overline{n},n^*)$) is valid if $h\in \wt{K}C$ while $g\notin \wt{K}C$.
\end{lem}

\begin{proof}
By Lemma \ref{l: dihedral} we may assume without loss of generality that $g\in \wt{K}C$ and  $h\notin \wt{K}C$, so that the orders of $g\wt{K}$ and $h\wt{K}$ in $G/\wt{K}$ are $\ell$ and $2$, respectively. Now $m^*$ is the smallest positive integer $k$ for which $g^k\in N=\wt{K}\cap O$, so $m^*= \lcm(\ell, \overline{m})$. But $\overline{m}$, being the order of an element in $G/O\cong \PGL_2(q)$, is either $p$ or a divisor of $q\pm 1$, and hence coprime to $\ell$, so $m^*=\ell \, \overline{m}$. On the other hand, $1 \ne h \not\in \wt{K}C$ but $h^2\in \wt{K}C=G_o$, and so $\overline{n}$ must be even. Then the same argument as in the previous sentence but applied to the pair $(\overline{n},2)$ instead of $(\overline{m},\ell)$ gives $n^*= \lcm(2, \overline{n}) =\overline{n}$. Finally, note that if $\overline{n}$ divides $q\pm 1$ but not $(q\pm 1)/2$, then $hN$ cannot be contained in $\wt{K}C/N\cong K\times C_\ell$, since $\ell$ is coprime to $q\pm 1$ and $K$ contains no element of order that divides $q\pm 1$ but not $(q\pm 1)/2$, and therefore $h\notin \wt{K}C$, which implies that $n^*=\overline{n}$.
\end{proof}

As an aside, the fact that $O\cong N\rtimes C$ from Proposition \ref{p: pgl2 main} implies that $\overline{m}$ and $m^*$ have the same parity, and if they are both odd, then since $N \le \wt{K}$, the order of $g\wt{K}$ in $G/\wt{K}$ must also be odd, and hence equal to $\ell$. A similar assertion applies if $m$ is replaced with $n$.

Our aim in the next two lemmas is to determine possibilities for $\overline{m}$ and $\overline{n}$ in the situation where $G/O\cong \PGL_2(q)$.

\begin{lem}\label{l: divides one}
The prime $r$ divides $p$, $q-1$ or $q+1$.
\end{lem}

\begin{proof}
The proof is similar to the one for Lemma \ref{l: not A5}. Here $|\overline{G}| = q(q^2-1)$, and again if this is divisible by an odd prime $t$ not dividing $\overline{m}\overline{n}$, then $t=r$ by Lemma~\ref{l: odd prime}, so $r$ divides $q(q^2-1)$ and the conclusion follows easily. So suppose that no such $t$ exists, in which case every odd prime divisor $t$ of $|G|$ divides $\overline{m}$ or $\overline{n}$.  Then we may suppose without loss of generality that $p$ divides $\overline{m}$, and hence $\overline{m}=p$, while every other odd prime divisor of $|\overline{G}|$ divides $\overline{n}$ and hence $\overline{n}$ divides $q\pm 1$. But also $|\overline{G}|$ is divisible by $8$, with one of $q+1$ and $q-1$ divisible by $4$ and the other by $2$, so by Lemma~\ref{l: 2} it follows that $\overline{n}$ is even.  Moreover, as every odd prime divisor of $|\overline{G}|$ other than $p$ divides $\overline{n}$ (which divides $q\pm 1$), it follows that the remaining factor $q\mp 1$ of $|\overline{G}|$ cannot have any odd prime divisors and so must be a power of $2$, and is at least $4$ (since $q \ge 5$), but then $|\overline{n}|_2=2$, so $|G|_2 \ge 8 > 4 = 2|\overline{m}|_2|\overline{n}|_2$, which contradicts Lemma~\ref{l: 2}.
\end{proof}

\begin{lem}\label{l: pgl restr}
Under Hypothesis $2$, one of the following possibilities holds$\,:$
\begin{enumerate}
\item[{\rm \,\,(i)}] $r=p$ and $\{\overline{m},\overline{n}\} = \{(q\pm 1)/2,q\mp 1\}\,;$
\item[{\rm \,(ii)}]  $r=p$  and $\{\overline{m},\overline{n}\} = \{q-1,q+1\}\,;$
\item[{\rm (iii)}] $q=p\ge 5$, $\,p\pm 1 = 2^ir^\beta$ for some $i\ge 1$ and $\beta\ge 1$, and $\{\ovl{m}, \ovl{n}\} = \{p,p\mp 1\}\,;$
\item[{\rm \, (iv)}]  $q=p\ge 5$, $\,|p\pm 1|_r= r^\beta$ for some $\beta\ge 1$, and $\{\ovl{m},\ovl{n}\} =\{p, (p\pm 1)/r^\alpha\}$ for some $\alpha \le \beta$, with $p\equiv \pm 1$ mod $4$ and $p\mp 1$ a power of $2$\,.

\end{enumerate}
\end{lem}

\begin{proof}
First suppose $r = p$.  If there exists an odd prime $t$ such that $|\overline{m}|_t\, |\overline{n}|_t < |q^2-1|_t$ ($\le |\overline{G}|_t$), then Lemma~\ref{l: odd prime} implies that $t$ divides $\chi_G$ and so $t=r = p$, which gives a contradiction since $|q^2-1|_p = 1$.  Hence in this case $|\overline{m}|_t\, |\overline{n}|_t \ge |q^2-1|_t$ for every odd prime $t$. Also $|\overline{G}|_2 = |q(q^2 - 1)|_2 \ge 8$, so by Lemma~\ref{l: 2} we have $2|\overline{m}|_2\, |\overline{n}|_2 \ge |q(q^2-1)|_2 = |q^2-1|_2$, and hence $\overline{m}\,\overline{n}$ is divisible by $(q^2-1)/2$.  Then since the order of every element of $\overline{G}$ is either $p$ or a divisor of $q\pm 1$, and the odd parts of $q+1$ and $q-1$ are coprime, we obtain the possibilities for $\{\overline{m},\overline{n}\}$ in the conclusions (i) and~(ii).

On the other hand, suppose $r \ne p$. Then $r$ divides $q \pm 1$ by Lemma~\ref{l: divides one},
and by Lemma~\ref{l: sylow}, every Sylow $p$-subgroup of $G$ is cyclic, so $q = p$ (for otherwise a
Sylow $p$-subgroup of $\overline{G}$ is elementary abelian of rank $> 1$). Also applying Lemma~\ref{l: odd prime} to $t= p\ne r$ gives $p \le |G|_p\le |\overline{m}|_p|\overline{n}|_p$ and so at least one of $\overline{m}$ and $\overline{n}$ must be divisible by $p$ and hence equal to $p$.  Without loss of generality we may suppose that $\overline{m} = p$, and then re-applying Lemma~\ref{l: odd prime} to any odd prime divisor of $p\pm 1$ other than $r$ gives $t \le |G|_t \le |\overline{m}|_t|\overline{n}|_t$ and hence $t$ divides $\overline{n}$, and $\overline{n}$ divides one of $p+1$ and $p-1$.

We continue by considering the case where $r$ divides $p-1$ (and leave the alternative that $r$ divides $p+1$ until the end of the proof).  To do this, let $\beta$ be the positive integer for which $|p- 1|_r=r^\beta$.

Suppose first that $\overline{n} \mid p+1$. Then $|p+1|_t=|\overline{n}|_t$ for every odd prime $t$, by earlier comments, but also $\overline{n}$ cannot divide $(p+1)/2$, because $\PGL_2(p)$ cannot be generated by two elements whose orders are  $p$ and a divisor of $(p+1)/2$; see \cite{cps2}. It follows that $|\overline{n}|_2 = |p+1|_2$, and so $\overline{n}=p+1$. This leaves $r$ and $2$ as the only prime divisors of $p-1$, and hence $p-1 = 2^ir^\beta$ for some $i\ge 1$.

On the other hand, suppose $\overline{n}$ divides $p-1$. Then $|p-1|_t=|\overline{n}|_t$ for every odd prime $t\ne r$, and so $(p-1)/\overline{n} = 2^ir^\alpha$ for some non-negative integer $\alpha\le \beta$ and some $i\ge 0$. But if $i\ge 1$, then $\overline{n}$ would divide $(p-1)/2$, and then $\PGL_2(p)$  would be generated by two elements whose orders are $p$ and a divisor of $(p-1)/2$, a similar contradiction.  Thus $\ovl{n}=(p-1)/r^\alpha$ for some $\alpha\le\beta$. It now follows that $p+1$ has no odd prime divisor, so $p+1$ is a power of $2$, and as $p\ge 5$ we find that $p+1=2^j$ for some $j\ge 3$, which also implies that $p\equiv -1$ mod $4$.

This proves the conclusions (iii) and (iv) in the situation where `lower signs' apply, that is, in case (iii) when $p- 1= 2^ir^\beta$ with $i, \beta\ge 1$ and $\overline{n}= p+1$, and case (iv) when $|p-1|_r =r^\beta$, $\beta\ge 1$, and $\overline{n}= (p-1)/r^\alpha$, $\alpha \le \beta$, with $p\equiv - 1$ mod $4$ and $p + 1=2^j$ for some $j\ge 3$.

The same arguments apply when $r$ divides $p+1$, with lower signs changed consistently to `upper signs',
except that in case (iv) we have $p-1=2^j$ for some $j\ge 2$. This completes the proof.
\end{proof}

The preceding lemmas enable us to determine the orders $m^*$ and $n^*$ of the elements $gN$ and $hN$ in
$G/N\cong (K\times C_\ell){\cdot}2$.

\begin{lem}\label{l: cases}
Depending on $r$, one of the following holds$\,:$ 
\begin{enumerate}
\item $r=p$ and $\{m^*,n^*\} = \{\ell (q+1)/2\, ,\, q-1\}\,;$
\item $r=p$ and $\{m^*,n^*\} = \{\ell (q-1)/2\, ,\,  q+1\}\,;$
\item $r=p$, $\,\ell=1$ and $\{m^*,n^*\} = \{q\pm 1\, ,\,  q\mp 1\}\,;$
\item $q=p\ge 7$, $\,p-1 = 2^ir^\beta$ for some $i,\beta\ge 1$, and $\{m^*,n^*\} = \{\ell p\, ,\, p+1\}\,;$
\item $q=p\ge 5$, $\,p+1 = 2^ir^\beta$ for some $i,\beta\ge 1$, and $\{m^*,n^*\} = \{\ell p\, ,\, p-1\}\,.$
\end{enumerate}
\end{lem}

\begin{proof}
We consider the four possibilities (i) to (iv) given by Lemma~\ref{l: pgl restr} for $\{\overline{m},\overline{n}\}$, and draw conclusions for $\{m^*,n^*\}$.

First, in cases (i) and (iii) we may assume without loss of generality that $\overline{n}=q\pm 1$, with $q=p$ in case (iii). Then by Lemma \ref{l:G?N}  we find that $n^*=\overline{n}$, and consequently $m^* = \ell\,\overline{m}$. This gives the four possibilities for $\{m^*,n^*\}$ in parts (a), (b), (d) and (e), after applying the upper and lower signs separately.  Similarly, in case (ii), Lemma \ref{l:G?N} gives $\ell=1$ and part (c).

Lemma \ref{l:G?N} can also be used in case (iv) of Lemma~\ref{l: pgl restr}, because in this case  $\overline{n}$ does not divide $(p\pm 1)/2$, so that again $n^*=\overline{n}$ and hence $\{m^*,n^*\} = \{\ell p\, ,\, (p\pm 1)/ r^\alpha\}$, with $p\mp 1$ a power of $2$. Nevertheless, we show that this case is impossible.

Specifically, as $m^*\mid m$ and $n^*\mid n$ we have $m=m^*m_1$ and $n=n^*n_1$ for some powers $m_1$ and $n_1$ of $r$. Then using Euler's formula \eqref{e: sunny} with $|G|=p(p^2-1)\ell|N|$, we obtain
\[ r^d=  \frac{p\mp 1}{2} \, \frac{|N|}{\lcm(m_1, n_1)}\, f(p) \ \ \ {\rm where} \ \ \
f(p) = \frac{1}{\gcd(m_1,n_1)}\left( \ell m_1n_1p\,\frac{p\pm 1}{2r^\alpha} - \ell m_1p - n_1\frac{p\pm 1} {2r^\alpha} \right) \ .\]
Now in part (iv) of Lemma \ref{l: pgl restr} we know that $p\equiv \pm 1$ mod $4$.  It follows that $(p{\pm }1)/2$ is odd, so $f(p)$ is odd, and then from the equation for $r^d$ just displayed, $(p\mp 1)/2$ must be odd as well.  But also we know from part (iv) of Lemma \ref{l: pgl restr} that $p\mp 1$ a power of $2$, and as the only odd power of $2$ is $1$, we conclude that $(p\mp 1)/2 = 1$, and hence $p\mp 1=2$, which contradicts the fact that $p \ge 5$.
\end{proof}

Next, as in the previous section, we let $m_o=m/\overline{m}$ and $n_o=n/\overline{n}$.  By Hypothesis $2$, each of these is either $1$ or a non-trivial divisor of $|O|=|NC|=|N|\ell$, where $\ell\ge 1$ is odd and coprime to $r$, and $|N|$ is a power of $r$. We also recall the Fitting subgroup $F(G)$ and the generalised Fitting subgroup $F^*(G)$, briefly introduced after Lemma \ref{l: vector block} in Section \ref{s: chi odd}. The relation of $F(G)$ and $F^*(G)$ to $N$ and the values of $m_o$ and $n_o$ is as follows.

\begin{lem}\label{l: fstar nilpotent}
If $|N|=1$, then $F^*(G)\ne F(G)$. For a partial converse, if $F^*(G)\ne F(G)$, then at least one of $m_o$ and $n_o$ must be $1$, and in the case $\{m_o,n_o\}=\{\ell,1\}$, also $|N|=1$.
\end{lem}

\begin{proof}
First, if $|N|=1$, then $G\cong (\PSL_2(q)\times C_\ell).2$ by part (a) of Proposition \ref{p: pgl2 main}, so the layer $E(G)$ of $F^*(G)$ contains the subgroup $K= \PSL_2(q)$ and hence $F^*(G)\ne F(G)$.

For the partial converse, suppose that $F^*(G)\ne F(G)$, so $F^*(G)=E(G)F(G)$ for a non-trivial layer $E(G)$. Then Hypothesis $2$ implies that $\PSL_2(q)$ is the unique non-abelian simple composition factor of $G$.  Information about the layer was summarised in the first part of the proof of Claim~1 in the proof of Proposition \ref{l: odd order ind}, and implies that if $E(G)\ne 1$, then $E(G)$ contains an isomorphic copy $K$ of $\PSL_2(q)$ as a characteristic subgroup, so that $K\lhd\, G$. It follows that $K\cap O=1$ (where $O=O(G)$), and so $KO$ is a normal subgroup of index $2$ in $G$ (with $|G/O|=2|K|$). Moreover, as $G/KO \cong (G/K)/(KO/K)$ and $KO/K\cong O$ it follows that $G/K$ is isomorphic to an extension of $O$ by $C_2\cong G/KO$, and this extension splits by the Schur-Zassenhaus theorem, so $G/K\cong O\rtimes C_2$.
We may also assume that $O\ne 1$, for otherwise the conclusion is trivial.

Next, recall that $G=\langle a,b,c \rangle$ is a group with presentation \eqref{eq:abc} that is also generated by the two elements $g=ab$ and $h=bc$. Hence $G/K\cong O\rtimes C_2$ is generated by the three involutions $aK$, $bK$ and $cK$, as well as by just $gK$ and $hK$. But now the same argument as used in the proof of Lemma \ref{l: dihedral} (with $\widetilde{K}$ replaced by $K$) shows that $gK$ and $hK$ generate the semidirect product $O\rtimes C_2$ with $O\ne 1$ only if exactly one of $aK$ and $cK$ is equal to $K$ while $bK\ne K$. In particular, it follows that $O\rtimes C_2$ is a dihedral group.  (Note here that since the order of $G/K$ is even but not divisible by $4$, the same conclusion follows also from \cite{LiS}, as the natural projection $G\to G/K$ induces a folded covering of the map $M(G;a,b,c)$ onto a map on a disc, as in the proof of Lemma \ref{l: dihedral}.)

Suppose without loss of generality that $cK=K$, so that $hK=bcK=bK$ has order $2$ in $G/K$. This implies that $h\notin KO$ but $h^2\in K$, and so the order $n$ of $h$ must be even, hence also $\overline{n}$ is even. Then since $K\cap O=1$, we find $n=\lcm(2,\overline{n}) = \overline{n}$ and so $n_o=1$. By symmetry of this argument, we conclude that one of $m_o$ and $n_o$ is equal to $1$.

Finally, suppose that $\{m_o,n_o\}=\{\ell,1\}$. Then since $G/K \cong O\rtimes C_2$ is dihedral, it must be generated by $\tilde{x}=abK$ and $\tilde{y}=bcK=bK$, of orders $\ell$ and $2$ in some order, with one of them generating the index $2$ cyclic subgroup of $O\rtimes C_2$, of order $\ell$. As $O\cong N\rtimes C_\ell$ by part (c) of Proposition \ref{p: pgl2 main}, it follows that $N=1$, completing the proof.
\end{proof}
\medskip

For the case where $F^*(G)=F(G)$, we will need information on centralisers of the direct factors of $F(G)$.
We recall that the parameter $q=p^e$ appears in $\PSL(2,q)\cong K=G_o/O$.
For our next lemma, let us also recall the function $\veps$ introduced immediately before the statement of Lemma \ref{l: covers of K}.
\medskip

\begin{lem}\label{l: fstar lower bound}
If $F^*(G)=F(G)$, then the centraliser $C_G(N)$ is soluble, and $|N|\ge r^{\veps(q,r)+1}$.
\end{lem}

\begin{proof}
First, solubility of the group $C_G(N)$ follows from Claim 3 of the proof of Proposition \ref{l: odd order ind}.
Next, as in the beginning of the proof of Claim 2 of Proposition \ref{l: odd order ind}, the group $G/C_G(N)$ is isomorphic to a subgroup of $\Aut(N)$, and the fact that $C_G(N)$ is soluble then implies that $K$ is isomorphic to a section of $\Aut(N)$. Automorphisms of $N$ preserve the Frattini subgroup $\Phi(N)$ of $N$ and induce an epimorphism $\Aut(N) \to \Aut(N/\Phi(N))$ with soluble kernel (by Burnside's theorem \cite[Theorem 1.4, p.174]{gor}), and it follows that $\Aut(N/\Phi(N))$ has a section isomorphic to $K$. But $N/\Phi(N)$ is an elementary abelian $r$-group of order (say) $r^k$, so $\Aut(N/\Phi(N))\cong \GL_k(r)$. Then since $K$ is still a section of the latter, Lemma~\ref{l: covers of K} implies that $k\geq\veps(q,r)$.  This also means that $\Phi(N)$ is a proper subgroup of $N$, and hence $|N|\ge r^{\veps(q,r)+1}$.
\end{proof}
\medskip

We now continue by successively examining the five cases (a) to (e) of Lemma \ref{l: cases}.
In doing so, we recall that $m^*$ and $n^*$ are the orders of $gN$ and $hN$ in $G/N$, respectively,
and we let $m_1=m/m^*$ and $n_1=n/n^*$, and observe that each of these integers is a power of $r$.
\medskip

\begin{lem}\label{l: mn3}
If $r=p$ and $\{m^*,n^*\} = \{\ell (q+1)/2,\, q-1\}$, with $q\equiv +1$ {\rm mod} $4$, as in
case {\rm (a)} of {\rm Lemma \ref{l: cases}}, then $m_1=n_1$ and $q=9$, giving row {\rm B4} of {\rm Table~\ref{t: b}}.
\end{lem}

\begin{proof}
In this situation $|G|=|N|\,\ell \, q(q^2-1)$, and up to duality, $m=\ell m_1(q+1)/2$ and $n=n_1(q-1)$. Substituting these values into Euler's formula \eqref{e: sunny} with $-\chi_G=r^d$ gives
\[ r^d = \frac{|N|q}{2\,\lcm(m_1,n_1)} \, f_{(a)}(q) \ \ {\rm where} \ \ f_{(a)}(q) = \frac{1}{\gcd(m_1,n_1)}\left( \frac{\ell(q^2{-}1)}{2}m_1n_1 - \ell(q{+}1)m_1 - 2(q{-}1)n_1\right). \]

With the aim of reaching a contradiction, suppose that $m_1\ne n_1$.  Then since both are powers of $r$,
exactly one of $m_1/\gcd(m_1,n_1)$ and $n_1/\gcd(m_1,n_1)$ is equal to $1$, and the other is a proper power of $r$. It follows that two of the three terms in the bracketed expression in the formula for $f_{(a)}(q)$ are divisible by $r\,\gcd(m_1,n_1)$, while the third one is not, and therefore $f_{(a)}(q)$ is an even positive integer that is not a multiple of $r$, and as also $\lcm(m_1,n_1)$ is a power of $r$, it must divide $|N|q$, and so $f_{(a)}(q) = 2$.

Now if $m_1>n_1$, then $\gcd(m_1,n_1)=n_1$ and the above formula for $f_{(a)}(q)$ simplifies to
$2 = f_{(a)}(q) = m_1\ell(q^2{-}1)/2 - \ell(q{+}1)m_1/n_1 - 2(q{-}1)$. Then using the facts that $m_1/n_1 \le m_1$
and $\ell\ge 1$ and $m_1\ge r \ge 3$, we find that
\begin{align*}
f_{(a)}(q) & = (\ell/2) \,\left( m_1(q^2{-}1) - 2(q{+}1)(m_1/n_1)- 4(q{-}1)\ell \right) \\
 & \ge (\ell/2) \, \left( m_1(q^2-1) - 2m_1(q+1) - 4(q-1)\right) \\
 & = (\ell/2) \, \left(  m_1(q^2-2q-3) - 4(q-1) \right) \\
 & \ge (1/2)\left( 3(q^2-2q-3) - 4(q-1)\right)   = (3q^2-10q-5)/2,
\end{align*}
which gives a contradiction since $3q^2-10q-5 \ge 40$ (given that $q \ge 5$).

Similarly, if $m_1<n_1$ we have $2 = f_{(a)}(q) \ge (1/2)(n_1(q^2 -4q +3) -2(q+1)) \ge (3q^2-14q+7)/2$,
which gives a contradiction since $3q^2-14q+7 \ge 12$.

Thus $m_1=n_1$.
It now follows that $f_{(a)}(q) = \ell \, m_1 (q^2-1)/2 - \ell(q+1) - 2(q-1)$, and this must be twice a power of $r$ ($= p)$, say $ f_{(a)}(q)=2p^j$ for some non-negative integer $j$, and then because $\lcm(m_1,n_1)$ divides $|N|$,
and $q = p^e = r^e$, we have $r^d \ge q\, f_{(a)}(q)/2 \ge r^{e} \,f_a(q)/2$, so $j \le d-e$.

Now from $2 p^j =  f_{(a)}(q) = \ell \, m_1 (q{+}1)(q{-}1)/2 - \ell(q{+}1) - 2(q{-}1)$,  we find that $p^{e}+1 = q+1$ divides $2p^j+2(q{-}1)$ and hence divides $2p^j-4$, with $e \le j$ (since $p^{e}+1 \le 2p^{j}-4$). Dividing $j$ by $e$ with remainder $t < e$ implies that $p^e+1$ divides into $2p^{j}-4$ with remainder $2p^{t} \pm 4$,
and hence $2p^{t} \pm 4$ is divisible by $p^e+1$.  To achieve this, the only possibilities are $(p,e,t) = (3,2,1)$
and $(5,1,0)$, giving $q = p^e = 5$ or $9$.
But if $q=5$ then quotient $G/N=(\PSL_2(5)\times C_\ell).2$ would be a $(2, 3\ell, 4)$-group for some $\ell$
coprime to $q(q^2-1) = 120$, and then factoring out the cyclic normal subgroup of order $\ell$ would make $\PGL_2(5)$ a $(2,3,4)$-group, which is impossible. Thus $q=9$, with $G/N=(\PSL_2(9)\times C_\ell).2$ being a $(2, 5\ell, 8)$-group, and the rest follows easily.
\end{proof}
\medskip

Here we note that the conclusion of Lemma \ref{l: mn3} (namely $p=r=3$, $q=9$ and $m_1=3^k$ for some $k\ge 0$) implies that $2{\cdot}3^j = f_{(a)}(9) = (40{\cdot}3^k-10)\ell-16$, from which we find $\ell = (3^j +8)/(20{\cdot}3^k-5)$. The requirement that this is a positive integer coprime to $q(q^2-1) = 720 = |\PGL_2(9)|$) puts severe restrictions on the possible values of $j$ and $k$.
\medskip

\begin{lem}\label{l: mn4}
If $r=p$ and $\{m^*,n^*\} = \{\ell (q-1)/2,\, q+1\}$, with $q\equiv -1$ {\rm mod} $4$, as in
case {\rm (b)} of {\rm Lemma \ref{l: cases}}, then $m_1=n_1$ and $r=p=q=7$, giving row {\rm B3} of {\rm
Table~\ref{t: b}}.
\end{lem}

\begin{proof}
In this case $|G|=|N|\,\ell\,q(q^2-1)$, with $q \ge 7$, and up to duality we may assume that $m=\ell \, m_1(q-1)/2$ and $n=(q+1)n_1$, and then Euler's formula \eqref{e: sunny} gives
\[ r^d = \frac{|N|q}{2\,\lcm(m_1,n_1)} \, f_{(b)}(q) \ \  {\rm where} \ \  f_{(b)}(q) =
\frac{1}{\gcd(m_1,n_1)}\left( \frac{\ell(q^2{-}1)}{2}m_1n_1 - \ell(q{-}1)m_1 - 2(q{+}1)n_1\right). \]

Now if $m_1\neq n_1$, then by the same reasoning as in the proof of Lemma \ref{l: mn3}, we find that exactly two of the three terms in the bracketed expression in the formula for $f_{(b)}(q)$ are divisible by $r\,\lcm(m_1,n_1)$, and so $f_{(b)}(q)=2$. A separate consideration of the cases $m_1>n_1$ and $m_1 < n_1$ as in the previous proof yields the inequalities $2 = f_{(b)}(q) \ge (3q^2-10q-1)/2$ and $2 = f_{(b)}(q) \ge (3q^2-14q-13)/2$, and in both cases the right-hand side is greater than $2$ for $q\ge 7$, a contradiction.

It follows that $m_1=n_1$ and then $f_{(b)}(q) =  \ell m_1 (q^2-1)/2 - \ell(q-1) - 2(q+1)$,
and as also $f_{(b)}(q)=2\cdot p^j$ for some non-negative integer $j\le d-e$, the same argument as in the proof of the previous lemma shows that $2p^{t} \pm 4$ is divisible by $p^e-1$.  To achieve this, the only possibilities are $(p,e,t) = (3,1,0)$ and $(7,1,0)$, giving $q =7$ (since $q > 3$), and the rest follows..
\end{proof}
\medskip

In the above case we have $q=p=r=7$ and $m_1=7^k$ for some $k\ge 0$, and it follows that
$2{\cdot}7^j = f_{(b)}(7) = (24{\cdot}7^k-6)\ell-16$, which gives $\ell = (7^j +8)/(12{\cdot}7^k-3)$.
Just in the previous case, the requirement that this is a positive integer coprime to $q(q^2-1) = 336 = |\PGL_2(7)|$) puts severe restrictions on the possible values of $j$ and $k$.
In the next case, $\ell=1$, which makes things easier.
\medskip

\begin{lem}\label{l: pgl2 case c}
If $r=p$ and $\{m^*,n^*\} = \{q\pm 1,\, q\mp 1\}$, as in
case {\rm (c)} of {\rm Lemma \ref{l: cases}}, then $r=q=5$ and $m_1 = n_1$, and $O(G)$ is a $5$-group,
and one of the following holds$\,:$
\begin{enumerate}
 \item[{\rm (\,i)}] $\{m,n\}=\{4,6\}$, giving row {\rm B1} of {\rm Table~\ref{t: b}};
 \item[{\rm (ii)}] $\{m,n\}=\{20,30\}$, giving row {\rm B2} of {\rm Table~\ref{t: b}}.
\end{enumerate}
\end{lem}

\begin{proof}
Here up to duality we may take $m=m_1(q+1)$ and $n=n_1(q-1)$, and $|G| = |N|\, q(q^2-1)$, where $q$ ($=p^e$), $\,m_1, n_1$ and $|N|$ are all powers of $p=r$. This time Euler's formula \eqref{e: sunny} gives
\[ p^d = \frac{|N|q}{4\,\lcm(m_1,n_1)}  \, f_{(c)}(q) \ \   {\rm where} \ \
f_{(c)}(q)=\frac{1}{\gcd(m_1, n_1)} \left( (q^2{-}1)m_1n_1 - 2(q{+}1)m_1-2(q{-}1)n_1\right). \]

If $m_1\ne n_1$, then exactly two of the three terms in the bracketed expression in the formula for $f_{(c)}(q)$ are divisible by $r\,\lcm(m_1,n_1)$, and so $f_{(c)}(q)$ must be $4$. Then the same approach as taken in the proofs of Lemmas \ref{l: mn3} and \ref{l: mn4} gives $4 \ge 3q^2-8q-7$ when $m_1 > n_1$,
while $4 \ge 3q^2-8q+1$ when $m_1 < n_1$,  and in both cases we get contradictions, so $m_1 = n_1$.

Now $f:=f_{(c)}(q)/4 = n_1(q^2-1)/4-q$, where all of $f$, $q$ and $n_1$ are powers of $p$, and we consider three possibilities depending on how $n_1$ compares with $q$.

First suppose that $n_1<q$.   Then $q$ is divisible by $n_1$ (because both are powers of $p$), and therefore
$f=n_1((q^2-1)/4 - q/n_1)$, so $(q^2-1)/4 - q/n_1$ must be a power of $p$, but as  $(q^2-1)/4$ is coprime to $p$,  we see that this can happen only if $(q^2-1)/4 - q/n_1 = p^0 = 1$.  The latter implies that $q/n_1 = (q^2-1)/4 - 1 = (q^2-5)/4$ and hence $q^2-5$ is divisible by $p$, so $p = 5$, but as also $q^2-5$ is not divisible by $5^2$,
it follows that $(q^2-5)/4 = q/n_1$ must be $5$, so $q^2-5 = 5$ and hence $q = 5$, and $m_1 = n_1 = 1$.
Thus $(q,r,\{m,n\})=(5,5,\{4,6\})$, giving row~B1 of Table~\ref{t: b}.

Instead, if $n_1=q$, then $f = n_1(q^2-1)/4-q = (q^2-5)q/4$, but $(q^2-5)/4 \ne 1$ (since $q \ne 3$),
so $(q^2-5)/4$ must be a proper power of $p$.  Again this implies that $p = 5$ and then also $q = 5$
(as above), so $m_1 = n_1 = 5$, and we find that $(q,r,\{m,n\})=(5,5,\{4{\cdot}5, 6{\cdot}5 \}) =(5,5,\{20,30\})$, which gives row~B2 of Table~\ref{t: b}.

Finally, if $n_1> q$, then $f = n_1(q^2-1)/4-q > n_1(q^2-1)/4-n_1 = n_1((q^2-1)/4-1) > n_1$, and as  both $q$ and $n_1$ are powers of $p$ it follows that $f$ is divisible by $n_1$, and hence $q = n_1(q^2-1)/4 - f$ is divisible by $n_1$, which is impossible since $q < n_1$.
\end{proof}
\medskip

\begin{lem}\label{l: case d}
If $q=p\ge 7$, with $p-1 = 2^ir^\beta$ for some $i,\beta\ge 1$, and $\{m^*,n^*\} = \{\ell p,\, p+ 1\}$ where $\ell$ is coprime to $p(p^2-1)$ and $r$, as in case {\rm (d)} of {\rm Lemma \ref{l: cases}}, then $i=1$, so $(p-1)/2 = r^\beta$, and $m_1=n_1=r^s$ and $\{m,n\} = \{\ell r^s p, r^s(p+1)\}$ for some $s\ge 1$, and also $d\ge r^\beta+3\beta+1$, giving row {\rm B5} of {\rm Table~\ref{t: b}} with $t = \beta$.
\end{lem}

\begin{proof}
Without loss of generality we may take $m=\ell pm_1$ and $n=(p+ 1)n_1$, and then using the same approach as in the proofs of Lemmas \ref{l: mn3} to \ref{l:  pgl2 case c}, and the fact that $p-1 = 2^ir^\beta$, we find that
\[ r^d=  \frac{|N| \, 2^ir^\beta}{4\,\lcm(m_1, n_1)} \, f_{(d)}(p), \ \ \ {\rm where} \ \ \
f_{(d)}(p) = \frac{1}{\gcd(m_1,n_1)}\left( \ell p(p{+}1)m_1n_1 - 2\ell pm_1 - 2(p{+}1)n_1\right) \ .\]

If $m_1 \ne n_1$, then exactly two of the three terms in the bracketed expression in the formula for $f_{(c)}(q)$ are divisible by $r\,\lcm(m_1,n_1)$, and it follows that $2^{i} f_{(d)}(p)$ must be $4$, and then since both $2^{i}$ and $f_{(d)}(p)$ are even, they both must be $2$ (with $i = 1$).
Then the same approach as used in the proofs of Lemmas \ref{l: mn3} to \ref{l: pgl2 case c} shows that $2 \ge f_{(d)}(p) \ge 3p^2-5p-2$ if $m_1>n_1$, while $2 \ge f_{(d)}(p)\ge 3p^2-5p-6$ if $m_1<n_1$, giving a contradiction in both cases.

Hence $m_1 = n_1$ once more.
Moreover, Euler's equation becomes $r^d = f^*\,(|N|/m_1)(p-1)/2$ where $f^*= m_1\, f_{(d)}(p)/2 = (\ell(p^2+ p)m_1 - 2\ell p-2(p+ 1))/2$. This means that both $f^*$ and $(p-1)/2$ must be powers of $r$, and as $p-1=2^i r^\beta$, this gives $i=1$, and hence $p-1 = 2r^\beta$.
Now if $m_1=n_1=1$, then $2f^* =  \ell(p^2+ p)m_1 - 2\ell p-2(p+ 1) = \ell(p^2\ -p)-2(p+ 1) = (p-1)(\ell p - 2)-4$, which is not divisible by $r$, so $2f^*=2$, but then $2 = 2f^* = \ell(p^2 - p)-2(p+ 1) \ge (p^2 - p) - 2(p+ 1) = p^2 -3p - 2$ and so $0 \ge p^2 -3p - 4 = (p-4)(p+1)$, a contradiction.
Therefore $m_1=n_1=r^s$ for some $s\ge 1$.

In particular, as $m_1$ and $n_1$ are greater than $1$, and since they respectively divide $m_o$ and $n_o$ (as defined earlier), it follows from Lemma \ref{l: fstar nilpotent} that $F^*(G)=F(G)$.
Hence by Lemma~\ref{l: fstar lower bound} we find that $|N| \ge r^{\veps(q,r)+1}$, where $\veps(q,r) = (p{-}1)/2$ according to its definition (just before the statement of Lemma \ref{l: covers of K}), because $q = p \ne 9$ and $r \ne p$.  Thus $|N| \ge r^{(p-1)/2+1} = r^{r^\beta+1}$.
Furthermore, as
\[ 2f^* = \ell(p^2+ p)m_1 - 2\ell p-2(p+ 1) \ge (p^2+ p)m_1 - 2p-2(p+ 1) \ge m_1(p^2-3p-2)\]
and $p^2-3p-2 \ge (p-1)^2/2$ for all $p \ge 7$, we find that
\[f^*/m_1 = 2f^*/(2m_1) \ge (p^2-3p-2)/2 \ge ((p-1)/2)^2 = (r^\beta)^2 = r^{2\beta}.\]
This gives $r^d =  f^*\,(|N|/m_1)(p-1)/2 =  (f^*/m_1) \, |N| \, (p-1)/2 \ge r^{2\beta} \, r^{r^\beta+1} \, r^\beta
=  r^{r^\beta+3\beta+1}$,
and therefore $d\ge r^\beta+1+3\beta$.
\end{proof}
\medskip

Finally we can deal with case (e) in a similar but slightly different way.

\begin{lem}\label{l: case e}
If $q=p\ge 5$, with $p = 2^ir^\beta - 1$ for some $i,\beta\ge 1$, and $\{m^*,n^*\} = \{\ell p, p-1\}$ for some positive integer $\ell$ relatively prime to $p(p^2-1)$ and $r$, as in case {\rm (e)} of {\rm Lemma \ref{l: cases}}, then $i=1$, so $p \equiv 1$ mod $4$, with $(p+1)/2 = r^\beta$, and $m_1=n_1=r^s$ and $\{m,n\} = \{\ell r^s p, r^s(p-1)\}$ for some $s\ge 0$.
Moreover, if $|N|=1$, so that $s=0$, then $d\ge 5$ when $\beta=1$, while $d\ge \beta^2\log_2 r + 2\beta$ for all $\beta\ge 2$, as in row {\rm B6} of {\rm Table~\ref{t: b}} with $t = \beta$. On the other hand, if $|N|\ne 1$, then $d\ge r^\beta+ 3\beta-1$ for $p>5$, as in row {\rm B7} of {\rm Table~\ref{t: b}}, again with $t = \beta$.
\end{lem}

\begin{proof}
Mimicking the approach taken in the proof of Lemma \ref{l: case d}, we may take $m=\ell pm_1$ and $n=(p-1)n_1$,  and then Euler's formula gives us
\[ r^d=  \frac{|N|(p+1)}{4\,\lcm(m_1, n_1)} \, f_{(e)}(p), \ \ \ {\rm where} \ \ \
f_{(e)}(p) = \frac{1}{\gcd(m_1,n_1)}\left( \ell p(p{-}1)m_1n_1 - 2\ell pm_1 - 2(p{-}1)n_1\right). \ \]

The analogous arguments prove that if $m_1 \ne n_1$ then $2^{i}$ and $f_{(d)}(p)$ must both be $2$,  and then contradictions are obtained by showing that $2 = f_{(e)}(p) \ge 3p^2-11p+2$ and $2 = f_{(e)}(p)\ge 3p^2-11p+6$
in the cases $m_1> n_1$ and $m_1< n_1$ respectively, so that $m_1 = n_1 =  r^s$ for some $s\ge 0$.
Then since $\lcm(m_1, n_1) = \gcd(m_1, n_1) = m_1$, the above equation for $r^d$ can be re-written as
\begin{equation}\label{eq:case(e)}
r^d = {f^\dagger}\, (|N|/m_1) \,((p+1)/2) \ \ \ {\rm where} \ \ \
f^\dagger = f_{(e)}(p)/2 = \left ( m_1 \ell \, p(p-1) - 2\ell p -2(p - 1) \right )/\,2,
\end{equation}
and it follows that both $f^\dagger$ and $(p+1)/2$ must be powers of $r$, so that $i = 1$ and $p+1=2 r^\beta$,
and hence that $p\equiv 1$ mod $4$.

In the proof of the previous lemma, we were able to deduce that $s > 0$, and hence (obviously) that $|N| > 1$, and  then also that $F^*(G) = F(G)$, but in the case of this lemma, $s$ can be zero, and when that happens, it can happen that $F^*(G) \ne F(G)$. So instead we continue by assuming that $F^*(G) = F(G)$, and return at the end to the possibility that $F^*(G) \ne F(G)$.

Under this assumption, we find by Lemma \ref{l: fstar lower bound} that $|N|\ge r^{\veps(q,r)+1} = r^{(p-1)/2+1} =  r^{(p+1)/2} = r^{r^\beta}$.

In particular, if $s > 0$, then $m_1 = n_1 \ge 3$, and then since $\ell \ge 1$ we have
\[ 2f^\dagger =  m_1 \ell \, p(p-1) - 2\ell p -2(p - 1) \ge m_1(p^2-p)-2(2p-1) \ge 2m_1\, (p^2-1)/4, \]
with the last part holding because
\[ 2(m_1 \,(p^2-p) -  2m_1\, (p^2-1)/4 ) = m_1\,( 2p^2 - 2p - (p^2-1)) \ge 3 \, (p^2-2p+1) = 3(p-1)^2 \ge  4(2p-1). \]
This gives $f^\dagger \ge m_1 \, (p^2-1)/4 = m_1 \,(p-1)/2 \, (p+1)/2 = m_1 \, r^{\beta-1} \, r^\beta = m_1\,r^{2\beta-1}$.
\smallskip

On the other hand, if $s =0$ and $m_1 = n_1 = 1$, but $\ell\ge 3$, then
\[ 2f^\dagger =  \ell(p^2-p-2p) -  2(p - 1) \ge 3(p^2-p-2p) -  2(p - 1) = 3p^2 - 11p + 2 \ge 2(p^2-1)/4 \]
and therefore $f^\dagger \ge (p-1)/2 \, (p+1)/2 = r^{2\beta-1} = m_1 \, r^{2\beta-1}$,
while if $m_1 = n_1 = 1$ and $\ell = 1$ (the only other possibility for the odd positive integer $\ell$),
then for all $p \ge 11$ we have
\[ 2f^\dagger = p(p-1) - 2p - 2(p - 1)  = p^2 - 5p + 2 \ge 2(p^2-1)/4 \]
and again $f^\dagger \ge (p-1)/2 \, (p+1)/2 = r^{2\beta-1} = m_1 \, r^{2\beta-1}$.
\smallskip

Thus $f^\dagger \ge m_1\,r^{2\beta-1}$ in all cases except when $m_1=\ell=1$ and $p=5$.
In all of these non-exceptional cases, just as in the proof of the previous lemma, this lower bound on $f^\dagger$ in combination with the earlier bound $|N|\ge r^{\veps(q,r)+1} = r^{(p+1)/2} = r^{r^\beta}$ and the equation \eqref{eq:case(e)} gives
\[ r^d = |N| \, (\frac{p+1}{2}) \, \frac{f^\dagger}{m_1} \ge  r^{r^\beta} \, r^\beta \, r^{2\beta-1} \]
and hence $d \ge r^\beta+3\beta-1$ when $p>5$, giving the information in row B7 of Table~\ref{t: b}.

It now remains for us to consider the situation where $F^*(G)\ne F(G)$.   In this case $|N|=1$ by Lemma \ref{l: fstar nilpotent} (since $\{m_o,n_o\} = \{\ell, 1\}$), and so $m_1=n_1=1$, and hence $\{m,n\}=\{\ell p,p-1\}$.
Then 
equation \eqref{eq:case(e)} reduces to $r^d=f^\dagger\,(p+1)/2 = f^\dagger r^\beta$, and hence to  $r^{d-\beta} = f^\dagger = (\ell(p^2-3p) -2p+2)/2$.
Letting $j = d-\beta$ and substituting $2r^\beta-1$ for $p$ in the latter equation and rearranging terms gives
\[  \ell=\frac{2(r^j+p-1)}{p(p-3)} = \frac{r^j+2r^\beta -2}{2r^{2\beta}-5r^\beta + 2}  = \frac{r^j+2r^\beta -2}{(2r^\beta -1)(r^\beta -2)}, \]
from which we find that $2r^\beta-1$ divides $r^j+2r^\beta -2$ and hence divides $r^j-1$,
which implies that $j\ge \beta+1$, and therefore $d=j+\beta \ge 2\beta+1$.
\smallskip

Writing $j = k\beta+v$ where $k = \lfloor j/\beta \rfloor \ge 1$ and $0 \le v < \beta$,
and then dividing $2^k(r^j-1)$ by $2r^\beta-1$, we obtain
\[ 2^k(r^j-1) = (2r^\beta-1)\left(2^{k-1} r^{j-\beta} + 2^{k-1} r^{j-2\beta} + \cdots + 2r^{j-(k-1)\beta} + r^{j-k\beta}\right)
 + r^{j-k\beta}-2^k,  \]
It follows that  $2r^\beta-1$ divides $r^{j-k\beta}-2^k = r^{v}-2^k$.
The latter is a negative integer, for otherwise $2r^\beta-1 \le r^v - 2^k < r^\beta - 2^k < r^\beta$, a contradiction.  It follows that $2^k > r^v$, and then since $2r^\beta-1$ divides $2^k- r^{v}$,
we have $2r^\beta-1 \le 2^k- r^{v}$ and so $r^\beta \le (2^k- r^{v} +1)/2 \le 2^{k-1}$,
which implies that $\beta\log_2 r \le k-1$ and thus $k\ge \beta\log_2 r + 1$.

Finally, as $j\ge k\beta$ we conclude that $d=j+\beta \ge (k+1)\beta \ge (\beta\log_2 r + 2)\beta = \beta^2\log_2 r + 2\beta$. This gives the inequality for $d$ in row B6 of Table~\ref{t: b}, when $t = \beta\ge 2$.
In the case $\beta=1$ we have $v = 0$ and $k=j$, and then $2r-1 = 2r^\beta-1$ must divide $r^j  - 1$,
which cannot happen when $j \le 3$, so $j \ge 4$ and hence $d = \beta+j \ge 5$.
\end{proof}

Let us remark that in the situation where $m_1=\ell=1$ and $p=5$, an infinite family of examples of type $\{5,4\}$ with $N\ne 1$ will be provided by Proposition \ref{p:psl-sum} in Section \ref{s: constructions}
\smallskip

We have now covered all the five cases listed in Lemma~\ref{l: cases}, and thereby obtained all the seven possibilities shown in Table~\ref{t: b}, hence yielding a proof of part (B) of Theorem~\ref{t: main}.
\bigskip


\section{The case where \texorpdfstring{$G$}{G} is soluble: proof of Theorem \ref{t: main} for family C}\label{s: soluble}
\bigskip


In this section we operate under the following hypothesis.
\medskip

\noindent {\bf Hypothesis 3:} {\sl The group $G= \langle a,b,c\rangle$ with presentation of the form \eqref{eq:abc} is a soluble $(2,m,n)^*$-group with Euler characteristic $\chi_G=-r^d$ for some odd prime $r$ and positive integer $d$.}
\medskip

Recall that Theorem~\ref{t: odd order} implies that $G/O(G)$ is isomorphic to a $2$-group or to $S_4$.

Also recall that a group is almost Sylow-cyclic if all of its Sylow subgroups of odd order are cyclic and all of its Sylow $2$-subgroups are trivial or have a cyclic subgroup of index at most $2$. 
In the proof of the following lemma, we write $F(H)$ for the Fitting subgroup of a group $H$, and we use the fact that if $H$ is a soluble group, then $C_H(F(H))=Z(F(H))$ \cite[Ch. 6, Theorem 1.3]{gor}. Also we use the fact hat if $K$ is a normal subgroup of a group $H$, then $H/C_H(K)$ is isomorphic to a subgroup of $\Aut(K)$, which we may write as $H/C_H(K) \lesssim \Aut(K)$.
\smallskip

\begin{lem}\label{l: soluble 1}
The group $G$ contains a normal $r$-subgroup $N$ such that $G/N$ is an almost Sylow-cyclic group,
with $|G/N|_3 = 1$ or $3$, and $|G/N|_r=1$ whenever $r>3$.
\end{lem}

\begin{proof}
Observe first that by Lemma~\ref{l: sylow}, the Sylow $t$-subgroups of $G$ are cyclic for every odd prime $t\ne r$, and by Lemma~\ref{l: dihedral}, the Sylow $2$-subgroups of $G$ are either cyclic or dihedral.
Both of these facts hold also for all subgroups and all quotients of $G$.

As for the prime $r$, let $N=O_r(G)$, the largest normal $r$-subgroup of $G$, and let $G_N=G/N$. By the correspondence theorem in group theory, the group $O_r(G_N)$ is trivial. It follows that the Fitting subgroup $F_N=F(G_N)$ has order coprime to $r$, and this implies that every odd-order Sylow subgroup of $F_N$ is cyclic, while the (unique) Sylow $2$-subgroup of $F_N$ is cyclic or dihedral.

Now suppose that the Sylow $2$-subgroup of $F_N$ is cyclic (or possibly trivial). Then $F_N$ is a direct product of cyclic groups of coprime order, implying that $F_N$ itself is cyclic, and hence that $C_{G_N}(F_N) = Z(F_N)=F_N$, and $\Aut(F_N)$ is abelian, which in turn implies that $G_N/F_N \cong G_N/C_{G_N}(F_N)$ is abelian as well, being isomorphic to a subgroup of $\Aut(F_N)$.
But $G_N/F_N$ (like $G_N$) is generated by three involutions, and hence must be an elementary abelian $2$-group. It follows that $|G_N/F_N|_r=1$, and as $|F_N|$ is coprime to $r$, also $|G_N|_r=1$, and hence $G_N = G/N$ is almost Sylow-cyclic.

On the other hand, suppose that a Sylow $2$-subgroup of $F_N$ is dihedral, isomorphic to $D_\ell$ of order $2\ell$ for some $2$-power $\ell$, with $D_2 \cong C_2\times C_2 \cong V_4$ (by convention) if $\ell = 2$.
This means that $F_N\cong D_\ell\times C$ for some cyclic group $C$ of odd order, and here it is still the case that $C_{G_N}(F_N) = Z(F_N)$.
Again $G_N/C_{G_N}(F_N) \lesssim \Aut(F_N)$, but now some more care is needed. A fact we will find useful is implied by Corollary 3.8 of \cite{bid}, namely that $\Aut(F_N) \cong \Aut(D_\ell \times C) \cong \Aut(D_\ell)\times \Aut(C)$.

If $\ell\ge 4$, then $\Aut(D_\ell)$ is the holomorph of a cyclic group of $2$-power order $\ell$ and hence is a $2$-group, and then since the isomorphic copy of $L=G_N/C_{G_N}(F_N)$ in $\Aut(F_N)$ is generated by $3$ involutions, and $\Aut(C)$ is abelian, it follows that this copy of $L$ is a $2$-group, and therefore $L$ is a $2$-group itself. Also the group $G_N/F_N$ is a quotient of $G_N/Z(F_N)=G_N/C_{G_N}(F_N)=L$, so $G_N/F_N$ is a $2$-group as well. Hence this gives $|G_N/F_N|_r=1$ as before, so $|G/N|_r=1$, and again $G/N$ is almost Sylow-cyclic.

Finally, if $\ell=2$, then $F_N\cong (C_2\times C_2)\times C$, so $\Aut(F_N) \cong \Aut(C_2 \times C_2)\times\Aut(C)\cong S_3\times \Aut(C)$.  Then since $\Aut(C)$ is abelian and the  copy of $L=G_N/C_{G_N}(F_N) = G_N/Z(F_N)$ in $\Aut(F_N)$ is generated by $3$ involutions, we find that $L$ has order $2^i$ or $2^i \cdot 3$ for some $i$. Hence if $r > 3$ then $|G/N|_r=1$, while if $r = 3$ then $|G/N|_r=1$ or $3$, and again $G/N$ is almost Sylow-cyclic.
\end{proof}

For the remainder of this section we let $N$ be a normal $r$-subgroup of $G$ as given by Lemma \ref{l: soluble 1}, and note that $G/N$ is a $(2,m',n')^*$-group for some positive integers $m'$ and $n'$.

\smallskip
The almost Sylow-cyclic groups $H=\langle a,b,c\rangle$ with presentation of type \eqref{eq:abc}
that are soluble were classified in \cite[Theorem 4.1]{cps}, and we present a restricted version of that  classification below.
To do this, we convert the presentations given in \cite[Theorem 4.1]{cps} in terms of generators $x,y,z,t$ to
presentations in terms of the generating set $\{a,b,c\}$ used in this paper. This can be done with the help of
substitutions given in Table 1 of \cite{cps}, but with the roles of $a$ and $b$ reversed to match our notation in  \eqref{eq:abc}. As a result of this conversion, one may check that only three of the eight types of presentations listed in that table have the property that $\langle a,b,c\rangle = \langle ab,bc \rangle$, giving rise to non-orientable regular maps, namely those displayed in rows 1,5 and 8. These are the three types presented in terms of generators $\{a,b,c\}$ in the next theorem.  We note that the details of conversion calculations for case (b) below can also be extracted from \cite[Section 4]{cosi}.
\medskip

\begin{thm}\label{t: ass}
Let $H = \langle a,b,c\rangle$ be a soluble almost Sylow-cyclic $(2,m',n')^*$-group with $|H|>2$. Then $H$ is
isomorphic to one of the following groups:
\begin{enumerate}
\item $H_1(\ell)=D_{\ell}=\langle\, a,b,c \ |\  a^2,b^2,c^2,(ac)^2,(ab)^2,(bc)^\ell,a(bc)^{\ell/2} \,\rangle$, a dihedral group of order $2\ell$ \\ for even $\ell$, represented as a $(2,2,\ell)^*$-group;
\item $H_2(j,k)=\langle\, a,b,c \ |\  a^2, b^2, c^2, (ac)^2,(ab)^{2j},(bc)^{2k},b(ab)^j(bc)^k\,\rangle \cong \langle a,(ab)^2\rangle \times \langle c,(bc)^2\rangle \cong \\ D_j\times D_k$, of order $4jk$ for odd and coprime $j$ and $k$, represented as a $(2,2j,2k)^*$-group;
\item $H_3(\ell) = \langle\, a,b,c \ |\  a^2, b^2, c^2, (ac)^2, (ab)^4, (bc)^\ell, cbabc(ab)^2 \,\rangle \cong \langle a,bab \rangle \rtimes \langle b,c\rangle \cong (C_2\times C_2)\rtimes D_{\ell}$, of order $8\ell$ for $\ell \equiv 3$ mod $6$, represented as a $(2,4,\ell)^*$-group.
 \end{enumerate}
\end{thm}
\medskip

Observe here that if the order of $H$ is exactly divisible by a power of $3$, say $3^i$ for some $i\ge 1$, then in cases (a) and (b) the group $H$ contains a normal cyclic subgroup of order $3^i$, while in case (c) it contains a normal cyclic subgroup of order $3^{i-1}$ (since $\Aut(C_2 \times C_2) \cong S_3$). This will be relevant later in connection with applications of Lemma \ref{l: soluble 1} when $r=3$.

Taking $H=G/N$, which is generated by $(aN,bN,cN)$, Theorem~\ref{t: ass} gives three possibilities for $G/N$, and we will analyse these one by one. Here we have $G/N=\langle aN,bN,cN\rangle=\langle gN,hN\rangle$, where $g=ab$ and $h=bc$, with $gN$ and $hN$ having respective orders $m'$ and $n'$, which divide $m$ and $n$.
In particular, if $G/N$ is dihedral, of order $2\ell$, then since $gh = ac$ has order $2$, the orders of the generators $gN$ and $hN$ must be $2$ and $\ell$, and so $G/N$ is a $(2,2,\ell)^*$-group, with $\chi_{G/N}=1$, and also $\langle aN,cN\rangle \cong C_2\times C_2$ so $|G/N|$ is divisible by $4$.

We will also use the observation that if $s$, $i$ and $j$ are positive integers with $s > 2$ such that $s^i+1$ is divisible by $s^j-1$, then $s = 3$ and $j = 1$, with $i$ arbitrary; this is easily provable as an exercise.
\medskip

\begin{lem}\label{l: soluble dihedral}
Suppose $G/N$ is the dihedral $(2,2,\ell)^*$-group $D_{\ell}$ of even degree $\ell\ge 2$, as in item {\rm  (a)} of Theorem~{\rm \ref{t: ass}}. Then $r=3$ and $\ell = 1+3^i$ or $3+3^i$ for some $i \ge 0$, so that $\gcd(\ell,r) = 1$ or $3$, and $G$ is a $(2,6,3+3^i)^*$-group. Moreover, $N$ is a $3$-group, of order divisible by $9$ except when  $\ell = 3+3^i$ with $i > 0$, and finally, $-\chi_G = 3^{i-1}|N|$, giving the information in rows {\rm C1} and {\rm C2} of {\rm Table \ref{t: c}}.
\end{lem}

\begin{proof}
First, as noted above, the orders of $gN$ and $hN$ in $G/N$ must be $2$ and $\ell$, in some order,
and then since $N$ is an $r$-group (as given by Lemma \ref{l: soluble 1}), the orders of $g$ and $h$ in $G$ are given by $\{m,n\}=\{2r^\alpha, \ell r^\beta\}$ for some integers $\alpha,\beta\ge 0$. Then Euler's formula \eqref{e: sunny} applied to the $(2,m,n)^*$-group $G$ gives
\begin{equation}\label{eq:3i}  r^d = -\chi_G
= 2 \ell |N| \left (\frac{1}{4}- \frac{1}{4r^\alpha} -\frac{1}{2\ell r^\beta} \right )
= \frac{|N|}{r^{\alpha+\beta}} \left(r^\beta \ell (r^\alpha-1)/2-r^\alpha\right).
\end{equation}

Now if $\alpha < \beta$, then $r^\beta$ divides $|N|$ and so \eqref{eq:3i} reduces to $r^i=r^{\beta-\alpha}(r^\alpha-1)\ell /2-1$ for some $i\ge 0$, which is impossible if $r^\alpha = 1$, so $\alpha > 0$, and then the situation is still impossible since $r^{\beta-\alpha}(r^\alpha-1)\ell /2-1$ cannot be a power of $r$. Hence $\alpha \ge \beta$.

\smallskip
Next, suppose $\alpha > \beta$. Then $r^\alpha$ divides $|N|$ and so \eqref{eq:3i} reduces to $r^i=(r^\alpha-1)\ell /2-r^{\alpha-\beta}$ for some $i\ge 0$, and letting $\ell = 2\lambda$ and $\gamma =\alpha-\beta$, this implies
\begin{equation}\label{eq:3ii} \lambda(r^\alpha -1)= r^i + r^\gamma. \end{equation}

If $\beta > 0$, then $\alpha > \gamma > 0$, so $\alpha \ge 2$,
and $\lambda(r^\alpha -1) \ge r^\alpha-1 > 1+r^\gamma$ and hence $i \ge 1$,
and also $\lambda(r^\alpha -1)= r^i + r^\gamma$ is divisible by $r^u$ where $u = \min(i,\gamma)$.
Moreover, letting $v = \max(i,\gamma)$, we find that $(\lambda/r^u) (r^\alpha -1) = 1+r^{v-u}$,
so $1+r^{v-u}$ is divisible by $r^\alpha -1$. Hence if $v > u$ then our earlier observation implies that $r = 3$ and $\alpha = 1$, or otherwise $v = u$ and then $r^\alpha -1$ divides $2$ and the same conclusion follows,
in both cases contradicting the fact that $\alpha \ge 2$.

Hence if $\alpha > \beta$ then $\beta = 0$, and $\gamma =\alpha > 0$.
Now condition \eqref{eq:3ii} gives $\lambda-1 = (r^i+1)/(r^\alpha-1)$, so $r^\alpha-1$ divides $r^i+1$,
and again this implies that $r = 3$ and $\alpha = 1$ (with $i$ arbitrary), and in that case $\lambda-1 = (r^i+1)/2$,  so $\ell =  2\lambda = r^i+1+ 2 = 3+3^i$.

Moreover, if $i = 0$ then $\ell = 4$, and this gives a non-orientable regular map of type $\{2r^\alpha,\ell r^\beta\} = \{6,4\}$ with automorphism group a $(2,6,4)^*$-group $G\cong N{\cdot}D_4$, of order $8{\cdot}3^t$ for some $t\ge \alpha=1$.  There is no such map with $t=1$ in the list at \cite{C600}, so $t$ must be at least $2$, and hence $|N|$ is divisible by $9$ in that case.
Indeed the smallest soluble $(2,6,4)^*$-group $G$ in the case $\chi = 3^d$ for some $d$ is an extension of $(C_3)^2$ by $D_4$, by \cite[Theorem 2.2]{bns}), with $|G| = 72$, and $t=2$.

On the other hand, if $i > 0$ the group $G$ is a $(2,6,\ell)^*$-group with $\ell = 3+3^i$ and $|N|$ need not be divisible by 9, as shown by the non-orientable regular map N5.4 of type $\{6,6\}$, with $\ell = 6$ and $|G| = 36$, and the non-orientable regular map N29.6 of type $\{6,30\}$, with $\ell = 30$ and $|G| = 180$, in both cases with $|N| = |G|/(2\ell) = 3$.

In both cases it is easy to verify that $-\chi_G=3^{i-1}|N|$, and so we obtain row C1 of Table~\ref{t: c}.

\smallskip
Instead, suppose that $\alpha=\beta \ge 0$.  Then with $\lambda=\ell/2$ again, \eqref{eq:3i} reduces to $r^i = \lambda(r^\alpha-1)-1$, which implies that $\alpha > 0$, and then once again $r=3$ and $\alpha=1$, so that $3^i = 2\lambda -1 $ and hence $\ell = 2\lambda = 1+ 3^i$. This gives $\{m,n\} = \{2r, \ell r \} = \{6,3\ell\} = \{6,3+3^{i+1}\}$, and $G\cong N{\cdot}D_\ell$ as before but here $|G/N| = 2\ell = 2(1+3^i)$ is not divisible by $3$.
Accordingly, $|N|$ cannot be $3$, for otherwise $|G|_r = |G|_3 = 3$, so $G$ would be almost Sylow-cyclic, and as  $\{m,n\} = \{6,3+3^{i}\}$ with $i\ge 0$, this can only happen for an almost Sylow-cyclic group covered by case (b) of Theorem \ref{t: ass}, but in that case having type $\{2j,2k\}$ requires oddness and relative primality of $j$ and $k$, which is impossible since either $\{m,n\} = \{6,4\}$ or $\gcd(m,n) = \gcd(3,3+3^{i}) = 6$.  Hence $|N|$ is divisible by $9$.

Finally, in this case (where $(\alpha,\beta) = (1,1)$ and $\ell = 1 + 3^i$), it is again easy to verify that $-\chi_G=3^{i-1}|N|$, and so we obtain row C2 of Table~\ref{t: c}.
Note that the case $i=0$ is somewhat exceptional, giving non-orientable regular maps of type $\{6,4\}$ with automorphism group $G = N \rtimes D_4$, with $-\chi_G = |N|/3 = 3^{-1}|N|$.
\end{proof}
\smallskip

Before proceeding, we observe that in some cases, the same group $G$ can satisfy the conditions for both C1 and C2 in Lemma~\ref{l: soluble dihedral}, for the simple reason that a type of the form $\{6,n\}$ with $n$ divisible by $3$ could in some cases be both $\{6,\ell\}$ with $\ell = 3+3^i$ and $\{6,3\ell\}$ with $\ell = 1+3^{i-1}$.
For example, the map N5.4 of type $\{6,6\}$  in the list at \cite{C600} has automorphism group of order $36$, and this group $G$ has both a normal subgroup $N$ of order $3$ with $G/N \cong D_6$, and a normal subgroup $N$ of order $9$ (isomorphic to $C_3\times C_3$) with $G/N \cong D_2$.  The same kind of thing holds for the maps N11.2, N29.2, N29.3, N29.6, N83.2, N83.3 and N83.4 given in the split table in Corollary~\ref{c: main}.
Similarly, as we will see later, such maps can also satisfy the conditions for C4 when $r = 3$ and $j =k = 1$.

\medskip

\begin{lem}\label{l: soluble G2lk}
Suppose $G/N$ is $D_j\times D_k$, of order $4jk$ where $j$ and $k$ are odd and coprime, expressed as a $(2,2j,2k)^*$-group, as in case {\rm (b)} of Theorem~{\rm \ref{t: ass}}. Then up to duality, $G$ is a $(2, 2jr^\alpha, 2kr^\beta)^*$-group for some $\alpha\ge \beta$, with $r\equiv 3$ mod $4$.
Moreover,  either $jk$ is not divisible by $3$, or just one of $j$ and $k$ is exactly divisible by $3$,
and $(jr^\alpha-1)(kr^\beta-1) = r^i$ for suitable $i$ such that $i+\beta$ is odd. Also $d\ge 5$ if $r>3$ and $\alpha\ge 1$. Further details are as given in rows {\rm C3} and {\rm C4} of Table~{\rm \ref{t: c}}.
\end{lem}
\smallskip

\begin{proof}
First, $G$ is a $(2, 2jr^\alpha, 2kr^\beta)^*$-group for integers $\alpha$ and $\beta\ge 0$ as given, since $|N|$ is an $r$-group, and then Euler's formula \eqref{e: sunny} gives ...
\begin{equation}\label{eq:jkrd}  r^d = -\chi_G
= 4jk|N| \left (\frac{1}{4}- \frac{1}{4j r^\alpha} -\frac{1}{4k r^\beta} \right )
= \frac{|N|}{r^{\alpha+\beta}} \left(jk r^{\alpha+\beta} - jr^\alpha - kr^\beta \right).
\end{equation}
By interchanging $j$ and $k$ if necessary, we may suppose  $\alpha\ge \beta$, in which case
$jk r^{\alpha+\beta} - jr^\alpha - kr^\beta$ is divisible by $r^\beta$, and so $|N|$ is divisible by $r^\alpha$,
and
\begin{equation}\label{eq:jkrd2} (jr^\alpha-1)(kr^\beta-1) = r^{i+\beta}+1\end{equation}
for some $i\ge 0$.
Then since both $jr^\alpha-1$ and $kr^\beta-1$ are even, we find that $r^{i+\beta}+1 = (jr^\alpha-1)(kr^\beta-1)$ is divisible by $4$, so $r^{i+\beta} \equiv 3$ mod $4$, and hence $r\equiv 3$ mod $4$, and $i+\beta$ is odd.

\smallskip
Now if $\alpha=0$, then also $\beta=0$, and we find the entries in row C3 of Table~\ref{t: c}, occurring for every prime $r\equiv 3$ mod $4$ and every odd integer $i\ge 1$ such that $r^i+1$ admits a factorisation of the form $(j-1)(k-1)$ for odd and coprime $j,k\ge 3$. For the rest of the proof, we will suppose $\alpha\ge 1$.

\smallskip
We proceed by considering the case $j=1$, for which \eqref{eq:jkrd2} becomes $(r^\alpha-1)(kr^\beta-1) = r^{i+\beta}+1$, so that $r^{i+\beta}+1$ is divisible by $r^\alpha-1$, and hence $r=3$ and $\alpha=1$ (as before), and $\beta = 0$ or $1$, and also $2(kr^\beta-1) = r^{i+\beta}+1$. Thus $2jr^\alpha = 6$, while $2kr^\beta = r^{i+\beta}+3 = 3^i + 3$ or $3^{i+1}+3$. In particular, this gives us the same types as encountered in Lemma \ref{l: soluble dihedral}, but that is not surprising, for the fact that $k$ is odd implies that $D_j \times D_k \cong D_{2k}$ and hence in this case $G \cong N{\cdot}D_\ell$ where $\ell = 2k$ is even.

\smallskip
Now supposing that both $\alpha \ge 1$ and $j > 1$, we can prove the following five facts:
\,(a)\, if $r>3$ then $r^{\alpha+1}$ divides $|N|$ and hence $d > i$,
\,(b)\, $i > \alpha$,
\,(c)\, if $r>3$ then $i \ge 3$,
\,(d)\, $i \ge 2\alpha-\beta+1$, and
\,(e)\, $i \ge \alpha+\beta+1$ if $i\ge 3$ and $\beta\ge 1$.

\smallskip
For (a), if $N$ is cyclic, then by Lemma \ref{l: soluble 1} the $(2,m,n)^*$-group $G$ is almost Sylow-cyclic and so it must be of type (b) in Theorem \ref{t: ass}, with  $\{m,n\} = \{2j,2k\}$ for odd and coprime $j,k \ge 3$, but then $\alpha=\beta=0$, a contradiction. Hence $N$ is not cyclic.  On the other hand, $G$ contains an element of order $2jr^\alpha$ and hence a cyclic subgroup of order $r^\alpha$,  so $|N|$ must be divisible by $r^{\alpha+1}$,
and then because we know from \eqref{eq:jkrd} and \eqref{eq:jkrd2} that $r^d = (|N|/r^\alpha)\,r^i$, it follows that $d > i$.

Next, for (b), as $j\ge 3$ and $k \ge 1$ we have  $jr^\alpha-1 \ge 2r^\alpha$ and $kr^\beta-1 \ge r^\beta-1$. Also if $\beta>0$ then $2r^\alpha(r^\beta-1) = r^{\alpha+ \beta} + r^\alpha(r^\beta -2) > r^{\alpha+ \beta}+1$, and then it follows easily from \eqref{eq:jkrd2} that $r^{i+\beta}+1= (jr^\alpha-1) (kr^\beta-1)\ge 2r^\alpha(r^\beta-1) > r^{\alpha+ \beta}+1$, and hence that $i >\alpha$. On the other hand, if $\beta=0$, then by \eqref{eq:jkrd2} we find $(jr^\alpha -1) (k-1) = r^i+1$, implying that $jr^\alpha \ge 3$ and $k\ge 3$ (since both are odd), but then $r^i+1 \ge (3r^\alpha-1){\cdot} 2 > r^\alpha+1$ and so again $i>\alpha$.

For (c), note that $i > \alpha \ge 1$ and hence $i \ge 2$.  But if $i = 2$, then $\alpha = 1$ and so $\beta \le 1$,
but since $i+\beta$ is odd (as proved much earlier), it follows that $\beta = 1$.  So now  \eqref{eq:jkrd2}  gives
$r^3 +1 = (jr-1)(kr-1)$, and then rearranging and dividing by $r$ gives $r^2 - jkr + (j+k) = 0$, which is symmetric in $j$ and $k$.
If $j = 1$ then $r^2 - kr + (k+1) = 0$, so $k+1$ is divisible by $r$ and hence $k = tr-1$ for some integer $t \ge 1$,
but then $0 = r^2 - (tr-1)r + tr$, so $k(r-1) = r^2 + 1$ and therefore $k = (r^2+1)/(r-1) = r+1 + 2/(r-1)$.  This implies that $2/(r-1)$ is an integer, so $r-1 \in \{\pm 1, \pm 2\}$, which is impossible since $r > 3$.  Hence $j > 1$, and the analogous argument shows that $k > 1$. Then since $j$ and $k$ are odd and distinct, at least one of them is greater than $3$, so $(j-1)(k-1) > 2$, and therefore $jk > j+k+1$.
Next, the discriminant of $r^2 - jkr + (j+k)$ as a quadratic in $r$ is the odd integer $D = (jk)^2 - 4(j+k)$, which is less than $(jk)^2$, and is also greater than $(jk-2)^2 = (jk)^2-4jk+4$ since $4jk > 4(j+k)+4$.
Then since $D$ must be a perfect square in order for $r^2 - jkr + (j+k) = 0$ to have a solution, it follows that $D = (jk-1)^2$, which is even, a contradiction. Thus $i \ge 3$ (when $r > 3$).

Next, for (d) we note that \eqref{eq:jkrd2} gives $r^{i+\beta} = jk r^{\alpha+\beta} - jr^\alpha - kr^\beta$,
and then since $i > \alpha$ (by (b)), it follows that $r^\alpha$ divides $jk r^{\alpha+\beta} - jr^\alpha - kr^\beta$
and hence divides $kr^\beta$, so that $k = r^{\alpha- \beta}\kappa$ for some odd $\kappa\ge 1$,
and therefore \eqref{eq:jkrd2} reduces to $(jr^\alpha-1)(\kappa r^\alpha-1) = r^{i+\beta}+1$.
Then since $jr^\alpha-1\ge 2r^\alpha$ and $\kappa r^\alpha-1 \ge r^\alpha-1$, the last equation in the previous sentence gives $ r^{i+\beta}+1 \ge 2r^\alpha(r^\alpha-1) > r^{2\alpha}+1$ and hence $i+\beta > 2\alpha$,
so $i > 2\alpha-\beta$ and therefore .$i \ge 2\alpha-\beta+1$.

Finally, (e) suppose that $i\ge 3$ and $\beta\ge 1$.
Rearranging \eqref{eq:jkrd2} gives $jr^\alpha + kr^\beta = jkr^{\alpha+\beta}-r^{i+\beta}$,
and dividing through by $r^\beta$ gives $r^{\alpha-\beta}+k = jkr^{\alpha}-r^i$, which equals $tr^\alpha$
for some $t$ (as $i>\alpha$), and this must be even, so $t \ge 2$.
In particular, $r^{\alpha-\beta}+k = tr^\alpha \ge 2r^\alpha$, so $k \ge 2r^\alpha - r^{\alpha-\beta} > r^\alpha$
and hence $k > 1$.
Also dividing $jkr^{\alpha}-r^i = tr^\alpha$ by $r^\beta$ gives $jkr^{\alpha-\beta} - r^{i-\beta} = tr^{\alpha-\beta}$,
which implies that
\[
2 < (jr^{\alpha-\beta}-1)(k-1) = jkr^{\alpha-\beta} - jr^{\alpha-\beta}-k+1 <  jkr^{\alpha-\beta} - (r^{\alpha-\beta}+k) +1 = (tr^{\alpha-\beta} +  r^{i-\beta}) - tr^\alpha + 1
\]
and so $tr^\alpha - tr^{\alpha-\beta} < r^{i-\beta}-1$.
But also $tr^\alpha - tr^{\alpha-\beta} = t(r^\alpha - r^{\alpha-\beta}) \ge 2(r^\alpha - r^{\alpha-\beta}) > r^\alpha$,
and so we find that $r^\alpha <  tr^\alpha - tr^{\alpha-\beta} < r^{i-\beta}-1 < r^{i-\beta}$, so $\alpha < i-\beta$,
and therefore $i \ge \alpha+\beta+1$.

\smallskip
These facts enable us to prove that $i \ge 4$ and $d\ge 5$ when $r>3$ and $\alpha\ge 1$.

So suppose that $r>3$ and $\alpha\ge 1$.
Then by (c) and (d) we have $i \ge 3$ and $i \ge 2\alpha-\beta+1$, and  hence by (e) also $i \ge \alpha+\beta+1$ if $\beta\ge 1$. So now if $\alpha\ge 3$, then $i \ge 7$ when $\beta = 0$, or $i \ge 5$ when $\beta \ge 1$, but also we know (from earlier) that $i+\beta$ is odd, and so  $i \ge 6$ when $\beta \ge 1$.
In particular, $i \ge 6$ when $\alpha \ge 3$.  Similarly, $i \ge 4$ when $\alpha = 2$.
When $\alpha=1$, however, we find that $i \ge 4$ when $\beta \ge 1$, but only $i \ge 3$ when $\beta = 0$.

We can eliminate the latter case (where $\alpha=1$, $\beta=0$ and $i = 3$) by a similar argument to the one
used in the proof of (c) above.  In this case \eqref{eq:jkrd2}  gives $r^3 +1 = (jr-1)(kr-1)$, and rearranging gives $r^3 - jkr + jr +k = 0$.  Thus $k = tr$ for some odd positive integer $t$, which divides $k$ and hence is coprime to $j$, so that $(j-1)(t-1) > 2$.  This time we obtain the quadratic equation $0 = r^2 - jtr + (j+t)$, which is the same as in (c) but with $t$ in place of $k$, and we reach a similar contradiction.

Therefore $i \ge 4$ in all cases, and by (a), it follows that  $d\ge 1+i\ge 5$, as claimed.

\smallskip
Next, we recall from Lemma \ref{l: soluble 1} that  $|G/N|_3 = 1$ or $3$, which implies that either
$j$ and $k$ are both not divisible by $3$, or that just one of them is exactly divisible by $3$.

\smallskip
The details in rows C3 and C4 of Table~\ref{t: c} follow from what we have shown above.
\end{proof}
\medskip

Here we make some additional observations about the case where $\beta < \alpha$.
In the simplest instance, $\alpha=1$ and $\beta=0$, and there are infinitely many examples of quadruples $(i,j,k)$ satisfying \eqref{eq:jkrd2} for each $r\in \{3,11,19\}$.
When $r=3$, we may take $k = 3$ and any odd $i\ge 3$, and set $j=(3^{i-1}+1)/2$, and obtain non-orientable regular maps with type $\{2j{\cdot}3,2k\} = \{3^i+3,6\}$, the duals of which appeared also in Lemma \ref{l: soluble dihedral}.
When $r=11$, we may take $j = 5$ and $i=9+18u$ for any integer $u\ge 0$, and set $k=(11^{9+18u}+55)/54$, giving maps with type $\{2j{\cdot 11},2k\}$.  (The smallest such $k$ is $43365699$, arising when $u=0$, and gives a map of type $\{110,87331398\}$.)
When $r=19$, we can do with the same with $j = 21$ and $i = 9 + 18u$ for $u\ge 0$, this time setting $k = (19^{9+18u}+399)/398$, and giving maps with type $\{2j{\cdot 19},2k\}$.  (The smallest example, occurring for $u=0$, has type $\{798,1621546220\}$.)
Note that $3\mid k$ when $r\in \{3,11\}$, and $r\mid k$ when $r\in \{11,19\}$.
\medskip

\begin{lem}\label{l: soluble G3k}
Suppose $G/N$ is $(C_2\times C_2)\rtimes D_{\ell}$ for some $\ell \equiv 3$ mod $6$, expressed as a $(2,4,\ell)^*$-group, as in case {\rm (c)} of Theorem~{\rm \ref{t: ass}}. Then $G$ is a $(2, 4r^\alpha, \ell r^\beta)^*$-group for some integers $\alpha\ge\beta\ge 0$, and $\ell$ is odd and exactly divisible by $3$, and either $r=3$ and $\alpha > \beta$, or $r\equiv 5$ mod $6$.  Further details are as given in one of rows {\rm C5} to {\rm C7} of {\rm Table~\ref{t: c}}.
Moreover, if $|N| > 1$, then $\alpha\ge 1$, and $|N|\ge r^{\alpha+1}$, and $d > \alpha^2\log_2r-\alpha-\beta +1$,
which is the bound in row {\rm C7} of Table~{\rm \ref{t: c}}. Also if $r=3$ then $d\ge 3$ when $(\alpha,\beta) =(1,0)$, while $d\ge 5$ in all the remaining cases where $r=3$ or $r\equiv 5$ mod $6$.
\end{lem}
\smallskip

\begin{proof}
Just as in the proof of Lemma \ref{l: soluble G2lk}, the hypotheses imply that $G$ is a $(2, 4r^\alpha, \ell r^\beta)^*$-group for integers $\alpha,\beta\ge 0$, and then Euler's formula gives
 \begin{equation}\label{eq:Eu(c)}
  r^d = -\chi_G
  = 8\ell |N| \left (\frac{1}{4}- \frac{1}{8 r^\alpha} -\frac{1}{2\ell r^\beta} \right )
  = \frac{|N|}{r^{\alpha+\beta}} \left(2\ell r^{\alpha + \beta}-4r^\alpha -\ell r^\beta\right).
 \end{equation}

Now if $\alpha < \beta$, then $|N|$ is divisible by $r^\beta,$ so $2\ell r^\beta - 4 - \ell r^{\beta-\alpha}$
must be $r^\gamma$ for some integer $\gamma \ge 0$, and then since $4$ is not divisible by $r$,
we have $\gamma = 0$, so $\ell\,(2r^\beta - r^{\beta-\alpha}) = 2\ell r^\beta - \ell r^{\beta-\alpha} = 5$,
which is impossible since $\ell \equiv 3$ mod $6$.

Hence $\alpha\ge \beta$, which implies that $|N|$ is divisible by $r^\alpha$, and then this time \eqref{eq:Eu(c)} tells us  that $2\ell r^{\alpha}-4r^{\alpha-\beta} -\ell = r^\gamma$ for some integer $\gamma \ge 0$, and then \begin{equation}\label{eq:alpbet} \ell
=  \frac{4r^{\alpha-\beta} + r^\gamma}{2r^\alpha - 1}
= r^{\alpha-\beta}\cdot \frac{4 + r^\delta}{2r^\alpha - 1}
\ \hbox{ where } \ \delta = \gamma-(\alpha-\beta).\end{equation}

Clearly we cannot have $\alpha=\gamma=0$, for otherwise $\beta = 0$ and $\delta = 0$ and then $\ell = 5$.  Also if $\alpha\ge 1$ and $\gamma \le \alpha-1$ then $4r^{\alpha-\beta} + r^\gamma \le 4r^\alpha +r^{\alpha-1} < 3(2r^\alpha-1)$ so $\ell = \frac{4r^{\alpha-\beta} + r^\gamma}{2r^\alpha - 1} < 3$, which is impossible, again since $\ell \equiv 3$ mod $6$.  Hence if $\alpha\ge 1$ then $\gamma\ge \alpha$, so  $\delta = \gamma-\alpha+\beta \ge 0$.
Similarly if $r>3$ and $\gamma\le \alpha$ then $4r^{\alpha-\beta} + r^\gamma \le 5r^\alpha < 3(2r^\alpha-1)$, again giving a contradiction, so if $r>3$ then $\gamma > \alpha$ and  $\delta = \gamma-\alpha+\beta \ge 1$.

\smallskip
Now suppose that $r=3$.  Then by \eqref{eq:alpbet} we have
\begin{equation}\label{eq:betgam} \ell = 3^{\alpha-\beta}\cdot \frac{4+3^\delta}{2{\cdot}3^\alpha-1},\end{equation}
and it follows that $\alpha \ge 1$, for otherwise $\alpha= \beta = 0$ and then $\ell = 4+3^\delta$, which is impossible. Hence by an observation in the previous paragraph, $\delta \ge 0$, and then the hypothesis that $\ell$ is an odd multiple of $3$ implies that $4+3^\delta$ is divisible by $2{\cdot}3^\alpha-1$ (for some $\alpha\ge 1$), and moreover, as $4+3^\delta$ is not divisible by $3$, neither is $(4+3^\delta)/(2{\cdot}3^\alpha-1)$, and so $\alpha - \beta > 0$.
This leads to the data in row C6 of Table~\ref{t: c}.

In fact there are infinitely many possible values of $\delta$, even in the case where $\alpha=1$ and $\beta=0$.
For $\ell$ to be an integer in this case, we need $4+3^\delta \equiv 0$ mod $5$, so $\delta \equiv 0$ mod $4$, and if $\delta=4\mu$ then $\ell = 3{\cdot}(4 + 3^{4\mu})/5 = 3{\cdot}(5 + (81^\mu-1))/5$, which is an odd multiple of $3$ for every $\mu\ge 0$.

Other such infinite families of suitable values of $\ell$ can be obtained similarly for pairs $(\alpha,\beta)$ with $\beta=\alpha-1$. Here the integrality condition coming from \eqref{eq:betgam} is satisfied for $\alpha=2$ whenever $\delta\equiv 4$ mod $16$, and for $\alpha =3$ whenever $\delta\equiv 20$ mod $52$, and for $\alpha=7$ whenever $\delta\equiv 2172$ mod $4372$, and for $\alpha=8$ whenever $\delta\equiv 1296$ mod $2624$, and for $\alpha=12$ whenever $\delta\equiv 531416$ mod $1062880$, and so on.
In fact, whenever $s=2{\cdot}3^\alpha-1$ is a prime and the order of $3$ mod $s$ is equal to $s-1$, there exists a positive integer $\delta_o < s$ such that the value of $\ell$ given by \eqref{eq:betgam} is an odd multiple of $3$ for every positive integer $\delta\equiv \delta_o$ mod $(s-1)$. This is the way examples for $\alpha= 1,\ 2,\ 3,\ 7$ and $12$ (with $\beta=\alpha-1)$ can be constructed, and such an approach also allows for varying the value of $\beta$ between $0$ and $\alpha-1$. (The examples for $\alpha\le 3$ will be used later in the proof.)
\smallskip

On the other hand, suppose that $r>3$.  In this case, as $r$ is prime we have $r \equiv \pm 1$ mod $6$,
but if $r \equiv 1$ mod $6$ then $r^{\alpha-\beta}$, $4+r^\delta$ and $2r^\alpha-1$ are congruent to $1$, $5$ and $1$ mod $6$, and then \eqref{eq:alpbet} gives $\ell \equiv 5$ mod $6$, a contradiction, so $r \equiv -1 \equiv 5$ mod $6$. By a similar argument, also $\delta$ must be odd.

Now if $|N| = 1$, then $G$ is a $(2, 4, \ell)^*$-group, and $\alpha=\beta=0$, and from \eqref{eq:Eu(c)} it follows that $\chi_G = r^d = \ell-4$, and also $G$ is almost Sylow-cyclic, as in row C5 of  of Table~\ref{t: c}.

Suppose instead that $|N| > 1$, so that $r^\alpha$ divides $|N|$. If $|N|=r^{\alpha}$, then $\alpha \ge 1$, and $N$ is cyclic (generated by $(ab)^4$). This implies, in turn, that $G$ is almost Sylow cyclic. But then Theorem~\ref{t: ass} shows there is only one possibility for the type $\{m,n\} = \{4r^\alpha,\ell r^\beta\}$, namely $\{4,\ell \}$, with $\alpha=0$ (and $\beta = 0$), giving a contradiction. Hence if $|N| >  1$ then $|N|$ is divisible by $r^{\alpha+1}$.
Moreover, as  $r^d = \frac{|N|}{r^{\alpha+\beta}} \left(2\ell r^{\alpha + \beta}-4r^\alpha -\ell r^\beta\right)
= \frac{|N|}{r^{\alpha}} \left(2\ell r^{\alpha}-4r^{\alpha-\beta} -\ell \right) = \frac{|N|}{r^{\alpha}}\,r^\gamma$,
this implies that $r^d \ge r^{1+\gamma}$, and therefore $d \ge \gamma+1$.
\smallskip

Next, we derive a lower bound on $d$ when $\alpha\ge 1$, whether $r=3$ or $r\equiv 5$ mod $6$.
Dividing $\delta$ by $\alpha$ gives $\delta = t \alpha + \eta$ with $0 \le \eta < \alpha$, and the requirement from \eqref{eq:alpbet} that $2r^\alpha - 1$ divides $r^\delta + 4$ is equivalent to
$2r^\alpha - 1$ dividing $2^{t}(r^\delta + 4)$, since $2r^\alpha - 1$ is odd.
A calculation now gives
\[
2^{t}r^\delta + 2^{t+2} = 2^{t}(r^\delta + 4)
= (2r^\alpha-1)\left(2^{t-1} r^{\delta-\alpha}+2^{t-2} r^{\delta-2\alpha}+\cdots +
 r^{\delta-t\alpha}\right) + r^{\delta-t\alpha}+2^{t+2}\ ,
\]
and it follows that $2r^\alpha-1$ divides $r^{\delta-t\alpha}+2^{t+2} = r^\eta + 2^{t+2}$.
In particular, $2r^\alpha-1 \le r^\eta + 2^{t+2}$, so  $2r^\alpha-r^\eta-1\le 2^{t+2}$,
and as $\eta < \alpha$, it follows that  $r^\alpha < 2r^\alpha-r^\eta-1 \le 2^{t+2}$.
Hence $t+2 > \alpha\log_2r$, and then because $\alpha\ge 1$ and $\delta \ge t\alpha$,
we find that $\delta > (\alpha\log_2r - 2)\alpha = \alpha^2\log_2r-2\alpha$ and therefore
$\gamma = \delta+\alpha-\beta > \alpha^2\log_2r-\alpha-\beta$.
Thus $d \ge \gamma+1 > \alpha^2\log_2r-\alpha-\beta +1$, which is precisely the lower
bound $d_{r,\alpha,\beta}$ in row C7 of Table~\ref{t: c}.
\smallskip

In particular, if $r\equiv 5$ mod $6$, then $r \ge 5$, so for $\alpha\ge 2$ we have
$d \ge d_{r,\alpha,\beta} > 2^2{\cdot}2-2-2+1=5$,
while if $r=3$ and $\alpha\ge 4$  then $d \ge 4^2 \cdot 1 -4-4+1 = 9$,
and if $r=3$ and $\alpha\ge 3$ then $d\ge 3^2\cdot  1 -3-3+1 =4$.
\smallskip

We now look in more detail at the cases where $\alpha=1$ for $r\equiv 5$ mod $6$, and $\alpha\le 3$ for $r=3$.

First suppose $r\equiv 5$ mod $6$ and $\alpha=1$.  Then $t=\delta$ and $\eta=0$, and then since
$2r^\alpha-1 \le r^\eta + 2^{t+2}$ we have $2r-1 \le 1 + 2^{\delta+2}$ and therefore $r\le 2^{\delta+1}+1$.
Hence if $d\le 3$, then $\gamma \le d-1 \le 2$ so $\delta = \gamma - \alpha + \beta \le 2-1+1 = 2$
and so $r\le 2^{\delta+1}+1 \le 2^3 +1 = 9$, which forces $r = 5$.
Similarly, if $d \le 4$ then $\gamma \le 3$ so $\delta \le 3$ and therefore $r= 5$ when $\delta\in \{1,2\}$
while $r \le 2^4 +1 = 17$ and so $r \in \{5,11,17\}$ when $\delta=3$.
But if $r=5$ and $\delta\in \{1,2\}$, then $\ell=r^{1-\beta} \cdot (r^\delta+4)/(2r-1) \in \{1,5,29/9,145/9\}$,
contradicting the assumption that $\ell$ is an integer multiple of $3$, and similarly, if $r\in \{5,11,17\}$ and $\delta=3$, then $r^{1-\beta} \cdot (r^\delta+4)/(2r-1)$ is not an integer multiple of $3$.

Hence if $r\equiv 5$ mod $6$ and $\alpha = 1$ then $d\ge 5$, and as we already saw that if $r\equiv 5$ mod $6$ then $d \ge 5$ for $\alpha\ge 2$, it follows that  if $r\equiv 5$ mod $6$ and $\alpha \ge 1$ then $d\ge 5$.

On the other hand, suppose $r=3$ and $\alpha \le 3$.
Then the same kind of calculations as above show that if $\alpha = 2$ or $3$, then $\delta\ge 4$ and hence $\gamma = \delta+(\alpha-\beta) \ge 4 + 1 = 5$. But if $\alpha=1$, $\beta=0$ and $\delta=0$ then $\ell = 3$ (and  $\gamma = 1$), and the map type is $\{4r^\alpha,\ell r^\beta\} = \{12,3\}$. In the latter case, Euler's formula gives $-\chi = 24|N|(1/4-1/24-1/6) = |N|$, with $|N|$ divisible by $r^{\alpha+1} = 3^2$. The smallest such example comes  from the map N29.1 at \cite{C600}), via a unique $(2,12,3)^*$-group $G$ of order $648 = 9\cdot 72$ with $-\chi_G = 3^3$.  It follows that when $r=3$ we have $d\ge 5$ if $\alpha=2$ or $3$, and $d\ge 3$ if $\alpha=1$ and $\beta=0$.

The facts derived in this proof imply validity of entries in rows C5 to C7 of Table~\ref{t: c}.
\end{proof}

The combination of Lemmas \ref{l: soluble dihedral} to \ref{l: soluble G3k} gives a proof of part (C) of Theorem
\ref{t: main}.

\smallskip
Finally, we note that for every prime $r>3$, as well as in the case when $r=3$ but $|G/N|_3=1$, the Schur-Zassenhaus theorem implies that $G \cong N\rtimes H$, where $H$ is one of $D_\ell$ or $D_j\times D_k$, depending on which of the cases (a) and  (b) of Theorem \ref{t: ass} applies.
\medskip


\section{Constructions and a proof of Corollary \ref{c: main}}\label{s: constructions}
\bigskip


In this section we establish existence and present constructions of examples of each of the 18 types of  $(2,m,n)^*$-groups corresponding to the entries in Tables~\ref{t: a}, \ref{t: b} and \ref{t: c}. We also prove Corollary
\ref{c: main}, which classifies $(2,m,n)^*$-groups $G$ such that $\chi_G= -r^d$ where $r$ is an odd prime and  $d\leq 4$.
\medskip

Most of the constructions for $(2,m,n)^*$-groups $G=\langle a,b,c\rangle$ for varying types $\{m,n\}$ are derived from covers of the carrier surfaces of the corresponding regular map $M=(G;a,b,c)$. We say that a $(2,m',n')^*$-group $H=\langle a',b',c'\rangle$ is a {\em cover} of a $(2,m,n)^*$-group $G$ as above if there is a group epimorphism $\theta\!: H\to G$ taking the ordered triple $(a',b',c')$ to $(a,b,c)$. In the case of branched covers, branch points of the same order will be allowed to occur at every vertex and at every face-centre in the map. For unbranched covers, which are also called {\em smooth}, we have $m'=\ord(a'b')=\ord(ab)=m$ and $n'=\ord(b'c')=\ord(bc)=n$. Similarly, the epimorphism $\theta$  is called smooth if it preserves the orders of $a'b'$ and $b'c'$.
\medskip

\begin{prop}\label{p: infinite family}
Let $G$ be a finite $(2,m,n)^*$-group with Euler characteristic $\chi < 0$. Then for every odd positive integer $s$,
there exists a finite $(2,m,n)^*$-group which is a smooth cover of $G$ with Euler characteristic $s^{1-\chi} \chi$.
\end{prop}

\begin{proof}
The full triangle group $\Delta= \Delta(2,m,n)$ is a smooth cover of $G$, and so the kernel $K$ of the corresponding epimorphism $\Delta\to G$ is torsion-free, and is isomorphic to the fundamental group of the non-orientable supporting surface $\mathcal{S}$ of the regular map given by the $(2,m,n)^*$-group $G$. (This is a consequence of a well known fact in algebraic topology; see, for example, the final statement in \cite[Subsection 19B, p.~38]{AhlSar}, according to which every topological space with a universal cover is the quotient of its universal cover by its fundamental group, acting as a group of covering transformations.) The non-orientable genus of $\mathcal{S}$ is $g = 2-\chi$, and so its fundamental group $K$ is a Kleinian group with signature $(g; - ; [-]; \{-\})$, and hence has a presentation in terms of $g$ generators $d_1,\, d_2,\, ..., d_{g}$ subject to the single defining relation  $d_{1}^{\,2}d_{2}^{\,2}...d_{g}^{\,2} = 1$; see \cite[\S 43B]{AhlSar} or \cite[Ch.~5]{Massey}. It follows that the abelianisation $K/K'$ of $K$ is isomorphic to $C_2 \times (\mathbb Z)^{g-1}$.

We can now apply a non-orientable version of the `Macbeath trick'.
For any odd positive integer $s$, let $K^{(s)}$ be the subgroup of $K$ generated by all the $s$-th powers of elements in $K$, and let $L=K'K^{(s)}$. Then $L$ is a characteristic subgroup of $K$ and hence a normal subgroup of $\Delta$. Moreover, $K/L$ is isomorphic to $(C_s)^{g-1}$, with the `$C_2$' part of $K/K'$ included in $L$ since $s$ is odd, and thus $\Delta/L$ is an extension of $K/L \cong (C_s)^{g-1}$ by $\Delta/K \cong G$ (which splits if $s$ is coprime to $|G|$).

Since $K$ is torsion-free, so is $L$. Also the natural epimorphism from $\Delta$ to $\Delta/L$ must take the ordinary triangle group $\Delta^+= \Delta^+(2,m,n)$ to $\Delta/L$, rather than a subgroup of index $2$ in $\Delta/L$, and so $\wt{G}=\Delta/L$ is a finite $(2,m,n)^*$-group, and hence a smooth $s^{g-1}$-fold cover of $G$.
Accordingly, the supporting surface $\wt{\mathcal{S}}$ of the regular map defined by the $(2,m,n)^*$-group $\wt{G}$ is  a non-orientable smooth cover of $\mathcal{S}$, with Euler characteristic $s^{g-1}\chi = s^{1-\chi} \chi$.
\end{proof}
\medskip

\begin{cor}\label{c: infinite family}
Let $G$ be a $(2,m,n)^*$-group with Euler characteristic $-r^d$ for some odd prime $r$. Then for every non-negative integer $\alpha$, there exists a $(2,m,n)^*$-group which is a smooth cover of $G$ with Euler characteristic $-r^\beta$, where $\beta = \alpha(1+r^d)+d$.
\end{cor}

\begin{proof}
Take $\chi=-r^d$ and $s=r^\alpha$ in Proposition~\ref{p: infinite family}
\end{proof}
\medskip

We can use Proposition \ref{p: infinite family} and Corollary \ref{c: infinite family} to build smooth covers
of examples of $(2,m,n)^*$-groups $G$ and hence provide an infinite number of examples in many of the cases addressed in parts (A), (B), and (C) of Theorem \ref{t: main}.
\medskip

\begin{prop}\label{p:psl-sum}
For each of the cases listed in Table~{\rm \ref{t: a}}, and those in rows {\rm B1} and {\rm B2} of Table~{\rm \ref{t: b}}, plus those in rows {\rm B3} and {\rm B7} of Table~{\rm \ref{t: b}} with $\ell=1$, and those in rows {\rm C1}, {\rm C3}, {\rm C5} and {\rm C6} of Table~{\rm \ref{t: c}}, there exist infinitely many finite $(2,m,n)^*$-groups $G$ with Euler characteristic a power of $r$, such that $G/O(G)$ is isomorphic  to $\PSL_2(q)$ in part {\rm (A)}, $\PGL_2(q)$ in part {\rm (B)}, and $D_j\times D_k$ and $(C_2\times C_2)\rtimes D_\ell$ in part {\rm (C)}, for suitable $q$ and $j,k,\ell$.
\end{prop}
\medskip

\begin{proof}
Except for the rows A2 and B2 of of Tables \ref{t: a} and \ref{t: b}, examples have been known to exist for some time, and were given in \cite{condersmall2} and/or \cite{C600}. For the remaining cases A2 and B2, examples have been found by the first author with the help of {\sc Magma} \cite{magma}.

Specifically, we have the following examples of $(2,m,n)^*$-groups $G$ for the relevant rows of Tables~\ref{t: a}, \ref{t: b} and \ref{t: c}, with corresponding parameters and $\chi=\chi_G$ as indicated.
\begin{enumerate}
 \item[A1:] \ a $(2,5,5)^*$-group $G\cong \PSL_2(5)$ with $\chi=-3$;
 \item[A2:] \ a $(2,3,15)^*$-group $G$ with $G/O \cong \PSL_2(5)$, where $|O|=3^6$ and $\chi=-3^7$;
 \item[A3:] \ a $(2,3,13)^*$-group $G\cong \PSL_2(13)$ with $\chi = -7^2$;
 \item[A4:] \ a $(2,3,7)^*$-group $G\cong \PSL_2(13)$ with $\chi=-13$;
 \item[B1:] \ a $(2,4,6)^*$-group $G\cong \PGL_2(5)$ with $\chi=-5$;
 \item[B2:] \ a $(2,20,30)^*$-group $G$ with $G/O \cong \PGL_2(5)$, where $|O|=5^3$ and $\chi_G=-5^5$;
 \item[B3:] \ a $(2,3,8)^*$-group $G\cong \PGL_2(7)$ for $\ell=1$,  with $\chi=-7$;
 \item[B6:] \ a $(2,4,5)^*$-group $G\cong \PGL_2(5)$ for $\ell=1$, with $\chi=-3$;
 \item[C1:] \ a $(2,6,4)^*$-group $G\cong (C_3 \times C_3) \rtimes D_4 $ with $\chi= -3$;
 \item[C3:] \ a $(2,2j,2k)^*$-group $G\cong D_j{\times}D_k$ with $\chi= -r= j{+}k{-}jk$ for coprime odd $j$ and $k$;
 \item[C5:] \ a $(2,4,\ell)^*$-group $G\cong (C_2\times C_2)\rtimes D_\ell$ with $\chi = r = 4-\ell$,
 for odd $\ell$ such that $3 \mid \ell$ but $9 \nmid \ell$;
 \item[C6:] \ a $(2,3,12)^*$-group $G$ with $G/E \cong S_4$, where $|E|=3^3$ and $\chi=-3^3$.
\end{enumerate}

For each of these $(2,m,n)^*$-groups $G$, the existence of infinitely many finite smooth covers of $G$ with the required properties is ensured by Corollary \ref{c: infinite family}. (In cases C3 and C5, there are also infinitely many examples with negative prime characteristic available from \cite{bns}.)

Note that the list above includes a $(2,4,5)^*$-group isomorphic to $\PGL_2(5)$ as a member of the family B6; the non-trivial smooth covers of this group guaranteed by Corollary~\ref{c: infinite family} provide the infinite family of groups of type B7.
\end{proof}
\medskip

Note that the small `starting' groups $G$ listed in the above proof  are not unique. We give some alternative examples below:
\begin{itemize}
 \item Case B1: There is also a $(2,4,6)^*$-group $G$ with $G/O \cong \PGL_2(5)$, where $O$ is elementary abelian of order $5^3$ and $\chi_G=-5^4$.
 \item Case B2: Examples of $(2,20,30)^*$-groups can be given in multiple ways, each with $\chi_G=-5^5$.
 \item Case B3: There are  multiple examples of $(2,3,8)^*$-groups $G$ with $G/O \cong \PGL_2(7)$ and with $O$ elementary abelian of order $7^3$.
 \item Case C3: Here the condition that ties $r$ with $j$ and $k$ can be written in the form $r+1=(j-1)(k-1)$, and the factorisation of $r+1$ offers possibilities for coprime odd $j$ and $k$ that yield more examples. One example is map N9.3, of type $\{6,10\}$ in the census \cite{C600}.
 \item Case C5: Another example is the map N13.1, of type $\{4,15\}$ in \cite{C600}).

\end{itemize}

With the exception of family B6, it is perfectly possible to use the method of Proposition~\ref{p:psl-sum} to find an infinite number of examples for the other families in Tables~\ref{t: a}, \ref{t: b} and \ref{t: c}. The main obstacle is to find nice `starting' groups in each case.
The family B6 is different because there we have specified that the normal subgroup $N$ is trivial. As in the proof above, non-trivial smooth covers of groups in this family constructed using Corollary~\ref{c: infinite family} will fall into family B7.

Less significant, but worth noting, is the fact that the smoothness of the covers constructed in this way means that all of the groups constructed from a given starting group  share the same values of the type parameters $m$ and $n$. In contrast, we give a different construction below which will show that we can find an infinite number of examples where $m$ and $n$ vary (for instance, in families B3 to B7, we would like to show that there are examples for different values of the parameter $\ell\,$).

Thus, to deal with the families omitted in Proposition~\ref{p:psl-sum} and to show that we can allow for $m$ and $n$ to vary, we now prove a more general version of Theorem \ref{p: infinite family} for extensions beyond smooth covers.
\medskip

\begin{prop}\label{p: jozef}
Let $H=\langle\, \alpha,\beta,\gamma \ |\  \alpha^2,\,\beta^2,\,\gamma^2,\,(\alpha\gamma)^2,\, (\alpha\beta)^m,\,
(\beta\gamma)^n,.\,\ldots \,\rangle$ be a finite $(2,m,n)^*$-group with odd Euler characteristic $\chi$, and let
$g=2-\chi$. Also let $u$ be the sum of the numbers of vertices and faces of the corresponding non-orientable regular map $M=(H;\alpha,\beta,\gamma)$. Then for every odd prime $r$ there exists a $(2,rm,rn)^*$-group $H^*$ and a normal $r$-subgroup $N\cong (C_r)^{i}$, where $i=g-1+u = 1+|H|/4$, such that $H^*$ is an extension of $N$ by $H$.

Furthermore,  suppose that $n$ is even, and $H$ has a subgroup $H_0$ of index $2$ such that either $\alpha,\beta \in H_0$ while $\gamma\in H{\setminus}H_0$, or $\alpha,\beta \in H{\setminus}H_0$ while $\gamma\in H_0$. Then for each odd $\ell\ge 1$ such that ${\rm gcd}(\ell,r|H|) = 1$, there exists a $(2,r\ell m,rn)^*$-group $G$ containing a normal subgroup isomorphic to $N\times C_\ell$, where $N$ is an $r$-group, such that $G$ is an extension of $N\times C_\ell$ by $H,$ and $G/N\cong (H_0\times C_\ell){\cdot}2$.
\end{prop}
\smallskip

\begin{proof}
Let $\Delta= \Delta(2,rm,rn)=\langle\, A,B,C\ |\ A^2=B^2=C^2=(AC)^2=(AB)^{rm}=(BC)^{rn}=1 \,\rangle$ be the full $(2,rm,rn)$-triangle group, and let $K$ be the kernel of the obvious epimorphism $\Delta\to H$ that takes $A,B,C$
to $\alpha,\beta,\gamma$, respectively. Since $\Delta$ is a non-Euclidean crystallographic group (or NEC group, for short) with compact quotient space, and $K$ has finite index in $\Delta$, it follows that $K$ is also an NEC group. The epimorphism $\Delta \to H$ induces a branched covering of the non-orientable regular map $M=(H;\alpha,\beta,\gamma)$ on a non-orientable surface with Euler characteristic $\chi=2-g$, via a tessellation of a hyperbolic plane of type $\{rm,rn\}$. This branched covering has a total of $u$ branch points, each of order $r$,  with one at each vertex and one at the centre of each face of $M$. Then from \cite{MacB,Sing} it follows that $K$ is an NEC group with signature $(g,-,[r^{(u)}],\{-\})$, admitting the presentation
\begin{equation}\label{eq:PresK}
K=\langle\, a_1,\ldots,a_{g},\,x_1,\ldots,x_u\ |\ x_1^r,\,\ldots\,,\, x_u^r,\, a_1^{\, 2}\ldots a_{g}^{\, 2} \, x_1\ldots x_u \,\rangle.
\end{equation}

Now $r$ is odd and so the abelianisation $K/K'$ is isomorphic to $\mathbb{Z}^{g-1}\times (C_r)^u$, generated by the images of $a_1,\ldots,a_{g-1}$ and $x_1,\ldots,x_u$ (with the image of $a_{g}$ being uniquely determined). Then a composition of the corresponding epimorphism $\theta:\ K\to \mathbb{Z}^{g-1}\times (C_r)^u$ with the natural projection $\vartheta:\ \mathbb{Z}^{g-1}\times (C_r)^u \to (C_r)^{g-1+u}$ (which reduces the first $g-1$ coordinates mod $r$ and keeps the remaining coordinates unchanged) gives an epimorphism $\eta:\ K\to (C_r)^{g-1+u}$. Letting $L$ be the kernel of $\eta$, and  $N=K/L$, we find that $N\cong (C_r)^{g-1+u}$, where $N$ is an elementary abelian $r$-group of rank $g-1+u$, with the set $\{\bar a_1,\ldots,\bar a_{g-1},\bar x_1,\ldots,\bar x_u\}$ of $\eta$-images of the generators other than $a_{g}$ in (\ref{eq:PresK}) being a minimal generating set for $N$.

The kernel $L$ can also be described as the product $L_r=K'K^{(r)}$ of the commutator subgroup $K'\lhd K$ and the normal subgroup $K^{(r)}$ of $K$ generated by the $r$-th powers of elements of $K$. Indeed $L_r\le L$,
because the $\eta$-image of every commutator and every $r$-th power of an element of $K$ is ttrivial. To prove
that $L = L_r$, observe that from $L_r\le L$ we have $N=K/L\cong (K/L_r)/(L/L_r)$, and as obviously $K/L_r \le
(C_r)^{g-1+u}=N$ it follows that $L/L_r=1$ and hence $L=K'K^{(r)}$. Both $K'$ and $K^{(r)}$ are characteristic
in $K$, so also $L$ is characteristic in $K$ and hence normal in $\Delta$.

We claim that $H^*=\Delta/L$ is a $(2,rm,rn)^*$-group.

Clearly $H^*$ is generated by the three involutions $\alpha^*=AL$, $\beta^*=BL$ and $\gamma^*=CL$,
the last two of which commute. We will  show that the order of $\alpha^*\beta^*$ is $rm$, by assuming the contrary, namely that the order of $\alpha^*\beta^*=ABL$ in $H^*=\Delta/L$ is a proper divisor of $rm$.
Then since the order of $ABK$ in $H\cong \Delta/K$ is $m$, and $r$ is prime, the order of $ABL$ in $H^*$ must be $m$, so that $(AB)^m\in L$.  But $(AB)^m$ has order $r$ in $\Delta$, and lies in $K$, and by an observation made in \cite{MacB}, every element of order $r$ in the NEC group $K$ is conjugate to some generator of order $r$ that appears in the presentation (\ref{eq:PresK}), so $(AB)^m$ must be conjugate to some $x_i$ (with $1\le i\le u$).  Then since $(AB)^m\in L \lhd \Delta$, it follows that $x_i\in L$ and it follows that the set $\{\bar a_1,\ldots,\bar a_{g-1},\bar x_1, \ldots,\bar x_u\}$ cannot generate $N\cong (C_r)^{g-1+u}$, a contradiction. Hence the order of $\alpha^*\beta^*$ is $rm$.  An analogous argument shows that the order of $\beta^*\gamma^*$ in $H^*$ is $rn$.

Also because $H\cong \Delta/K\cong (\Delta/L)/(K/L) \cong H^*/N$, we see that $H^*$ is isomorphic to an extension of the $r$-group $N\cong (C_r)^{g-1+u}$ by $H$, and so $|H^*| = r^{g-1+u}|H|$. If $S^*$ and $S$ are the surfaces carrying the regular maps corresponding to $H^*$ and $H$, respectively, then the natural projection $H^*\to H^*/N\cong H$ induces a branched covering $S^*\to S$, which is a smooth $r^{g-1+u}$-fold covering outside the branch points, and non-orientability of $S$ and the fact that $r$ is odd imply that $S^*$ is non-orientable as well.
Thus $H^*$ is a $(2,rm,rn)^*$-group , as claimed.

Next, the number of edges of the map determined by $H$ is $\chi+u$, and therefore $\chi+u= |H|/4$, so the dimension of the elementary abelian $r$-group $N$ is $g-1+u=1+\chi + u = 1+|H|/4$.
This completes the proof of the first part of the proposition.

\smallskip
To prove the second part, recall that the projection $\Delta \to \Delta/L=H^*$ takes the
generators $A,B,C$ of $\Delta$ to the generators $\alpha^*,\beta^*,\gamma^*$ of $H^*$. Note also that the natural projection $H^*\to H=H^*/N$, or equivalently, $\Delta/L\to \Delta/K \cong (\Delta/L)/(K/L)$, takes $(\alpha^*,\beta^*,\gamma^*)$ to the generating triple $(\alpha,\beta, \gamma)$ of $H$, and so we may identify $\alpha$, $\beta$ and $\gamma$ with $N\alpha^*$, $N\beta^*$ and $N\gamma^*$.
Under our extra assumption that $H$ has a (normal) subgroup $H_0$ of index $2$ such that either $\alpha,\beta \in H_0$ while $\gamma\in H{\setminus}H_0$, or $\alpha,\beta \in H{\setminus}H_0$ while $\gamma\in H_0$,
we see that the group $H^*_0=N.H_0$ is a normal subgroup of $H^*=N.H$, and either \,(i) $\,\alpha^*,\beta^*\in H^*_0$ while $\gamma^*\in H^*{\setminus}H^*_0,\,$ or \,(ii) $\,\alpha^*,\beta^*\in H^*{\setminus}H^*_0$ while $\gamma^*\in H^*_0$.

Now consider the semi-direct product $G=C_\ell\rtimes H^*$, where $\ell$ is a positive integer such that $\gcd(\ell,r) = \gcd(\ell,|H|)=1$, and where $H^*$ acts on $C_\ell$ in such a way that conjugation by those involutions in $\{\alpha^*,\beta^*, \gamma^*\}$ contained in $H^*_0$ centralise $C_\ell$ while conjugation by those contained in $H^*{\setminus}H^*_0$ invert every element of $C_\ell$.
%
Also let $z$ be a generator for $C_\ell$, and let $a$, $b$ and $c$ be elements of $G=C_\ell\rtimes H^*$ given by $(a,b,c)=(\alpha^*,\beta^*,z\gamma^*)$ in case (i), and $(a,b,c)=(z\alpha^*,z\beta^*,\gamma^*)$ in case (ii).

Clearly, $a$, $b$ and $c$ are involutions, and $a$ commutes with $c$, in both cases. Also in both cases $bc=z\beta^*\gamma^*$, and hence $bc$ has order $rn$, because $\beta^*\gamma^*$ inverts every element of $C_\ell$, and we have assumed that $n$ is even. On the other hand, in both cases $ab=\alpha^*\beta^*$ and has order $\ell rm$, because $\alpha^*\beta^*$ centralises $C_\ell$, and $\ell$ is relatively prime to $rm = \ord(\alpha^*\beta^*)$.
Note also that if $\delta^*=(\beta^*\gamma^*)^2$, which lies in $H^*_0$ and has order $rn/2$, then conjugation by $\delta^*$ centralises $C_\ell$. This implies that $(z\delta^*)^{rn/2}=z^{rn/2}\delta^{rn/2}=z^{rn/2}$, and then because $\ell$ is coprime to both $r$ and $|H|$ and hence also to $n$, it follows that $z^{rn/2}$ generates $C_\ell$, which in turn implies that $G=\langle a,b,c\rangle$.

Topologically, the natural projection $G=C_\ell\rtimes H^*\to H^*$ induces outside the branch points an $\ell$-sheeted covering of the non-orientable surface $S^*$ by the surface $\tilde S$ carrying the regular map with automorphism group isomorphic to $G$, and the assumption that $\ell$ is odd together with non-orientability of $S^*$ imply that also $\tilde S$ is non-orientable.

All of the above shows that $G$ is a $(2,\ell rm,rn)^*$-group, and that there exists an epimorphism
$G=C_\ell\rtimes H^* \to C_\ell\rtimes H$ given by $(a,b,c) \mapsto (\alpha, \beta, z\gamma)$ in case (i), and by $(a,b.c) \mapsto (z\alpha,z\beta,\gamma)$ in case (ii), with kernel $N$ in both cases.  Hence in particular, $N$ is a normal subgroup of $G$. But also $C_\ell$ is a normal subgroup of $G$, and then since $\gcd(|N|,\ell)=1$, we conclude that $C_\ell N\cong C_\ell\times N$ is normal in $G$ and hence
$G$ is an extension of $N\times C_\ell$ by $H$. Finally, the way in which $C_\ell\rtimes H^*$ was defined implies that $G/N\cong (H_0\times C_\ell).2$, which completes the proof.
\end{proof}

We now apply Proposition~\ref{p: jozef} to a situation where $H=\PGL(2,q)$ and $H_0=\PSL(2,q)$, with $q=p^e$ for an odd prime $p$, and $q\ge 5$. Indeed by Propositions 3.2, 4.6 and 6.1 of \cite{cps2}, we may extract the fact that for every hyperbolic type $\{m,n\}$ such that $n\in \{q-1,q+1\}$ and either $m=p$ or $m$ divides $(q\pm 1)/2$, the group $H=\PGL(2,q)$ can be given the structure of a $(2,m,n)^*$-group satisfying the assumptions of
Proposition~\ref{p: jozef}, and the following corollary builds on that.

\begin{cor}\label{c: jozef}
Let $q\ge 5$ be a power of an odd prime $p$, and let $r$ be any odd prime.
Also suppose $H\cong \PGL_2(q)$ is a $(2,m,n)^*$-group inducing a non-orientable regular map $M$ of hyperbolic type $\{m,n\}$ with $\Aut(M)\cong H$, where $n = q\pm 1$, and either $m=p$, or $m$ is odd and $2m$ divides $q\pm 1$. Then for every $\sigma\in \{0,1\}$ and every positive integer $\ell$ coprime to $r|H|$, there exists a $(2,\ell r^\sigma m,r^\sigma n)^*$-group $G$ containing a normal subgroup isomorphic to $N\times C_\ell$ for some $r$-group $N$ (trivial if $\sigma =0$), such that $G$ is an extension of $N\times C_\ell$ by $\PGL(2,q)$, and $G/N\cong (\PSL(2,q)\times C_\ell).2$.
\end{cor}

Note that the two values of $\sigma$ in the statement of Corollary \ref{c: jozef} correspond to the relationship between the Fitting subgroup and the generalised Fitting subgroup of the resulting group $G$. In particular, $F^*(G)\ne F(G)$ when $\sigma=0$, but $F^*(G)=F(G)$ when $\sigma = 1$.

\begin{proof}
First, Propositions 3.2 and 4.6 of \cite{cps2} give an explicit representation of  $H\cong \PGL_2(q)$ as a quotient of the full triangle group $\Delta = \Delta(2,m,n)$, in which the three involutory generators $X_2,Y_2,Z_2$ (up to a $\pm$ sign) take the form of $2\times 2$ matrices  contained in the group $\PSL_2(q^2)$ and depend on primitive $2m$-th and $2n$-th roots of unity in $\mathbb{F}_{q^2}$.
One may check that these generators satisfy condition (C) in Proposition 4.6 of \cite{cps2}, which implies
that $X_2,\,Y_2,\,Z_2$ generate a subgroup of $\PSL_2(q^2)$ isomorphic to $\PGL_2(q)$,
and moreover, the product $Z_2X_2$, which has order $n=q\pm 1$, lies outside the unique copy of $\PSL_2(q)$ in $\PGL_2(q)$, while the product $Y_2Z_2$, which has odd order $m$, lies inside it. The fact that $\langle
X_2,\,Y_2,\,Z_2 \rangle$ is a $(2,m,n)^*$-group then follows from part (1) of \cite[Proposition 6.1]{cps}.

For $\sigma=1$, the required conclusion now follows by taking $(\alpha,\beta,\gamma) = (Y_2,Z_2,X_2)$ in Proposition \ref{p: jozef}, with $H=\langle \alpha, \beta, \gamma\rangle \cong \PGL_2(q)$ and $H_0\cong \PSL_2(q)$, as the conditions on containment of $\alpha$, $\beta$ and $\gamma$ in $H_0$ and $H{\setminus} H_0$ just reflect the two possibilities implied by the condition on containment of $Y_2Z_2$ (but not $Z_2X_2$) in the unique copy of $\PSL_2(q)$ in $\PGL_2(q)$.

For $\sigma=0$, the second part of the proof of Proposition \ref{p: jozef} works without building a covering with branch points of order $r$ at vertices and faces in the first part.
\end{proof}
\medskip

Unfortunately Proposition \ref{p: jozef} and Corollary \ref{c: jozef} do not guarantee that the Euler characteristic of the covering group is a power of a prime, in contrast to Corollary \ref{c: infinite family} where this was automatic if one started with a suitable group.

More specifically, a $(2,m,n)^*$-group arising in Corollary \ref{c: jozef} and occurring in one of rows B3 to B7 of Table \ref{t: b} will have Euler characteristic $-r^d$ if and only if the corresponding integer entry appearing in the column of Table \ref{t: b} marked $\chi_G$ is equal to $r^d$. In all five rows this requires a number-theoretic condition which may be far from obvious to penetrate; we will comment on each case in the application of Corollary \ref{c: jozef} below.

Note also that if Proposition \ref{p: jozef} and Corollary \ref{c: jozef} are applied to a particular group $H$ with $H/O(H)\cong \PSL_2(q)$ with characteristic $-\chi = g-2 = r^d$,  and the sum of the numbers of vertices and faces of the associated map is equal to $u$, then the proof of Proposition \ref{p: jozef} allows no other possibility than $r^{g-1+u}$ for the order of the normal $r$-subgroup $N$. The only potential flexibility (if any) lies in the choice of $\ell$ to satisfy the relevant number-theoretic condition.

Ideally, to apply Corollary \ref{c: jozef} one would like to start with $H=\PGL_2(q)$ being a $(2,m,n)^*$-group, with Euler characteristic $\chi = -r^d$, for suitable $m,n,q$ and $r$. But even then, the choice is severely limited, as we now demonstrate for possible types $\{m,n\}$ implied by cases (a) to (e) of Lemma \ref{l: cases} with $\ell=1$ and $|N|=1$.

\begin{lem}\label{lem:PGLcases}
Let $H = \PGL_2(q)$ be a $(2,m,n)^*$-group with $q = p^e \ge 5$, for which $\{m,n\}$ is one of $\{(q\pm 1)/2, q\mp 1\}$, $\{q\pm 1,q\mp 1\}$, or $\{p,p\pm 1\}$ with $p\ge 5$. Then the Euler characteristic of $H$ is equal to $-r^d$ for some prime $r\ge 3$ if and only if $d=1$ and one of the following cases occurs:
\begin{enumerate}
\item[($\alpha$)] $q=r=7$ and $\{m,n\}=\{3,8\}$;
\item[($\beta$)]  $q=r=5$ and $\{m,n\}=\{4,6\}$;
\item[($\gamma$)] $q=p=5$, $r=3$ and $\{m,n\}=\{4,5\}$.
\end{enumerate}
\end{lem}

\begin{proof}
First, suppose $\{m,n\}=\{(q\pm 1)/2,q\mp 1\}$.  Then by Euler's formula \eqref{e: sunny} we find that $r^d = -\chi_H = q(q^2-1)(1/4 - 1/(q\pm 1) - 1/(2(q\mp 1)) =  q(q^2-6q+1)/4$ or $q(q^2-6q-3)/4$, and it follows that $q$ is a power of $r$, so $r = p$, but then $p$ cannot divide $q^2-6q+1$, and also cannot divide $q^2-6q-3$ unless $p =3$, in which case $(q^2-6q-3)/4$ cannot be a proper power of $3$, so $q^2-6q-3 = 4$.  This implies $q = 7$ and therefore $r = p = q = 7$ and $\{m,n\}=\{q-1)/2,q+1\} = \{3,8\}$.

For the second case, where $\{m,n\}=\{q\pm 1,q\mp 1\}$, Euler's formula gives $r^d= q(q^2-4q-1)/4$, so again $r = p$, but $(q^2-4q-1)/4$ cannot be a proper power of $3$, so $q^2-4q-1 = 4$ and therefore $q = 5$,
which implies that $r = p = q = 5$ and $\{m,n\}=\{5\pm 1,5\mp 1\}= \{4,6\}$.

Finally, consider the third case, where $\{m,n\}=\{p,p\pm 1\}$.  In this case Euler's formula gives
$r^d = q(q^2-1)(p^2 - 3p - 2)/(4p(p+1)$ or $q(q^2-1)(p^2 -5p + 2)/(4p(p-1)$.
Now if $e > 1$, then $r^d$ is the product of the integers $q/p$, $(q^2-1)/(2(p \pm 1))$
and either $(p^2 - 3p - 2)/2$ or $(p^2 - 5p + 2)/2$, all three of which must be a power of $r$,
so again $r = p$, but then it follows that $p^2 - 3p - 2 = 4$ or $p^2 - 5p + 2 = 4$,
both of which are impossible, so $e = 1$, and $q = p$.
It follows that  $r^d$ is the product of $(p \mp 1)/2$ and either $(p^2 - 3p - 2)/2$ or $(p^2 - 5p + 2)/2$,
and so $p = 2r^i \pm 1$ for some $i$, and correspondingly, either $(p^2 - 3p - 2)/2$ or $(p^2 - 5p + 2)/2$ is a power of $r$. In the former case, $(p^2 - 3p - 2)/2 = (4r^{2i} + 4r^i +1 - 6r^i - 3 - 2)/2 = 2r^{2i} - r^i - 2$,
which is not divisible by $r$ and hence must be $1$, but then $0 = 2(r^{i})^2 - r^i - 3 = (2r^i -3)(r^i+1)$,
which is impossible.
In the other case,  $(p^2 - 5p \pm 2)/2 = (4r^{2i} - 4r^i +1 - 10r^i + 5 + 2)/2 = 2r^{2i} - 7r^i + 4$, and this must be~$1$, so $0 = 2(r^{i})^2 - 7r^i + 3 = (2r^i -1)(r^i-3)$, which implies that $r^i = 3$ and hence that $i = 1$ and $r = 3$ and $p = 2r^i-1 = 5$, with $\{m,n\}=\{p,p-1\}= \{4,5\}$.
\end{proof}

\medskip

Case ($\beta$) of Lemma \ref{lem:PGLcases} was addressed in Proposition \ref{p:psl-sum}, and features in
row B1 of Table~\ref{t: b}. We will continue by presenting infinite sets of examples for each of the five cases appearing in rows B3 to B7 of Table \ref{t: b}. Recall that infinite families of examples for rows B3 and B7 in the case $\ell=1$ were given in Proposition \ref{p:psl-sum}.
\medskip

\noindent{\bf Row B3 of Table~\ref{t: b}.} Up to duality, and as in case ($\alpha$) of Lemma \ref{lem:PGLcases} for $q=r=7$, one may start with two inequivalent ways of representing $H=\PGL_2(7)$ as a $(2,3,8)^*$-group, giving rise to a pair of non-isomorphic non-orientable regular maps of type $\{3,8\}$ on a surface with $-\chi= 7$.

Taking $\sigma = 0$ in Corollary \ref{c: jozef} then supplies an infinite number of further examples of $(2,3\ell,8)^*$-groups with negative Euler characteristic a power of $r = 7$ (with $N=1$), with $\ell$ chosen so that $\gcd(\ell,2{\cdot}3{\cdot}7) = \gcd(\ell,42) = 1$, namely as follows. The equation for the characteristic in row B3 with $N=1$ simplifies to $7^d = -\chi = 7(9\ell-8)$, from which we find that $\ell = (7^{d-1}+8)/9$. If the latter rational expression for $\ell$ is an integer, then $\ell$ is odd and coprime to $7$, and also it is easy to see that it is relatively prime to $3$ if and only if $d\equiv 1$ or $7$ mod $9$. Hence for every such positive integer $d$ (and the implied values of $\ell$, including $1$ as covered in Proposition \ref{p:psl-sum}), we find a pair of examples of $(2,3\ell,8)^*$-groups with Euler characteristic $-7^d$ for $s=0$, with $G\cong (\PSL_2(7)\times C_\ell).2$.

Similarly, taking  $\sigma=1$ in Corollary \ref{c: jozef} supplies examples from noting that $|H|/4=84$ when $H=\PGL_2(7)$, so that in Proposition \ref{p: jozef} we may take $N\cong (C_7)^{85}$. By the Euler formula for the $(2,21\ell,56)^*$-group from row B3 of Table~\ref{t: b} with $s = 1$, we obtain $7^d = 7^{85}(84\ell-3\ell-8)= 7^{85}(81\ell-8)$, which requires $\ell$ to be a positive integer coprime to $42$ such that $81\ell-8=7^j$ for some $j$. Hence we need $\ell = (7^j+8)/81$ to be an integer coprime to $3$, and an easy exercise shows that $j$ must be $12$ or $39$ mod $81$ (with $\ell \equiv 2$ or $1$ mod $3$, respectively). Each of these two possibilities gives an infinite family of $(2,21\ell,56)^*$-groups $G$ with $-\chi_G = 3^d$ for $d=85+j$, with $G\cong N{\cdot}(\PSL_2(7)\times C_\ell).2$. 
\medskip

\noindent{\bf Row B4 of Table~\ref{t: b}.} Here Lemma \ref{lem:PGLcases} shows that it is not possible to start with $\PGL_2(9)$ in its representation as a $(2,5,8)^*$-group (with $|N|=1$ and $\ell=1$), and in fact Corollary \ref{c: jozef} cannot be applied with $|N|=1$ for any $\ell>1$.
Perhaps surprisingly, however, examples can be built by taking $\sigma=1$, $p=3$ and $r=3$ in Corollary
\ref{c: jozef}, and choosing an appropriate $3$-group for $N$. A regular map of type $\{5,8\}$ coming from $H\cong \PGL_2(9)$ as a $(2,5,8)^*$-group has $720/4= 180$ edges, and by Proposition \ref{p: jozef} one can take $N\cong (C_3)^{181}$. Substituting these values into the expression for $-\chi_G$ in the right-most column of row B4 of Table ~\ref{t: b} with $s = 1$ gives $3^d = 3^{182}(55\ell -8)$, which implies that $55\ell-8 = 3^j$ for some $j$, with $\ell = (3^j+8)/55$ coprime to $5$. This time an easy exercise shows that $j$ must be $11$ mod $20$ (with $\ell \equiv 11$ mod $30$). Every such $j$ gives a $(2,15\ell,24)^*$-group $G$ with $-\chi_G = 3^d$ for $d=182+j$, with $G\cong N{\cdot}(\PSL_2(9)\times C_\ell).2$. 
There are also some smaller examples than those in this infinite family, and one that results from a {\sc Magma}
computation is a $(2,15{\cdot}3221,24)^*$-group $G$ with a normal elementary abelian $3$-subgroup $N$ of order $3^6$, such that $G/N\cong (\PSL_2(9)\times C_{3221}).2$, with $-\chi_G =3^{18}$.
\medskip

\noindent{\bf Row B5 of Table~\ref{t: b}.} Here Lemma \ref{lem:PGLcases} indicates that it is not possible to start with $\PGL_2(7)$ as a $(2,7,8)^*$-group for $p=7$ and $r=3$ when $\sigma=0$ (and $|N|=1$ and $\ell=1$), and again in this case Corollary \ref{c: jozef} does not help when $|N|=1$ for any $\ell>1$.
But one can take $\sigma=1$ instead, using the regular map of type $\{7,8\}$ coming from $H\cong \PGL_2(7)$ as a as a $(2,7,8)^*$-group, which has $336/4 = 84$ edges, and then take $r=(p-1)/2= 3$ and $N$ as the $3$-group from Proposition \ref{p: jozef}, namely with $N\cong (C_3)^{85}$.
Then the expression for $-\chi_G$ in the right-most column of row B5 of Table ~\ref{t: b} with $s = 1$ gives
$r^d = 3^{85}(77\ell -8)$, and by Corollary \ref{c: jozef} one obtains a $(2,21\ell,24)^*$-group with Euler characteristic $3^d$ whenever $\ell = (3^j+8)/77$ is relatively prime to $7$.
Another exercise shows that $j\equiv 21$ mod $30$ but $j\not\equiv 141$ mod $210$, and then these possible values of $j$ give an infinite family of $(2,21\ell,24)^*$-groups $G$ with $-\chi_G = 3^d$ for $d=85+j$, with $G\cong N{\cdot}(\PSL_2(7)\times C_\ell).2$. 
%
\medskip

\noindent{\bf Row B6 of Table~\ref{t: b}.}  We start using case ($\gamma$) of Lemma \ref{lem:PGLcases} for $N=1$, taking the $(2,5,4)^*$-group $\PGL_2(5)$ with Euler characteristic $-3$, for which Corollary \ref{c: jozef} with $\sigma=0$ gives us an infinite supply of examples. Here the condition that $\chi_G = r^d$ from \eqref{eq:case(e)} with $p=5$, $r=3$, $|N|=1$ and $m_1=1$ simplifies to $3^d= 3(5\ell-4)$, so that  $\ell = (3^j+4)/5$ with $j=d-1$ and $\gcd(\ell,30)=1$.  This happens if and only if $j \equiv 4$ mod $5$ but $j \not\equiv 16$ mod $20$, and gives an infinite family of $(2,5\ell,4)^*$-groups $G$ with $-\chi_G = 3^d$ for $d=j+1$ for such $j$, and $\ell$ as stated, with $G\cong (\PSL_2(5)\times C_\ell).2$.

\medskip

\noindent{\bf Row B7 of Table~\ref{t: b}.} Constructing examples with $|N| > 1$ and $\sigma=1$ from the $(2,5,4)^*$-group $H=\PGL_2(5)$ requires determining a suitable $3$-group $N$, and from Proposition \ref{p: jozef} with $1+|H|/4=31$, we can take $N\cong (C_3)^{31}$.
Then the expression for $-\chi_G$ in the right-most column of row B6 of Table ~\ref{t: b} with $r = 3$ and $s = 1$ gives $3^d = 3^{31}(25\ell-4)$, so we need $25\ell-4=3^j$, with $\ell = (3^j+4)/25$ being coprime to $5$. This happens if and only if $j \equiv 16$ mod $20$ but $j \not\equiv 36$ mod $100$, and gives an infinite family of $(2,15\ell,12)^*$-groups $G$ with $-\chi_G = 3^d$ for $d=31+j$ for such $j$ and $\ell$, with $G\cong N{\cdot}(\PSL_2(5)\times C_\ell).2$.

\medskip

Having now produced an infinite family for each of the cases in Tables~\ref{t: a} and \ref{t: b}, and earlier for the cases C1, C3, C5 and C6 from Table~\ref{t: c} in Proposition \ref{p:psl-sum}, we deal with C2, C4 and C7.
\medskip

\noindent{\bf Row C2 of Table~\ref{t: c}.}  Let $H=H_1(\ell)$ be the $(2,2,\ell)^*$-group isomorphic to the dihedral group $D_\ell$ of order $2\ell$ coming from part(a) of Theorem \ref{t: ass} for $\ell = 3^i+1$ (where $i\ge 1$). The underlying graph of the corresponding regular map is a cycle of length $\ell/2$ embedded in the real projective plane, with Euler characteristic $\chi=1$, and having a single face.  Hence for the parameters appearing in the first part Proposition \ref{p: jozef}, we find that $u=1+\ell/2$ and $k=1$. Letting $r=3$ as implied by Lemma \ref{l: soluble dihedral}, we find that the $3$-group $N$ from Proposition \ref{p: jozef} is isomorphic to $(C_3)^{k-1+u} = (C_3)^{1+\ell/2}$. This gives a regular map of type $\{6,3\ell\}$ coming from a $(2,6,3\ell)^*$-group isomorphic to $N{\cdot}D_\ell$, with Euler characteristic $3^{i-1}|O| = 3^{i+\ell/2}$, for every $i\ge 1$.
\medskip

\noindent{\bf Row C4 of Table~\ref{t: c}.} Here we produce an infinite number of examples with parameters as in row C4 with $r=3$ and $\alpha=\beta=1$. The entry in the column headed  $-\chi_G$ ($=r^d$) gives $rjk-j-k = r^i$, or, equivalently, $(rj-1)(rk-1)=r^{i+1}+1$, for some $i$, which suggests considering factorisations of $r^{i+1}+1$ into two factors, each congruent to $-1$ mod $r$.
For $r=3$ this is manageable for $j=5$, when the condition requires $14k-5 = 3^i$ and so $14(3k-1)=3^{i+1}+1$.  Taking $i \equiv 2$ mod~$6$ makes
$3^{i+1}+1 = (3^3+1)(3^{i-2}-3^{i-5}+3^{i-8}-\dots+3^6-3^3+1) = 28f$ with $2f \equiv 2 \equiv -1$ mod~$3$ and hence $3^{i+1}+1 = 14 \cdot 2f = (3j-1)(3k-1)$ where $j = 5$ and $k = (2f+1)/3$, which is odd. Then also $3^i = (28f-1)/3 = (14(2f+1)-15)/3 =14k-5$, implying that $5\nmid k$, and hence that $j$ and $k$ are relatively prime.
Now part (b) of Theorem \ref{t: ass} implies that there is a $(2,10,2k)^*$-group $H=H_2(5,k)$ of order $4{\cdot}5{\cdot}k = 20k$, and application the first part of Proposition \ref{p: jozef} to this group (with $1+ |H|/4 = 1+5k$) for the prime $r=3$ gives an infinite family of $(2,30,6k)^*$-groups of the form $N{\cdot}H$ with $N\cong (C_3)^{1+5k}$, one for each $k$ as above.

Similar examples can be generated using other ways of writing $3^{i+1}+1 = (3j-1)(3k-1)$, such as
$(j,k) = (13,173)$ or $(25,89)$ when $i=8$. A more striking example was found from a {\sc Magma} computation for $\alpha=\beta=3$ and $(j,k)=(5,49)$, with $rjk-j-k = 3^{3}jk-j-k=3^8$, as follows.
The regular map of type $\{10,98\}$ given by part (b) of Theorem \ref{t: ass} and marked N193.3 in the list at \cite{C600} has  automorphism group $H_2(5,49)\cong D_5\times D_{49}$, and this can be considered as a non-smooth quotient $\Delta/K$ of the full $(2,270,2646)$-triangle group $\Delta=\langle A,B,C\rangle$, with $AB$ of order $3^3{\cdot}10=270$ and $BC$ of order $3^3{\cdot}98 = 2646$. The kernel $K$ is generated by $246$ elements, and contains $(AB)^{10}$ and $(BC)^{98}$, each which has order 27 in $\Delta$, and its abelianisation is a rank 246 abelian group isomorphic to $(\mathbb{Z}_{27})^{53} \oplus \mathbb{Z}_{54} \oplus \mathbb{Z}^{192}$. Hence a variant of the Macbeath trick
can be used to obtain a smooth quotient $G$ of the full $(2,270,2646)$-triangle group, of order $980{\cdot}27^{246}$, and this gives a non-orientable map of the required type, with characteristic $-3^{743}$.
\medskip

\noindent{\bf Row C7 of Table~\ref{t: c}.} Here we apply the first part of Proposition \ref{p: jozef} to a map from row C7 with $\alpha=\beta=1$, using a suitable $\ell\equiv 3$ mod $6$ and a prime $r\equiv 5$ mod $6$ such that $2\ell r-4-\ell = r^j$ for some~$j$.  In particular, taking $r = 5$, we have $5^j + 4 = 9\ell \equiv 27$ mod $54$ and hence $j \equiv 13$ mod $18$.
Now for every integer $i\ge 0$ we can take $j = 18i + 13$ and $\ell_i = (5^j+4)/9 \equiv 3$ mod $6$, and then by part (c) of Theorem \ref{t: ass}, the group $H = H_3(\ell_i)$ of order $8\ell_i$ admits a presentation as a $(2,4,\ell_i)^*$-group, and since $1 + |H|/4 = 1 + 2\ell_i$, the first part of Proposition \ref{p: jozef} applied with $r=5$ gives a $(2,20, 5\ell_i)^*$-group $G$ of the form $N{\cdot}H$, where $N\cong (C_5)^{1+2\ell_i}$, with characteristic $-5^{13+2\ell}$.
\bigskip

We can now prove Corollary \ref{c: main}, which gives a classification of all $(2,m,n)^*$-groups $G$
(equivalently, of all non-orientable regular maps) with characteristic $\chi_G=-r^d$, where $r$ is an odd prime and $d \in \{1,2,3,4\}$.
\bigskip

{\bf Proof of Corollary \ref{c: main}}. We divide the proof into the same three cases (A), (B) and (C) that were used in Theorem \ref{t: main}, namely depending on whether $G/O(G)$ is isomorphic to $\PSL_2(q)$ or $\PGL_2(q)$ for some odd prime-power $q\ge 5$, or $G/O(G)$ is soluble.
\medskip

{\sl Case {\rm (A)}: $G/O(G)\cong \PSL_2(q)$.} Part (A) of Theorem \ref{t: main} implies that the pair $(q,r)$ must be  $(5,3)$, $(13,7)$ or $(13,13)$. A search using {\sc Magma} reveals that there are only four such groups $G$ with Euler characteristic $\chi_G \ge -r^4$ for $r\in \{3,7,13\}$, namely, a $(2,5,5)^*$-group $G\cong \PSL_2(5)$ with $\chi_G=-3$, a $(2,3,7)^*$-group $G\cong \PSL_2(13)$ with $\chi_G=-13$, a $(2,3,13)^*$-group $G\cong \PSL_2(13)$ with $\chi_G=-7^2$, and a $(2,3,7)^*$-group $G\cong N{\cdot}\PSL_2(13)$ with $\chi_G=-13^4$ in which $N$ is an elementary abelian group of order $13^3$. These give the first four rows of the table in Corollary \ref{c: main}.
\medskip

{\sl Case {\rm (B)}: $G/O(G)\cong \PGL_2(q)$.}  Part (B) of Theorem \ref{t: main} implies that the pair $(q,r)$ is $(5,5)$, $(7,7)$ or $(9,3)$ when $d\le 4$. A search using {\sc Magma} shows that the first pair gives a $(2,4,6)^*$-group $G\cong \PGL_2(5)$ with $\chi_G= -5$, and a $(2,4,6)^*$-group $G\cong N{\cdot}\PGL_2(5)$ with $\chi_G= -5^4$ for which $N$ an elementary abelian group of order $5^3$, both coming from case B1.  Similarly the second pair gives a $(2,3,8)^*$-group $G\cong \PGL_2(7)$ with $\chi_G= -7$, and a $(2,3,8)^*$-group $G\cong N{\cdot}\PGL_2(7)$ with $\chi_G= -7^4$ in which $N$ is an elementary abelian group of order $7^3$, both coming from case B3.  The third pair $(q,r)=(9,3)$ gives no suitable group in the case B4 with $d\le 4$ at all. Also there are no suitable groups with $d\le 4$ in the case B5 for $p>7$ by Lemma \ref{l: case d}, and for the remaining possibility where $p=7$ and $r=3$ there is also no group with $d\le 4$. By Lemma \ref{l: case e} there is only one possibility for the case B6 with $d\le 4$, namely, a $(2,4,5)^*$-group $G\cong \PGL_2(5)$ with $\chi_G=-3$. Finally, for case B7  there are no suitable groups with $d\le 4$ for $p>5$ by Lemma \ref{l: case e}, and the remaining case where $p=5$ and $r=3$ also gives no group. This implies validity of the five rows of the table associated with Corollary \ref{c: main} for part (B). 
\medskip

{\sl Case {\rm (C)}: $G$ soluble, $N$ a normal $r$-subgroup of $G$ with $G/N$ almost Sylow-cyclic.} The data in
rows C1 and C2 of Table~\ref{t: c} come from Lemma \ref{l: soluble dihedral} (with $r=3$), and by the list at \cite{C600}
there are exactly $13$ groups $G$ satisfying the conditions C1 or C2 with $\chi_G=-3^d$ for some $d\le 4$, namely three $(2,4,6)^*$-groups $G$ with $-\chi_G$ equal to $3$, $3^2$ and $3^4$, plus four $(2,6,6)^*$-groups $G$ with $-\chi_G$ equal to $3$, $3^2$, $3^3$ and $3^4$, plus four $(2,6,12)^*$-groups $G$, three with  $-\chi_G = 3^3$ and one with $-\chi_G=3^4$, and two $(2,6,30)^*$-groups $G$ with $-\chi_G$ equal to $3^3$ and $3^4$. These are marked either C1 or C2 (or both) in the left-most column of the table accompanying Corollary \ref{c: main}, with those of type $\{6,6\}$ and $\{6,30\}$ satisfying also the requirements for the case C4.

Groups $G$ with parameters described in rows C3 and C4 of Table~\ref{t: c} were dealt with in Lemma
\ref{l: soluble G2lk}. In the case of C3, if the largest normal $r$-subgroup $N$ of $G$ is trivial, then a $(2,2j,2k)^*$-group $G\cong D_j\times D_k$ exists if and only if $j,k$ are relatively prime odd positive integers such that $(j-1)(k-1)=r^d+1$, for odd $d\ge 1$, with $r\equiv 3$ mod $4$.
In the case of C4, also in Lemma \ref{l: soluble G2lk} it was shown that when $\alpha\ge 1$ we have $d\ge 5$ unless $r=3$.  Again, checking the list in \cite{C600} for maps with $r=3$ and $d\le 4$ gives four $(2,6,6)^*$-groups with characteristic $-3^d$ for $d = 1$ to $4$, and two $(2,6,30)^*$-groups with characteristic $-3^d$ for $d\in \{3,4\}$, overlapping with C1 or C2 (as noted above). It follows that the only $(2,m,n)^*$-groups $G$ with $\chi_G=-r^d$ for $d\le 4$ satisfying the conditions in rows C3 and C4 of Table~\ref{t: c} are those described in part (a) of Corollary \ref{c: main} plus the six groups of type $\{6,6\}$ and $\{6,30\}$ marked $C2,4$ in the Table associated with Corollary \ref{c: main}.

It remains for us to consider $(2,m,n)^*$-groups $G$ associated with rows C5 to C7 of Table~\ref{t: c}. These were handled in Lemma \ref{l: soluble G3k}. For row C5 the $(2,4,\ell)^*$-groups $G$ with $-\chi_G= r^d$ for $d\le 4$ are exactly those for which $d$ is odd, $r\equiv 5$ mod $6$, and $\ell$ is given by $\ell=4+r^d \equiv 3$ mod $4$, as described in part (b) of Corollary \ref{c: main}. Indeed for every such pair $(r,\ell)$ there exists a $(2,4,\ell)^*$-group $G\cong (C_2\times C_2)\rtimes D_\ell$, with order $8\ell$.
Also by Lemma \ref{l: soluble G3k} and its proof, there is only one exceptional group $G$ satisfying the required conditions for which $d\le 4$, namely a unique $(2,3,12)^*$-group $G$ satisfying C6, which is an extension of a $3$-group $N$ of order $27$ by $(C_2\times C_2)\rtimes D_3$ with $-\chi_G=3^3$.

Hence in the statement of Corollary \ref{c: main}, we have covered cases (a) and (b), as well as case (c) listing $22$ groups (and $23$ maps) in the accompanying Table, and so completing the proof. \hfill $\Box$
\bigskip


\section{Remarks on what else might be possible}\label{s:rem}


Corollary~\ref{c: main} is really only a `sample' theorem, with the bound $d\leq 4$ chosen for convenience. In
principle this bound can be altered as much as we want, subject only to the constraint imposed by computing power, and the complications that arise in the soluble case as $d$ increases. One of the implications of our main Theorem \ref{t: main} is that for every odd $d\ge 1$ there exist infinitely many primes $r$ for which there exists a non-orientable regular map of Euler characteristic $-r^d$. From this point of view it would be of interest to extend our classification to $d\leq 10$, if only to see if finiteness encountered for $d=2$ and $d=4$ extends to larger even values of $d$.

Our basic theoretic approach can also be generalised in several ways. For example, rather than considering
$(2,m,n)^*$-groups with $\chi_G$ odd, one could consider automorphism groups $G$ of general regular maps (whether orientable or not) with $\chi_{G}\equiv 2$ mod $4$. In this direction, the first thing one would need is an analogue of Proposition~\ref{p: odd order}. Such an analogue can be extracted from available knowledge in group theory, but does not appear explicitly in the literature. Then the approach taken in the rest of this paper could be followed, for groups $G$ with $\chi_G = -2\hskip 1pt r^d$. One would encounter a number of interesting extra cases in this setting, including some with $G/O(G)\cong \PSL_3(3)$ or $M_{11}$.

In another direction, one could retain the focus on $(2,m,n)^*$-groups, but weaken the restriction on the Euler
characteristic to the case where $\chi_G=-r^as^b$ for odd primes $r$ and $s$. This, however, would be likely to
introduce significant technical problems, for instance in analysing the possibilities for $\overline{m}$ and
$\overline{n}$ when $G/O(G)=\PSL_2(q)$. In principle, our approach should remain valid in this case, and it is tempting to conjecture that something like our main theorem should hold in
considerably more generality.
\bigskip

\section*{Acknowledgments}

The first author acknowledges the use of the {\sc Magma} system \cite{magma} to find and check examples
and perform other computational experiments and searches in the course of the research for this paper,
and also acknowledges support from New Zealand's Marsden Fund (grant no.\ UOA2320). The third author acknowledges support from APVV Research Grants 22-0005 and 23-0076, and from VEGA Research Grants 1/0567/22 and 1/0069/23.

\bibliographystyle{amsalpha}


\begin{thebibliography}{xxxxxx}

\bibitem[AS60]{AhlSar} L.H. Ahlfors and L.~Sario, \emph{Riemann surfaces}, Princeton Univ. Press, 1960.

\bibitem[Asc00]{aschfgt}
M.~Aschbacher, \emph{Finite group theory}, 2nd ed., Cambridge Studies in
  Advanced Mathematics, vol.~10, Cambridge University Press, Cambridge, 2000.

\bibitem[BJ05]{BeJo} M. Belolipetsky and G.A.~Jones, \emph{Automorphism
groups of Riemann surfaces of genus $p+1$, where $p$ is a prime},
Glasgow Math. J. {\bf 47} (2005), 379--393.

\bibitem[BCM06]{bid} J.N.S.~Bidwell, M.J.~Curran and D.J.~McCaughan, \emph{Automorphisms of direct products of finite groups}, Arch. Math. {\bf 86} (2006), 481--489.


\bibitem[BCP97]{magma}
W.~Bosma, J.~Cannon, and C.~Playoust, \emph{The {M}agma algebra system. {I}.
  {T}he user language}, J. Symbolic Comput. \textbf{24} (1997), no.~3-4,
  235--265, Computational algebra and number theory (London, 1993).

\bibitem[BNS05]{bns}
A.~Breda~d'Azevedo, R.~Nedela, and J.~{\v{S}}ir{\'a}{\v{n}},
  \emph{Classification of regular maps of negative prime {E}uler
  characteristic}, Trans. Amer. Math. Soc. \textbf{357} (2005), no.~10, 4175--4190.

\bibitem[BS85]{BrSi} R.P.~Bryant and D.~Singerman, \emph{Foundations of the
theory of maps on surfaces with boundary}, Quart. J. Math. Oxford
Ser. (2) {\bf 36} (1985) no. 141, 17--41.

\bibitem[Con09]{condersmall2}
M.~Conder, \emph{Regular maps and hypermaps of {E}uler characteristic {$-1$} to
  {$-200$}}, J. Combin. Theory Ser. B \textbf{99} (2009), no.~2, 455--459.

\bibitem[Con13]{C600}
M.~Conder, \emph{Regular maps of Euler characteristic $-1$ to $-600$},
lists available on webpages linked from \url{http://www.math.auckland.ac.nz/~conder}.


\bibitem[CJST25]{CJST} M.D.E.~Conder, G.A.~Jones, J.~\v{S}ir\'a\v{n} and T.W.~Tucker,
\emph{Regular maps}, monograph in preparation, to appear in Cambridge Univ. Press.

\bibitem[CNS12]{cns} M.~Conder, R.~Nedela and J. ~\v{S}ir\'a\v{n}, \emph{Classification of regular maps
of Euler characteristic $-3p$}, J. Combin. Theory Ser. B {\bf 102} (2012), 967--981.

\bibitem[CP{\v{S}}08]{cps2}
M.~Conder, P.~Poto{\v{c}}nik, and J.~{\v{S}}ir{\'a}{\v{n}}, \emph{Regular
  hypermaps over projective linear groups}, J. Aust. Math. Soc. \textbf{85}
  (2008), no.~2, 155--175.

\bibitem[CP{\v{S}}10]{cps}
\bysame, \emph{Regular maps with almost {S}ylow-cyclic automorphism groups, and
  classification of regular maps with {E}uler characteristic {$-p^2$}}, J.
  Algebra \textbf{324} (2010), no.~10, 2620--2635.

\bibitem[CST10]{CST} M.~Conder, J.~\v{S}ir\'a\v{n} and T.~Tucker, \emph{The genera, reflexibility and
simplicity of regular maps}, J. Europ. Math. Soc. {\bf 12} (2010), 343--364.

\bibitem[CS20]{cosi} M.~Conder and J.~\v{S}ir\'a\v{n}, \emph{Classification of regular maps of prime characteristic revisited: Avoiding the Gorenstein-Walter theorem}, Journal of Algebra 548 (2020), 120--133.

\bibitem[Dic58]{dickson}
L.E.~Dickson, \emph{Linear groups: {W}ith an exposition of the
  {G}alois field theory}, with an introduction by W. Magnus, Dover Publications
  Inc., New York, 1958.

\bibitem[FT78]{FT} W.~Feit and J.~Tits, \emph{Projective representations of minimum degree of group extensions},
Canad. J. Math. {\bf 30} (1978), 1092--1102.

\bibitem[Gor80]{gor}
D.~Gorenstein, \emph{Finite groups}, second ed., Chelsea Publishing Co.,
  New York, 1980.

\bibitem[GW65a]{gor2}
D.~Gorenstein and J.H.~Walter, \emph{The characterization of finite groups
  with dihedral {S}ylow {$2$}-subgroups. {II}}, J. Algebra \textbf{2} (1965),
  218--270.

\bibitem[GW65b]{gor3}
\bysame, \emph{The characterization of finite groups with dihedral {S}ylow
  {$2$}-subgroups. {III}}, J. Algebra \textbf{2} (1965), 354--393.

\bibitem[JS78]{JoSi} G.A.~Jones and D.~Singerman,
\emph{Theory of maps on orientable surfaces}, Proc. London Math. Soc. (3)
{\bf 37} (1978), 273--307.

\bibitem[KL89]{kl1}
P.~Kleidman and M.~Liebeck, \emph{On a theorem of Feit and Tits}, Proc. Amer. Math. Soc.
{\bf 107} (1989) 2, 315--322.

\bibitem[KL90]{kl}
P.~Kleidman and M.~Liebeck, \emph{The subgroup structure of the finite
  classical groups}, London Mathematical Society Lecture Note Series, vol.~129,
  Cambridge University Press, Cambridge, 1990.

\bibitem[LS05]{LiS} C.H.~Li and J.~\v{S}ir\'a\v{n},  \emph{Regular maps whose groups do not act faithfully on vertices, edges, or faces}, Europ. J. Comb. 26 (2005), 521--541.

\bibitem[Mac67]{MacB}
A.M.~Macbeath, \emph{The classification of non-euclidean plane crystallographic groups},
Canad. J. Math. \textbf{19} (1967), 1192--1205.

\bibitem[Mas77]{Massey}
W.S.~Massey, \emph{Algebraic topology: an introduction},
Grad. Texts in Math., vol.~56, Springer-Verlag, New York-Heidelberg, 1977.



\bibitem[Sin74]{Sing}
D.~Singerman, \emph{On the structure of non-{E}uclidean crystallographic
  groups}, Proc. Cambridge Philos. Soc. \textbf{76} (1974), 233--240.

\bibitem[Sir13]{siran}
J.~\v{S}ir\'a\v{n}, \emph{How symmetric can maps on surfaces be?},
  In: Surveys in combinatorics, Cambridge University Press, 2013, pp.~161--238.

\bibitem[S1904]{sch} I.~Schur, \emph{\"Uber die Darstellung der endlichen Gruppen durch gebrochene lineare Substitutionen}, J. Reine Angew. Math. 127 (1904), 20--50.

\end{thebibliography}

\end{document}